\documentclass[a4paper]{amsart}
\usepackage[utf8]{inputenc}

\usepackage{bbm}
\usepackage{amsmath}
\usepackage{amsfonts}
\usepackage{amssymb}
\usepackage{amsthm}
\usepackage{amsaddr}
\usepackage{enumerate}
\usepackage{color}
\usepackage{mathrsfs}
\usepackage{stmaryrd}
\usepackage{hyperref}
\usepackage{graphicx}
\usepackage[a4paper,margin=1in,headsep=25pt,footskip=35pt]{geometry}
\usepackage{fancyhdr}

\SetSymbolFont{stmry}{bold}{U}{stmry}{m}{n}

\newtheorem{thm}{Theorem}
\newtheorem{cor}[thm]{Corollary}
\newtheorem{lem}[thm]{Lemma}
\newtheorem{prop}[thm]{Proposition}

\theoremstyle{definition}
\newtheorem{defn}[thm]{Definition}
\newtheorem*{defn*}{Definition}

\theoremstyle{remark}
\newtheorem{rem}[thm]{Remark}

\newcommand{\deq}{\mathrel{\mathop:}=}
\newcommand{\e}[1]{\mathrm{e}^{#1}}
\newcommand{\R} {\mathbb{R}}
\newcommand{\C} {\mathbb{C}}
\newcommand{\D} {\mathbb{D}}
\newcommand{\K} {\mathbb{K}}
\newcommand{\N} {\mathbb{N}}
\newcommand{\Z} {\mathbb{Z}}

\newcommand{\E} {\mathbb{E}}

\newcommand{\adj}{^{*}} 
\newcommand{\tp}{^{\intercal}}

\DeclareMathOperator{\diag}{diag}

\DeclareMathOperator{\Tr}{Tr}

\DeclareMathOperator{\supp}{supp}
\DeclareMathOperator{\spann}{span}

\DeclareMathOperator{\re}{\mathrm{Re}}
\DeclareMathOperator{\im}{\mathrm{Im}}

\newcommand{\caA}{{\mathcal A}}
\newcommand{\caB}{{\mathcal B}}
\newcommand{\caC}{{\mathcal C}}
\newcommand{\caD}{{\mathcal D}}

\newcommand{\caF}{{\mathcal F}}

\newcommand{\caI}{{\mathcal I}}
\newcommand{\caJ}{{\mathcal J}}

\newcommand{\caL}{{\mathcal L}}

\newcommand{\caN}{{\mathcal N}}
\newcommand{\caO}{{\mathcal O}}

\newcommand{\caS}{{\mathcal S}}

\newcommand{\bbK}{{\mathbb K}}

\newcommand{\bbR}{{\mathbb R}}

\newcommand{\fra}{{\mathfrak a}}
\newcommand{\frb}{{\mathfrak b}}

\newcommand{\frm}{{\mathfrak m}}

\newcommand{\frC}{{\mathfrak C}}

\newcommand{\frK}{{\mathfrak K}}

\newcommand{\frX}{{\mathfrak X}}

\newcommand{\bsa}{{\boldsymbol a}}
\newcommand{\bsb}{{\boldsymbol b}}

\newcommand{\bse}{{\boldsymbol e}}

\newcommand{\bsl}{{\boldsymbol l}}

\newcommand{\bsr}{{\boldsymbol r}}
\newcommand{\bss}{{\boldsymbol s}}

\newcommand{\bsu}{{\boldsymbol u}}
\newcommand{\bsv}{{\boldsymbol v}}
\newcommand{\bsw}{{\boldsymbol w}}
\newcommand{\bsx}{{\boldsymbol x}}
\newcommand{\bsy}{{\boldsymbol y}}
\newcommand{\bsz}{{\boldsymbol z}}

\newcommand{\wt}{\widetilde}
\newcommand{\ol}{\overline}

\newcommand{\beq}{ \begin{equation} }
	\newcommand{\eeq}{ \end{equation} }
\newcommand{\beqs}{\begin{equation*}}
	\newcommand{\eeqs}{\end{equation*}}

\newcommand{\lone}{\mathbbm{1}} 

\newcommand{\dd}{\mathrm{d}}
\newcommand{\ii}{\mathrm{i}}

\renewcommand{\P}{\mathbb{P}}

\newcommand{\bsfrm}{\boldsymbol{\frm}}

\newcommand{\AND}{\quad\text{and}\quad}

\newcommand{\llbra}{\llbracket}
\newcommand{\rrbra}{\rrbracket}

\newcommand\norm[1]{\Vert#1\Vert}
\newcommand\Norm[1]{\left\Vert#1\right\Vert}

\newcommand\prob[1]{\mathbf{P}\left[#1\right]}

\newcommand\Absv[1]{\left\vert#1\right\vert}
\newcommand\absv[1]{\vert#1\vert}

\newcommand\brkt[1]{\langle#1\rangle}
\newcommand\Brkt[1]{\left\langle#1\right\rangle}
\newcommand\bbrktt[1]{\llbra #1\rrbra}

\numberwithin{equation}{section} 
\numberwithin{thm}{section}

\newcommand{\Cb}{\color{black}}

\newcommand{\nc}{\normalcolor}

\title[Condition number of random matrices]{Wegner estimate and upper bound on the eigenvalue condition number of 
non-Hermitian random matrices}

\author{L\'{a}szl\'{o} Erd\H{o}s$^{\dagger}$}
\author{Hong Chang Ji$^{\ddagger}$}
\address{Institute of Science and Technology Austria \\ Am Campus 1, 3400 Klosterneuburg, Austria}
\email{lerdos@ist.ac.at}
\email{hongchangji@ista.ac.at}
\thanks{$^\dagger$ Partially supported by ERC Advanced Grant "RMTBeyond" No. 101020331}
\thanks{$\ddagger$ Supported by ERC Advanced Grant "RMTBeyond" No. 101020331}

\date{\today}
\begin{document}
	
	\pagestyle{fancy}
	\fancyhead{}
	\fancyhead[C]{Condition number of random matrices}

\maketitle

\begin{abstract}
	We consider $N\times N$ non-Hermitian random matrices of the form $X+A$, where $A$ is a general deterministic matrix
	 and $\sqrt{N}X$ consists of independent entries with zero mean, unit variance, and bounded densities. 
	 For this ensemble, we prove (i) a Wegner estimate, i.e. 
	  that the local density of eigenvalues is bounded  by $N^{1+o(1)}$ 
	 and (ii) that the expected condition number of any bulk eigenvalue is bounded by $N^{1+o(1)}$; both
	   results are optimal up to the factor $N^{o(1)}$.  The latter result complements
	   the very recent matching lower bound obtained in~\cite{Cipolloni-Erdos-Henheik-Schroder2023arXiv} and improves
	   the  $N$-dependence of the upper bounds 
	   in~\cite{Banks-Kulkarni-Mukherjee-Srivastava2021,Banks-Vargas-Kulkarni-Srivastava2020,Jain-Sah-Sawhney2021}.
	   Our main ingredient, a near-optimal lower tail estimate for the small singular values of $X+A-z$, is of independent interest.
\end{abstract}

\section{Introduction}

 \subsection{Setup}
 We consider large $N\times N$ random matrices with independent entries and without any symmetry constraint, so that they are typically not normal and have complex eigenvalues. More precisely, our random matrix is of the form $X+A$, where $A$ is a general deterministic `data' matrix and $X$ is a random `noise matrix' consisting of independent complex or real entries with continuous distributions; 
the ensemble is often called \emph{``Information-plus-Noise''} model, see e.g. \cite{Bordenave-Capitaine2016,Dozier-Silverstein2007_1,Dozier-Silverstein_2,ElKaroui2010}.
The main objects of this paper are the three fundamental spectral quantities of $X+A$ determining its stability properties, namely, eigenvalues, diagonal \emph{eigenvector overlaps} (also known as \emph{eigenvalue condition number}), and small singular values. We will typically use the scaling where both $\|A \|$ and $\|X\|$ are  bounded, independently of $N$, however our results on the singular values will be uniform in $A$, hence they also cover the especially interesting small noise regime after a trivial rescaling. 

We next explain our main motivations to study these quantities, along with their  definitions. The fundamental difficulty in studying a non-normal matrix $A$, as opposed to a Hermitian one, is the instability of its spectrum; even a tiny perturbation to $A$ may result in a very large change in its eigenvalues. To make this precise, we assume that $A$ has a simple spectrum, and write its spectral decomposition
\beq
A=\sum_{i=1}^{N}\sigma_{i}\bsr_{i}\bsl_{i}\tp,\qquad 
A\bsr_{i}=\sigma_{i}\bsr_{i},\qquad \bsl_{i}\tp A=\sigma_{i}\bsl_{i}\tp,
\eeq  
where $\sigma_i=\sigma_{i}(A)$'s are the eigenvalues of $A$ and $\{\bsl_i=\bsl_{i}(A), \bsr_i=\bsr_{i}(A):1\leq i\leq N\}$ is a bi-orthogonal family (so that $\bsl_{i}\tp\bsr_{j}=\delta_{ij}$) in $\C^{N}$ consisting of left and right eigenvectors of $A$. Under this normalisation, the \emph{diagonal eigenvector overlaps} are defined as the scale-invariant quantity 
\beq
\caO_{ii}(A)\deq \norm{\bsl_{i}(A)}^{2}\norm{\bsr_{i}(A)}^{2}.
\eeq

With this definition, we have the well-known variational identity  (see e.g.~\cite{Banks-Vargas-Kulkarni-Srivastava2020})
\beq\label{eq:ovlp_var}
\sqrt{\caO_{ii}(A)}=\lim_{t\to0}\;\sup\left\{ \Big| \frac{\sigma_{i}(A+tE)-\sigma_{i}(A)}{t}\Big|:E\in\C^{N\times N},\norm{E}=1\right\}.
\eeq
This formula shows that the diagonal overlap  quantifies the instability of $\sigma_{i}$ against the worst perturbation, therefore $\sqrt{\caO_{ii}}$ is also called the \emph{eigenvalue condition number}. However, it is practically impossible to control $\caO_{ii}(A)$ for a general deterministic $A$: One can easily construct a bi-orthogonal family with arbitrarily large $\norm{\bsl_{1}}^{2}\norm{\bsr_{1}}^{2}$, and thus construct a matrix $A$ with arbitrarily large overlap $\caO_{11}(A)$ but bounded $\sigma_{1}(A)$. 

Besides quantifying the instability $\sigma_{i}$, another interesting feature of the overlap $\caO_{ii}(A)$  is that it connects $\sigma_{i}$ to the shifted singular values of $A$. This follows 
 from another standard variational identity;
\beq\label{eq:ovlp=s/l}
\sqrt{\caO_{ii}(A)}=\lim_{z\to\sigma_{i}}\frac{\absv{\sigma_{i}(A)-z}}{\lambda_{1}(A-z)},
\eeq
where $\lambda_{1}(A-z)\leq\cdots\leq \lambda_{N}(A-z)$ stand for the singular values of $A-z$ for $z\in\C$. Note that singular values, as (square roots of) the eigenvalues of the \emph{Hermitian} matrix $(A-z)(A-z)^*$, are much more stable under perturbations. Thus~\eqref{eq:ovlp=s/l} shows that the overlap exactly quantifies the magnification of the instability of $\sigma_{i}$ in terms of $\lambda_{1}(A-z)$ for $z$ near $\sigma_{i}$.

While the overlap is an uncontrolled object in the worst case, numerical algorithms, whose computational cost or accuracy should deteriorate with $\caO_{ii}$, still perform well in practice; see, e.g.  \cite{Sankar-Spielman-Teng2006} and references therein concerning Gaussian elimination. This is often explained by a mechanism called \emph{smoothed analysis}  capitalizing on the fact that even for the worst $A$, the instability of its spectrum is typically regularized by a noise, i.e. a random perturbation, for example in \cite{Banks-Vargas-Kulkarni-Srivastava2019,Banks-Kulkarni-Mukherjee-Srivastava2021,Banks-Vargas-Kulkarni-Srivastava2020,Jain-Sah-Sawhney2021,Sankar-Spielman-Teng2006}. This motivates the investigation of three quantities $\sigma_{i}(A+X)$, $\caO_{ii}(A+X)$, and $\lambda_{k}(X+A)$ for small $k$, where $A$ is a general deterministic matrix and $X$ is a random matrix with independent entries.

Interestingly, overlaps also appear in a completely unrelated context, namely in the Dyson-type stochastic eigenvalue dynamics. Consider the matrix Brownian motion $A_{t}=A+B_{t}$ where the entries of $B_{t}\in\C^{N\times N}$ are independent standard complex Brownian motions and $A$ has a simple spectrum. The eigenvalues $\sigma_i(A_t)$ of $A_{t}$ are martingales with brackets given by
\beq\label{eq:DBM}
	\dd\brkt{\sigma_{i}(A_{t}),\ol{\sigma_{j}(A_{t})}}_{t}=\bsl_{j}(A_{t})\adj\bsl_{i}(A_{t})\bsr_{j}(A_{t})\adj\bsr_{i}(A_{t})\dd t,\qquad \brkt{\sigma_{i}(A_{t}),\sigma_{j}(A_{t})}_{t}=0;
\eeq
see \cite[Appendix A]{Bourgade-Dubach2020}. Taking $i=j$, the diagonal overlap $\caO_{ii}(A_{t})$ is exactly the time derivative of the quadratic variation of $\sigma_{i}(A_{t})$. In particular, \eqref{eq:DBM} shows that the dynamics of $\sigma_{i}(A_{t})$ involves eigenvectors, hence is not autonomous in contrast to the standard Hermitian Dyson Brownian motion \cite{Dyson1962}.
 
 Finally, we point out yet another context of our results. The noise matrix is rescaled by a factor $1/\sqrt{N}$ to keep its norm of order one as $N$ grows, i.e. the entries of $\sqrt{N}X$ remain on constant scale. With this scaling, the celebrated \emph{circular law} \cite{Bai1997,Girko1984,Tao-Vu2010} states that if the entries of $\sqrt{N}X$ are  centered i.i.d. random variables with unit variance, then $\sigma_{i}(X)$'s are asymptotically uniformly distributed on the unit disk with density of order $N$ (we consider unnormalized densities throughout the paper unless stated otherwise). Extension of the circular law  for $X+A$ was proved in \cite[Corollary 1.17]{Tao-Vu2010} as well if $A$ has a limiting $*$-distribution $a$, in which case the limit is the 
Brown measure of free sum $x+a$ where $x$ is a free circular element. Recently this Brown measure was studied in detail in \cite{Zhong2021}, where it was proved that there is an open set $\caB\subset\C$ so that this measure is supported on $\ol{\caB}$ and has a strictly positive real analytic density in $\caB$. Consequently, the number of eigenvalues $\sigma_{i}(X+A)$ in a complex domain $\caD\subset\caB$ is typically comparable to $N\absv{\caD}$ where $\absv{\caD}$ is the area of $\caD$; see Remark \ref{rem:DOS} for details. We refer to \cite[Section 1.1]{Zhong2021} for a historical exposition on the eigenvalue density of $X+A$.
    
 In the special case of the \emph{Ginibre ensemble}, i.e. when the entries of $X$ are i.i.d. Gaussian and $A=0$, the density, or one-point function, of the eigenvalues can be computed explicitly~\cite{Ginibre1965} and it is essentially given by $\rho_N(z)= \frac{N}{\pi}$ for $|z|<1$. For general entry distribution of $X$, the \emph{local circular law} \cite{Bourgade-Yau-Yin2014} (see also~\cite{Alt-Erdos-Kruger2018,Alt-Kruger2021}), asserts that $\sum_i f(\sigma_i)$ is well approximated by $\frac{N}{\pi}\int f$ as long as $f$ lives on a scale much larger than $N^{-1/2}$, the typical eigenvalue spacing. Thus the density of eigenvalues is well understood on \emph{mesoscopic} scales, but it is a highly nontrivial statement that the density remains absolutely continuous and even constant  on arbitrary small scales (in particular this obviously requires that the distribution of the matrix elements has a continuous component). Motivated by the theory of random Schr\"odinger  operators, in the Hermitian setup an upper bound on the eigenvalue density is called \emph{Wegner estimate}~\cite{Wegner1981}. Its optimal form has been established in~\cite{Erdos-Schlein-Yau2010,Maltsev-Schlein2011} for a large class of Hermitian random matrices (Wigner matrices). Prior to the current work, no Wegner-type result below the mesoscopic scale $N^{-1/2+\epsilon}$ has been known for non-Hermitian i.i.d. random matrices besides the Ginibre case and its elliptic generalizations where explicit formulas using planar orthogonal polynomials and contour integral methods are available, see the recent comprehensive paper~\cite{Akemann-Duits-Molag2022} and references therein. We finally remark that optimal upper bounds on the local $k$-point correlation functions (in particular Wegner estimates for $k=1$) for two-dimensional Coulomb gases with general potential were established in \cite{Thoma2022arXiv} using non-explicit conditioning methods.

 Eigenvalues and eigenvector overlaps of $X+A$ are genuinely \emph{non-Hermitian} objects which are hard to access 
 directly. Instead, one  typically works with the singular values and singular vectors of $X+A$, which are accessible  via 
 \emph{Hermitian} methods, however, the relation between them is subtle. On the level of the spectrum, Girko's 
 formula (see \eqref{eq:Girko} below) provides an explicit connection, but it requires to control the lower tail of the 
 small singular values of  $X+A-z$ for all $z\in \C$ simultaneously. For eigenvectors the situation is much more delicate; 
 there is no direct formula relating eigenvectors of $X+A$ and singular vectors of $X+A-z$ apart from the special case 
 when $z$ is exactly an eigenvalue, a conditioning that is technically very hard to handle. 
 Our main results on eigenvector overlaps and Wegner estimate for eigenvalues both 
 heavily rely on very accurate lower tail estimates on the lowest singular values of $X+A-z$ 
 and we develop a new method to estimate them. We remark that singular values and singular vectors of the Information-plus-Noise model $X+A$ have also been extensively studied in statistics, but with a strong focus on the large singular values. Our current main interest is the opposite regime, so we refrain from reviewing the extensive statistics literature on the subject.


\subsection{Our results} 
The standing assumption on $X$ for all our results is that the entries of $\sqrt{N}X$ have continuous distributions with bounded densities. This guarantees the main mechanism behind the key Wegner estimates. For certain results, we  may additionally assume that  the entries of $\sqrt{N}X$ are centered, have variance one and  have all moments finite. The matrix $A$ can be completely general, for some results we only assume that $A$ has bounded norm. Before formally stating  our main theorems with precise conditions in Section~\ref{sec:results} we informally summarize their content and compare them with previous results.

The first main result, given precisely in Theorem~\ref{thm:main} later, concerns an upper bound for the averaged density of the eigenvalues $\sigma_{i}$'s. Specifically, we prove that 
\beq\label{eq:thm_1_simple}
	\E \#\{ i: \sigma_{i}(X+A)\in \caD \} \lesssim \begin{cases}
			N^{o(1)}(N\absv{\caD})^{1-o(1)} & \text{ for complex }X,\\
			N^{o(1)}\dfrac{(N\absv{\caD})^{1-o(1)}}{\min_{z\in\caD}\lambda_{1}(\Im[A-z])} &\text{ for real }X,
			\end{cases}
\eeq
for any ball $\caD$ \emph{of arbitrarily small size} in the bulk spectrum (see \eqref{eq:Btau_def} for its definition), where $\Im[A-z]$ is the matrix with entrywise imaginary part of $A-z$. This result is a weak \emph{Wegner estimate}: Up to the $o(1)$-powers it almost shows the absolute continuity and boundedness of the normalized eigenvalue density. Note that in the real case the estimate necessarily deteriorates near the real axis;  this is because real random matrices tend to have many real eigenvalues (e.g. the real Ginibre matrices have approximately $\sqrt{N}$ real eigenvalues \cite{Edelman-Kostlan-Schub1994}) and once eigenvalues concentrate on a one-dimensional submanifold, their \emph{two-dimensional} density naturally has a singular component.

The second main results (Theorem \ref{thm:overlap}--\ref{thm:overlap_strong} later) concern an upper bound on the diagonal overlap $\caO_{ii}$ in expectation sense. Since in the Ginibre case it is known  $\caO_{ii}$ has a fat tail~\cite{Bourgade-Dubach2020}, hence controlling 
second or higher moments is impossible.  In Theorem \ref{thm:overlap} we prove that
\beq\label{eq:thm_2_simple}
	\E\left[\lone_{\Xi_{\caD}}\sum_{i:\sigma_{i}(X+A)\in\caD}\caO_{ii}(X+A)\right]\lesssim\begin{cases}
		\absv{\log N}^{2}\cdot N(N\absv{\caD}) &\text{ for complex }X,\\
		\absv{\log N}^{2}\cdot N\dfrac{N\absv{\caD}}{\min_{z\in\caD}\lambda_{1}(\Im[A-z])} &\text{ for real }X,
	\end{cases} 
\eeq
where $\caD\subset\C$ can be any polynomially small (in $N$) ball and $\Xi_{\caD}$ is a very high-probability event. 
\Cb  In Theorem \ref{thm:overlap_strong},  under somewhat stronger assumptions, we remove the exceptional set, i.e. the factor of $\Xi_{\caD}$ from \eqref{eq:thm_2_simple}. Theorem \ref{thm:overlap_strong} also covers overlaps for \emph{real eigenvalues} when both $X$ and $A$ are real, and shows that
\beq\label{eq:thm_2_1_simple}
	\E\left[\sum_{i:\sigma_{i}(X+A)\in\caI}\sqrt{\caO_{ii}(X+A)}\right]\lesssim N^{\delta}\cdot \sqrt{N}(\sqrt{N}\absv{\caI}),
\eeq
where $\caI\subset\R$ can be any interval in the bulk spectrum. \nc
Up to expectation and the $\log N$ (or $N^{\delta}$) factors, both of \eqref{eq:thm_2_simple} and \eqref{eq:thm_2_1_simple} roughly imply $ \caO_{ii}=O(N)$: 
The first factors $N$ and $\sqrt{N}$ indicate the sizes of $\caO_{ii}$ and $\sqrt{\caO_{ii}}$, while $(N\absv{\caD})$ and $(\sqrt{N}\absv{\caI})$ stand for the expected number of eigenvalues in the given domains. Our estimates also naturally provide an upper bound on the \emph{off-diagonal} overlaps $\caO_{ij}: =(\bsl_i, \bsl_j)(\bsr_i, \bsr_j)$ as well by a trivial Schwarz inequality, $|\caO_{ij}|\le ( \caO_{ii}\caO_{jj})^{1/2}$.

Finally,  in Theorems \ref{thm:s_weak} -- \ref{thm:s_real_weak}, 
we  prove almost optimal lower tail bounds on the singular values of $X+A$. To be precise, in Theorem \ref{thm:s_weak} we prove that, \emph{uniformly over all} $s>0$,
\beq\label{eq:thm_3_1_simple}
	\P[N\lambda_{k}(X+A)\leq s]\lesssim \begin{cases}
			s^{2k^{2}}(\absv{\log s}+\log N)^{k} &\text{ for complex $X$, {\Cb fixed $k\geq 1$}},	\\
			s^{k^{2}}(\absv{\log s}+\log N)^{k} & \text{ for real $X$, {\Cb fixed $k\geq 2$}},
	\end{cases}
\eeq
where $\lambda_{k}(X+A)$ denotes the $k$-th smallest singular value of $X+A$. In Theorem \ref{thm:s_strong} we remove the $\log s$ factor to exhibit the optimal $s$-dependence in~\eqref{eq:thm_3_1_simple} assuming more restrictive conditions on $X$, \Cb and extend the result for real $X$ to $k=1$ when $A$ is also real\nc. The exponents $2k^2$ and $k^2$ exactly express the correct strength of  \emph{level repulsion} among singular values. Finally, in Theorem \ref{thm:s_real_weak}, we prove that for \emph{real} $X$, a genuinely \emph{complex shift} $A$ can improve \eqref{eq:thm_3_1_simple} for $k=1$ in the sense that
\beq\label{eq:thm_3_2_simple}
	\P[N\lambda_{1}(X+A)\leq s]\lesssim s^{2}\left(1+\frac{\absv{\log s}+\log N}{\lambda_{1}(\Im[A])}\right)
\eeq
uniformly over all $s>0$. The second power of $s$ shows that the tail of the lowest singular value of $X+A$ follows the behavior of the \emph{complex} ensemble even though $X$ is \emph{real}. As a consequence, for a genuinely complex data matrix $A$, the same regularization effect can be achieved with a real noise matrix as with a complex one.

In the next sections we collect previously known results on $\sigma_{i}$, $\caO_{ii}$, and $\lambda_{k}$, and explain how our results~\eqref{eq:thm_1_simple}--\eqref{eq:thm_3_2_simple} improve upon them. The situation is somewhat different for general $A$ matrix and for the important special case when $A=-z$ scalar matrix for some $z\in \C$. Most  previous results have only been established for the real or complex Ginibre ensemble, i.e. when the entries of $\sqrt{N}X$ are i.i.d. real or complex standard Gaussians.
 Results valid for general i.i.d. matrices will be collected separately.

\subsection{Previous results I: $X=$ Ginibre, $A=$ scalar.}
Since adding a scalar matrix $A$ affects $\sigma_{i}(X+A)$ and $\caO_{ii}(X+A)$ merely as a shift, for these quantities we may assume $A=0$.

The density of eigenvalues $\sigma_i$ was explicitly computed for complex and real Ginibre ensembles respectively in \cite{Ginibre1965} and \cite{Edelman1997}. In particular, results therein imply that uniformly over $z\in\C$
\beq\label{eq:Gini}
	\rho_{\C}(z)\sim \lone(\absv{z}\leq1)+o(1),\qquad \rho_{\C\setminus\R}(z)\sim \lone(\absv{z}\leq 1)\min(\sqrt{N}\absv{\im z},1)+o(1),
\eeq
where $\rho_{\C}$ and $\rho_{\C\setminus\R}$ denote the normalized densities of all and non-real eigenvalues of complex and real Ginibre ensembles, respectively. Taking a domain $\caD$ away from the real axis, our result~\eqref{eq:thm_1_simple} generalizes the upper bound~\eqref{eq:Gini} to an almost boundedness of $\rho_{\C}$ and $\rho_{\C\setminus\R}$ even for general (bounded) matrix $A$ and general i.i.d. distribution for $X$. See Remark \ref{rem:e_R} for more detailed comparison in the real case. For the real eigenvalues of the real Ginibre ensemble $X$ we mention that \cite[Corollaries 4.5 and 5.2]{Edelman-Kostlan-Schub1994} proved
 \beq\label{eq:Gini_R_R}
	E_{N}\deq\absv{\{i:\sigma_{i}(X)\in\R\}}\sim \sqrt{N},\qquad \rho_{\R}(x)\deq \frac{1}{E_{N}}\frac{\dd}{\dd x}\E\absv{\{i:\sigma_{i}(X)\leq x\}}\sim \lone(x\in[-1,1])+o(1).
\eeq
While \eqref{eq:Gini_R_R} is not directly related to \eqref{eq:thm_1_simple}, it shows that the factor of $(\sqrt{N}\absv{\caI})$ in \eqref{eq:thm_2_1_simple} indeed accurately describes the number of real eigenvalues $\sigma_{i}(X+A)$ in a real interval $\caI$.

Likewise, the overlap $\caO_{ii}$ has also been studied thoroughly for Ginibre ensembles. In \cite{Starr-Walters2015} it was proved for the complex Ginibre ensemble that, uniformly over $\absv{z}<1$, 
\beq\label{eq:ovlp_true}
\E[\caO_{ii}(X)\vert \sigma_{i}=z]=N(1-\absv{z}^{2})+o(1),
\eeq
with a sub-exponential error in $N$, where the expectation is conditioned on the event that $\sigma_i=z$. Furthermore, it was recently proved in \cite{Bourgade-Dubach2020,Fyodorov2018} that $\caO_{ii}$ has a heavy tail even after proper rescaling: Uniformly over $\absv{z}<1$ and conditionally on the event $\sigma_{i}=z$,
\beq\label{eq:ovlp_fluc}
\frac{\caO_{ii}(X)}{N(1-\absv{z}^{2})}\Longrightarrow \begin{cases}
	\gamma_{2}^{-1} & \text{if $X$ is complex Ginibre ensemble \cite[Theorem 1.1]{Bourgade-Dubach2020}},\\
	\gamma_{1}^{-1} & \text{if $X$ is real Ginibre ensemble and $z\in\R$ \cite[Theorem 2.1]{Fyodorov2018}},
\end{cases}
\eeq
where $\gamma_{1}$ and $\gamma_{2}$ are gamma distributed random variables with parameters $1$ and $2$. This implies that $\caO_{ii}$ has a heavy tail without first and second moment in the real and complex case, respectively. In particular, the probabilistic $L^{1}$-control in \eqref{eq:thm_2_simple} is practically the strongest possible form of an upper bound for $\caO_{ii}$, and the same applies to \eqref{eq:thm_2_1_simple}. We also mention that for real Ginibre ensemble $X$ and $|\caD|\gtrsim 1/N$, the following nearly optimal upper bound corresponding to the second estimate of \eqref{eq:thm_2_simple} was proved in \cite[Theorem 2.5]{Cipolloni-Erdos-Schroder2022SIAM}:
\beq\label{eq:ovlp_Gini_R_C}
	\E\lone_{\Xi_{\caD}}\sum_{\sigma_{i}(X)\in\caD} \caO_{ii}(X) \lesssim (\log N)\cdot N(N\absv{\caD})
\eeq
for a high-probability event $\Xi_{\caD}$. Comparing \eqref{eq:ovlp_true}--\eqref{eq:ovlp_Gini_R_C} with \eqref{eq:thm_2_simple}--\eqref{eq:thm_2_1_simple}, we find that our results are optimal modulo \Cb the factors of $\log N$ (or $N^{\delta}$) \nc and the denominator in the second estimate of \eqref{eq:thm_2_simple}. The main novelty
 is that we are not restricted to Ginibre ensembles and we allow a general matrix $A$.

Next, we recall previous results on the singular values when $X$ is real or complex Ginibre ensemble. For singular values a scalar shift $A=-z$ does matter. For the special $A=0$ case, the following optimal result was proved in \cite{Szarek1991}:
\beq\label{eq:Szarek}
\P[N\lambda_{k}(X)\leq s]\sim \begin{cases}
	s^{2k^{2}} &\text{if $X$ is complex Ginibre},	\\
	s^{k^{2}} & \text{if $X$ is real Ginibre}.
\end{cases}
\eeq
The exponents reflect the strong level repulsion among the first $k$ singular values indicating an underlying Vandermonde determinant structure for the joint density function.

For the more general $A=-z\ne 0$ case; when $X$ is\ complex Ginibre
 one can explicitly compute $\P[\lambda_{k}(X-z)\leq s]$ for any fixed $k$, in fact the singular values form a determinantal point process \cite{BenArous-Peche2005} and thus the analogue of~\eqref{eq:Szarek} for the complex case holds. No explicit formula is available when $X$ is a real Ginibre ensemble but via supersymmetric methods \cite{Cipolloni-Erdos-Schroder2022SIAM} proved an almost optimal counterpart of \eqref{eq:thm_3_2_simple} for the lowest singular value. See Remarks \ref{rem:s_C} and \ref{rem:s_R_C} for more details on $\lambda_{k}(X+A)$ when $X$ is Ginibre and $A=-z$.

\subsection{Previous results II: $X=$ Ginibre, $A$ general}
As we already mentioned, 
Wegner-type estimates for $\sigma_{i}$'s such as \eqref{eq:thm_1_simple} have not been considered for general $A$ before.

As far as $\caO_{ii}$ and $\lambda_{k}$ 
for general $A$ are concerned, all 
upper bounds on $\caO_{ii}$
 listed below are obtained using (variants of) the inequalities
\beq\label{eq:ovlp_s}
\begin{aligned}
	\pi\E\sum_{i:\sigma_{i}\in\caD}\caO_{ii}(X+A)\leq N^{2}\liminf_{s\to0}\int_{\caD}\frac{\P[N\lambda_{1}(X+A-z)\leq s]}{s^{2}}\dd^{2} z, \\
	2\E\sum_{i:\sigma_{i}\in\caI}\sqrt{\caO_{ii}(X+A)}\leq N\liminf_{s\to0}\int_{\caI}\frac{\P[N\lambda_{1}(X+A-x)\leq s]}{s}\dd x,
\end{aligned}\eeq
proved respectively in \cite[Section 3.6]{Bourgade-Dubach2020} and \cite[Lemma 6.2]{Banks-Vargas-Kulkarni-Srivastava2020} for the condition numbers of the complex and real eigenvalues. Notice that it is essential to have a tail bound on $\lambda_1$  \emph{uniformly} for all small $s$; in particular the prominent  tail bounds on $N\lambda_1$ that
 are valid only down to some scale $s\gtrsim N^{-a}$ with some exponent $a>0$, e.g.~\cite[Section 4.4]{Bordenave-Chafai2012},
\cite[Theorem 2.1]{Tao-Vu2008},
\cite[Theorem 1.18]{Cook2018}, or bounds with a tiny additive factor (to account for the 
possible discrete distribution of $X$ which we exclude in our setup),
e.g.~\cite[Theorem 1.2]{Rudelson-Vershynin2008}
would not be useful in~\eqref{eq:ovlp_s}.

  The inequalities~\eqref{eq:ovlp_s} translate estimates such as \eqref{eq:thm_3_1_simple} and \eqref{eq:thm_3_2_simple} into upper bounds for $\caO_{ii}$, and in particular one can simply replace $A$ by $A-z$ and plug these estimates into \eqref{eq:ovlp_s} if not for the logarithmic corrections. In light of \eqref{eq:ovlp_s}, in what follows, we list previous results on $\lambda_{1}(A+X)$ together with their direct consequences for $\caO_{ii}$, 
  even if the resulting upper bounds for $\caO_{ii}$ were not formulated explicitly.

The first result concerning $\caO_{ii}(X+A)$ for general $A$ appeared in \cite{Banks-Kulkarni-Mukherjee-Srivastava2021}, where the authors proved for real and complex Ginibre ensemble that
\begin{align}
	\P[N\lambda_{1}(X+A)\leq s]&\leq s^{2},	&\E\sum_{i:\sigma_{i}(X+A)\in\caD}\caO_{ii}(X+A)\leq& \frac{N}{\pi} (N\absv{\caD}),&\text{for complex $X$},\label{eq:OA_Gini_C}\\
	\P[N\lambda_{1}(X+A)\leq s]&\leq s, 	&\E\sum_{i:\sigma_{i}(X+A)\in\caI}\sqrt{\caO_{ii}(X+A)}\leq&\frac{\sqrt{N}}{2}(\sqrt{N}\absv{\caI}),&\text{for real $X$ and $A$}\label{eq:OA_Gini_R}.
\end{align}
A weaker version of \eqref{eq:OA_Gini_R} was proved earlier in \cite{Sankar-Spielman-Teng2006} with an additional factor of 2.35. 
 Comparing these estimates with \eqref{eq:ovlp_true}--\eqref{eq:ovlp_Gini_R_C} shows that \eqref{eq:OA_Gini_C}--\eqref{eq:OA_Gini_R}
 are optimal.
 
\subsection{Previous results III: $X, A$ general}
Beyond Gaussian ensembles, \eqref{eq:OA_Gini_C} and \eqref{eq:OA_Gini_R} 
have been generalized to any $X$ such that the entries of $\sqrt{N}X$ have bounded densities, as in the current paper. 
Namely, for such general $X$ and $A$, it is known that
\begin{align}
	\P[N\lambda_{1}(X+A)\leq s]&\lesssim Ns^{2},	&\E\sum_{i:\sigma_{i}(X+A)\in\caD}\caO_{ii}(X+A)\lesssim& 
	\; N^{2}(N\absv{\caD}),&\text{for complex $X$},\label{eq:OA_gen_C}\\
	\P[N\lambda_{1}(X+A)\leq s]&\lesssim \sqrt{N}s, 	&\E\sum_{i:\sigma_{i}(X+A)\in\caI}\sqrt{\caO_{ii}(X+A)}\lesssim 
	&  \; N(\sqrt{N}\absv{\caI}),&\text{for real $X$ and $A$}, \label{eq:OA_gen_R}
\end{align}
where the first result \eqref{eq:OA_gen_C} was essentially\footnote{Estimates therein carry an additional factor of $N$ compared to \eqref{eq:OA_gen_C}, due to an apparent typo in the proof of \cite[Lemma 3.4]{Jain-Sah-Sawhney2021}; the factor $n^{2}$ in the rightmost side of the last inequality in \cite[p.3018]{Jain-Sah-Sawhney2021} should be $n$.} proved in \cite[Lemma 3.4]{Jain-Sah-Sawhney2021} and the second result \eqref{eq:OA_gen_R} is due to \cite[Corollary 1.4]{Tikhomirov2020}. {\Cb Prior to \cite{Tikhomirov2020}, a similar result in the real case appeared in \cite[Corollary 1.8]{Nguyen2018}, which had weaker bound but covered $\lambda_{k}$ beyond $k=1$.} In the same setting with general real $X$, the overlap $\caO_{ii}$ at a genuinely complex eigenvalue $\sigma_{i}$ was considered simultaneously in \cite{Jain-Sah-Sawhney2021} and \cite{Banks-Vargas-Kulkarni-Srivastava2020} improving~\eqref{eq:OA_gen_R} away from the real axis. Specifically, \cite[Proposition 4.2]{Jain-Sah-Sawhney2021} proved for real $X$ and $A$ and $z\in\C$ that
\beq\label{eq:OA_gen_R_C}\begin{aligned}
	&\P[N\lambda_{1}(X+A-z)\leq s]\lesssim \frac{N^{3/2}s^{2}}{\absv{\im z}},&\quad  &\E\sum_{i:\sigma_{i}(X+A)\in\caD}\caO_{ii}(X+A)\lesssim \frac{N^{5/2}(N\absv{\caD})}{\min_{z\in\caD}\absv{\im z}}.
\end{aligned}\eeq
A similar result was obtained in \cite[Theorem 1.5]{Banks-Vargas-Kulkarni-Srivastava2020}  with larger powers of $N$ on the right-hand sides, but this result also covers $\lambda_{k}$ for $k>1$. Notice the additional powers of $N$ in \eqref{eq:OA_gen_C}--\eqref{eq:OA_gen_R_C} compared to the optimal results \eqref{eq:OA_Gini_C}--\eqref{eq:OA_Gini_R} in the Ginibre case. These are largely due to the geometric methods used in \cite{Jain-Sah-Sawhney2021}, namely `invertibility via distance' introduced in \cite{Tikhomirov2020}. In contrast, in~\eqref{eq:thm_2_simple}--\eqref{eq:thm_2_1_simple} we obtain upper bounds on the overlaps and in~\eqref{eq:thm_3_1_simple}--\eqref{eq:thm_3_2_simple} for the singular values with a nearly optimal $N$-dependence using purely analytic methods. 
 
Finally, we remark that a matching \emph{lower} bound for $\caO_{ii}(X+A)$ with general $A$ and i.i.d. $X$ was recently obtained in \cite[Theorem 2.4]{Cipolloni-Erdos-Henheik-Schroder2023arXiv}, which is optimal up to a factor of $N^{o(1)}$. Lower bounds require very different methods than upper bounds and  they are much more robust, in particular $\caO_{ii}(X+A)\ge N^{1-o(1)}$ can be proven with very high probability, while  the overlap has a heavy upper tail. The main method of~\cite{Cipolloni-Erdos-Henheik-Schroder2023arXiv} is a very precise multi-resolvent local law on the Hermitisation of $X+A$ closely related to the \emph{Eigenstate Thermalization Hypothesis} and \emph{Quantum Unique Ergodicity} for Hermitian random matrices. No regularity on the density of $X$ was necessary in \cite{Cipolloni-Erdos-Henheik-Schroder2023arXiv}.

\medskip

Summarizing, the current paper substantially generalizes many previous results on upper bounds on overlaps and lower tail bounds on singular values for matrices of the form $X+A$. We can deal with very general distributions of  $X$ and general matrices $A$ without sacrificing much of the optimality.

The key new approach is a Wegner type estimate for the singular values; an idea that has been exploited earlier in the Hermitian (Wigner) setup~\cite{Erdos-Schlein-Yau2010} but has never been used for non-Hermitian problems. In the main body of this paper we develop the necessary tools to establish a Wegner estimate for the Hermitisation of $X+A$, then we derive all our estimates from this input. Before entering the precise details, in the next subsection we sketch the main ideas, especially we indicate how Hermitian Wegner type estimates  for singular values are used to estimate non-Hermitian eigenvalue density and overlaps.

\subsection{Outline of the proofs}

The common ingredient for our main non-Hermitian results 
\eqref{eq:thm_1_simple}--\eqref{eq:thm_2_1_simple} 
is the  lower tail bound~\eqref{eq:thm_3_1_simple} on the first few singular values
$\lambda_{k}(X+A)$ or  $\lambda_{k}(X+A-z)$. First we explain how we use this bound for our main results
and then we comment on its proof.

For the  proof of~\eqref{eq:thm_1_simple} (Theorem \ref{thm:main}), 
we use Girko's Hermitization formula \cite{Girko1984,Tao-Vu2010} 
\beq\label{eq:Girko}
	\sum_{i=1}^{N}f(\sigma_{i}(X+A))
	=-\frac{1}{4\pi }\int_{\C}\Delta f(z)\int_{0}^{\infty}\Tr \frac{\eta}{(X+A-z)(X+A-z)\adj+\eta^{2}}\dd\eta\dd^{2}z.
\eeq
Girko's formula translates the eigenvalue statistics of $X+A$ into singular value statistics of $X+A-z$, and has been proved to be an effective tool for studying the spectrum of non-Hermitian random matrices; see for example \cite{Alt-Erdos-Kruger2018,Bai1997,Cipolloni-Erdos-Schroder2021,Girko1984,Gotze-Tikhomirov2010,Guionnet-Manjunath-Zeitouni2011,Tao-Vu2015}. 
To control the left-hand side of \eqref{eq:thm_1_simple} we take $f$ in \eqref{eq:Girko} to be a smoothed indicator $\lone_{\caD}$. 
To estimate the right hand side of~\eqref{eq:Girko} from above, we first perform
  two integrations by parts with respect to $z$. Then, for an upper bound, besides the  lower bounds on the singular values
of $X+A-z$, we also need \emph{upper} bounds on certain scalar products of their singular vectors uniformly in $z$ -- this information is
provided by the thermalisation result~\cite[Theorem 2.2]{Cipolloni-Erdos-Henheik-Schroder2023arXiv}.
We remark that exactly the same information was used in~\cite[Theorem 2.4]{Cipolloni-Erdos-Henheik-Schroder2023arXiv} 
to prove the high probability lower bounds for $\caO_{ii}$,
so morally an \emph{upper} bound on the r.h.s. of~\eqref{eq:Girko} requires \emph{lower} bound on 
the singular values and \emph{lower} bound on the overlaps $\caO_{ii}$.

As for the \emph{upper} bounds on the overlap~\eqref{eq:thm_2_simple}--\eqref{eq:thm_2_1_simple} (Theorem \ref{thm:overlap})
we relate the smallest singular value $\lambda_{1}(X+A-z)$ to the overlap $\caO_{ii}(X+A)$ via the contour integral
\beq\label{eq:contour_simple}
	-\frac{1}{2\pi\ii}\oint_{\absv{z-w}=r}\frac{1}{X+A-w}\dd w=\sum_{i:\absv{\sigma_{i}(X+A)-z}<r}\bsr_{i}(X+A)\bsl_{i}(X+A)\adj.
\eeq
Using Weyl's inequality between products of eigenvalues and singular values and 
 the estimates for the \emph{second} smallest singular value $\lambda_{2}(X+A-z)$
 we can take $r\sim N^{-C}$ so small  that there is at most one $\sigma_{i}$ 
within distance $r$ from $z$. For such an $r$, the overlap $\caO_{ii}$ is exactly the squared norm of the right-hand side of \eqref{eq:contour_simple}, thus an upper bound on it can be obtained by the square of the left-hand side, which
in turn can be bounded from above via a lower tail bound on $\lambda_{1}(X+A-z)$. We need to tile the spectrum
into disjoint tiny boxes of linear size $r$ and carefully control the error terms to obtain a linear bound in  the area
$|\caD|$ after summing them up.

Now we explain the proof of the optimal lower tail estimates~\eqref{eq:thm_3_1_simple}--\eqref{eq:thm_3_2_simple} for the lowest singular values $\lambda_{k}(X+A)$. These are proven in Theorems \ref{thm:s_weak}--\ref{thm:s_real_weak} using Wegner type arguments. The key point is that the regularity of the randomness in $X$ forces the singular values to \emph{spread}; they cannot stick to a particular value (like zero) as $X$ is sampled from a continuous probability space. While this mechanism is quite transparent, the essential difficulty is to get the optimal $N$-powers in the estimate. Playing with the randomness of a \emph{single} matrix element $x_{ij}$ would not have a sufficiently strong effect on $\lambda_1$ since $x_{ij}$ has size only $N^{-1/2}$, while playing with \emph{all}  matrix elements simultaneously would not be manageable due to the correlations of their effects. We operate with the randomness of a few rows of $X$ at the same time; this already has the desired effect on $\lambda_1$ and is still technically feasible.

Part of these arguments is inspired by \cite{Erdos-Schlein-Yau2010} that proved Wegner estimates for eigenvalues of Hermitian random matrices. The lower tail of $\lambda_{1}(X+A)$ is estimated by the trace of the resolvent of $(X+A)(X+A)\adj$ outside its spectrum at a negative spectral parameter $-\eta^{2}$ with $\eta\lesssim 1/N$. Then we express the $i$-th diagonal entry of the resolvent as a quadratic form of the $i$-th row vector of $(X+A)$ with a random weight involving the minor of $(X+A)$ with the $i$-th row deleted. We work in the probability space of this row vector which is independent of the  weight and prove effective anti-concentration bounds. Higher singular values $\lambda_{k}(X+A)$ require an inductive procedure to remove $k$ rows. Since the actual proofs are rather technical, we present a more detailed outline in Section~\ref{sec:s_strong}. \Cb See also \cite{Aizenman2017,Farrell-Vershynin2016} for other instances of smoothed analysis in the Hermitian context. \nc

While our proofs of Theorems \ref{thm:s_weak}--\ref{thm:s_real_weak} start with a similar idea as~\cite{Erdos-Schlein-Yau2010}, the proofs in~\cite{Erdos-Schlein-Yau2010} actually assumed much stronger conditions as they were studying the entire bulk spectrum. This resulted in a non-vanishing Hermitian part of the resolvent whose control required an additional integration by parts in the probability space forcing much stronger regularity assumptions. In our case, the resolvent is always skew-Hermitian (see \eqref{eq:G_Schur}) since we work at the \emph{hard edge} of $(X+A)(X+A)\adj$. We also point out that \cite{Erdos-Schlein-Yau2010} assumed certain decay for the Fourier transforms of entry-distribution, which was later weakened in \cite[Appendix A]{Maltsev-Schlein2011} by using Brascamp-Lieb inequality. Instead, here we use a more refined input from \cite{Rudelson-Vershynin2015} whose proof contains the same arguments as in \cite{Maltsev-Schlein2011}, eventually requiring only a minimal regularity condition (bounded density) for the entry-distribution.

As a final remark, we point out that the improved bound~\eqref{eq:thm_3_2_simple}
for the least singular value $\lambda_{1}(X+A)$ for real $X$ and complex $A$ is more difficult to handle than that for complex $X$, while their lower tails have the same $s^{2}$-decay by \eqref{eq:thm_3_1_simple} and \eqref{eq:thm_3_2_simple}. Along the proof of Theorem \ref{thm:s_real_weak}, we find that the singular vectors of $(X+A)$ 
affect the lower tail of the singular values; we need to show that  singular vectors are genuinely complex if $A$ is complex.
Hence the difference between real and complex \emph{non-Hermitian} $X$ is fundamentally
 harder to capture than that between real symmetric and complex \emph{Hermitian} random matrices which do not involve eigenvectors.

\subsection{Organization}
In Section \ref{sec:results}, we rigorously define our model and state the main results. Sections~\ref{sec:Wegner}~and~\ref{sec:O_proof} are devoted to the proofs of Theorems \ref{thm:main} and \ref{thm:overlap} concerning eigenvalues and overlaps of $A+X$, respectively. We prove the singular value estimates in the remaining sections; Sections \ref{sec:s_strong} -- \ref{sec:s_R_weak} contain proofs of Theorems \ref{thm:s_strong}, \ref{thm:s_weak}, and \ref{thm:s_real_weak}, in that order.

\subsection{Notations}
For $x,y\in\R$, we denote $\bbrktt{x,y}=[x,y]\cap\Z$ and $\bbrktt{x}=[1,x]\cap \Z$. For a square matrix $A\in\C^{d\times d}$, we write $\brkt{A}=\frac{1}{d}\Tr A$ for its normalized trace and denote its Hermitian, skew-Hermitian, real, and imaginary parts by
\beq\label{eq:def_re}\begin{aligned}
	\re[A]&=\frac{A+A\adj}{2}, &
	\im[A]&=\frac{A-A\adj}{2\ii }, &
	\Re[A]&=\frac{A+\ol{A}}{2}, &
	\Im[A]&=\frac{A-\ol{A}}{2\ii }.
\end{aligned}\eeq
For a matrix $B$ of any size we denote by $\norm{B}$ its operator norm. For each $i\in\N$, we denote $\bse_{i}$ by the $i$-th coordinate vector, whose dimension can vary by lines. 

For $p\in\N$ and integrable functions $f$ and $g$ respectively on $\R^{p}$ and $\C^{p}$, we denote their Lebesgue integrals by
\beqs
	\int_{\R^{p}}f(\bsx)\dd^{p}\bsx,\qquad \int_{\C^{p}}g(\bsz)\dd^{2p}\bsz.
\eeqs
For a random variable $x$ and $q\geq 1$, we write $\norm{x}_{q}$ to denote $(\E\absv{x}^{q})^{1/q}$. 

We denote $\C_{+}\deq\{z\in\C:\im z>0\}$ and $\D\deq\{\absv{z}<1\}\subset \C$. 
For a Borel set $\caD\subset\C$, we define $\absv{\caD}$ to be its Lebesgue measure.
 The letter $N$ always denotes the dimension of our matrix,  we consider the large $N$ regime
 and we use the standard asymptotic
  relation $\lesssim$ and $\sim$ with respect to $N$. For example, for nonnegative functions
   $f$ and $g$ of $N$ we write $f\lesssim g$ when there is a constant $C>0$ such that 
   $f(N)\leq Cg(N)$ for all  $N$, and write $f\sim g$ when $f\lesssim g$ and $g\lesssim f$ are both true.

\section{Main results}\label{sec:results}
We consider matrices over the real or complex field, denoted commonly by $\bbK=\R$ or $\C$.
\begin{defn}\label{defn:model}
	An $(N\times N)$ real ($\bbK=\R$) or complex ($\bbK=\C$) 
	random matrix $X$ is called \emph{regular} if the following hold true for a constant $\frb>0$:
	\begin{itemize}
		\item[(i)]  The collection $\{\re X_{ij},\im X_{ij} :i,j\in\bbrktt{N}\}$ is independent:
		\item[(ii.$\R$)] When $\bbK=\R$, the random variables $\sqrt{N}X_{ij}$ have densities bounded by $\frb$:
		\item[(ii.$\C$)] When $\bbK=\C$, the random variables $\sqrt{N}\re X_{ij}$ and $\sqrt{N}\im X_{ij}$ have densities bounded by $\frb$.
	\end{itemize}
	A regular matrix $X$ is  called a \emph{regular i.i.d. matrix} if the following hold in addition to (i)--(ii).
	\begin{itemize}
		\item[(iii)] The entries are identically distributed, that is, $X_{11}\overset{d}{=}X_{ij}$ for all $i,j\in\bbrktt{N}$:
		\item[(iv)] $\E X=0$ and $\E XX\adj =N^{-1}I$: If $\bbK=\C$ we also assume $\E X^{2}=0$:
		\item[(v)] For all $p\in \N$, $\frm_{p}\deq\E\absv{\sqrt{N}X_{11}}^{p}$ is finite. 
		In this case we define $\bsfrm\deq(\frm_{p})_{p\in\N}\in \R^{\N}$.
	\end{itemize}
\end{defn}
	Note that the regularity of $X$ alone is unrelated to the decay of its entry distribution, 
	 in particular some moment of $X_{ij}$ may diverge in this case. 
	 For example, when $\bbK=\R$, we may sample entries from the Cauchy distribution, i.e.
	\beqs
		\frac{\dd}{\dd x}\P[\sqrt{N}X_{11}\leq x]=\frac{1}{\pi}\frac{1}{x^{2}+1}.
	\eeqs
	Also note that the shifted matrix $X+A$ remains regular with the same $\frb>0$ for any $A\in\bbK^{N\times N}$ when $X$ is regular. 
	
	For a given deterministic matrix $A\in\C^{N\times N}$, we denote the (complex) eigenvalues of $X+A$ by $\{\sigma_{i}\equiv\sigma_{i}(X+A):i\in\bbrktt{N}\}$ and their normalized empirical distribution by
	\beq\label{eq:rho}
		\rho\equiv\rho_{X+A}^{(N)}\deq\frac{1}{N}\sum_{i=1}^{N}\delta_{\sigma_{i}(X+A)}.
	\eeq
	The eigenvalues of $X+A$ do not have a canonical ordering, but we still consider them indexed by $\bbrktt{N}$ to simplify notations. Nonetheless all our arguments are insensitive to this ad hoc ordering. 
	Finally, for each $z\in\C$ and $r>0$, we define the number of eigenvalues in a ball of radius $r$ about $z$;
	\beqs
	\caN_{z,r}\deq \absv{\{i\in\bbrktt{N}:\absv{\sigma_{i}-z}\leq r\}}.
	\eeqs
	
	Our results on eigenvalues and overlaps are often restricted to the indices $i$ for which $\sigma_{i}$ is well inside the limiting spectrum of $X+A$. We quantify the `bulk' of the non-Hermitian spectrum of $(X+A)$ in terms of the singular value spectrum
	of $(X+A-z)$. Namely, for $\tau>0$ we define 
	\beq\label{eq:Btau_def}
		\caB_{\tau}\deq \Big\{z\in\C:\frac{\dd}{\dd x}(\mu_{\mathrm{sc}}\boxplus \mu^{\mathrm{symm}}_{\absv{A-z}})(0)>\tau\Big\}\subset\C,
	\eeq
	where $\mu_{\mathrm{sc}}$ is the usual semi-circle distribution, $\mu^{\mathrm{symm}}_{\absv{A-z}}$ is the symmetrization of the singular value distribution of $A-z$, and $\boxplus$ denotes the free additive convolution. Using the standard defining equation for the Stieltjes transform of $\mu_{\mathrm{sc}}\boxplus\mu^{\mathrm{symm}}_{\absv{A-z}}$ (see, e.g.~\cite{Pastur1972} or~\cite[Eq. (2.4)]{Cipolloni-Erdos-Henheik-Schroder2023arXiv}), $\caB_{\tau}$ may be written in terms of the singular value distribution of $A-z$ as 
		\beq\label{eq:Btau_char}
		\caB_{\tau}=\Big\{z\in\C:\Brkt{\frac{1}{(A-z)(A-z)\adj +\tau^{2}}}>1\Big\}.
		\eeq
	 Note that $\mu_{\mathrm{sc}}\boxplus\mu^{\mathrm{symm}}_{\absv{A-z}}$ is exactly the symmetrization of the limiting singular value distribution of $(X+A-z)$, see~\cite{Benaych2009,Haagerup-Larsen2000}. Thus $z\in\caB_{\tau}$ for some fixed $\tau>0$ guarantees that $z$ is well inside the limiting spectrum of $(X+A)$. We will make this relation more precise in the following Remark~\ref{rem:DOS} that may be skipped at the first reading as it is not used in the rest of the paper. 
		\begin{rem}[Characterization of the bulk]\label{rem:DOS} For technical convenience we defined the ``bulk'' spectrum of 
		$X+A$ in terms  of $\caB_{\tau}$. Here we show that the domains $\caB_{\tau}$ for small $\tau$	indeed  coincide with the traditional concept of `bulk spectrum' of $\rho$ in the following sense:
	\begin{itemize}
	\item[(i)] The sets $\caB_{\tau}$ roughly cover the support of $\rho$ in the limit  $\tau\to0$; 
	\item[(ii)] for any disk $\caD\subset\caB_{\tau}$ on macroscopic scale ($\absv{\caD}\sim 1$) we have
		\beq\label{eq:DOS_macro}
			\absv{\{i\in\bbrktt{N}:\sigma_{i}\in\caD\}}\sim N\absv{\caD}.
		\eeq
	\end{itemize}
		To make these two statements a bit more precise, consider the Brown measure $\rho_{a+x}$ of the sum $a+x$ in a noncommutative probability space where $a\equiv a_{N}$ has the same $*$-distribution as $A$ and $x$ is a circular element $*$-free from $a$. By a 
		standard corollary of the proof of \cite[Theorem 2.6]{Cipolloni-Erdos-Henheik-Schroder2023arXiv}\footnote{This result gives the optimal bulk local law for the Hermitisation of $X+A$, see also~\eqref{eq:ll}. If we restrict $w\in\ii(\eta,\infty)$ in \cite[Theorem 2.6]{Cipolloni-Erdos-Henheik-Schroder2023arXiv} with a small enough but fixed $\eta$, their result easily extends to all norm-bounded $A$, beyond those with $0$ in the bulk spectrum of $\absv{X+A-z}$. From this local law, a standard argument using Girko's formula \eqref{eq:Girko} implies~\eqref{eq:deformed_circular}, the corresponding analogue of the macroscopic ``circular law''. The standard additional ingredient on the tail of the lowest singular value is given as usual by the regularity condition on $X$.}, we can prove that if $A$ is norm-bounded, then the empirical eigenvalue distribution $\rho=\rho^{(N)}$ from~\eqref{eq:rho} is close to $\rho_{a+x}$ in a weak sense, namely, for any fixed, smooth, compactly supported test function $f$ on $\C$,
		\beq\label{eq:deformed_circular}
			\int_{\C}f(z)\dd(\rho^{(N)}-\rho_{a+x})(z)\to 0\qquad \text{as $N\to\infty$}
		\eeq
		with high probability.
		On the other hand, define the open set $\caB_{0}\subset\C$  by 
		\beq\label{eq:B0}
			\caB_{0}\deq\left\{z\in\C:\Brkt{\frac{1}{(A-z)(A-z)\adj}}>1\right\}=\bigcup_{\tau>0}\caB_{\tau},
		\eeq
		where the last equality is due to \eqref{eq:Btau_char}. For this open set \cite[Theorem B]{Zhong2021}
		 implies that the measure $\rho_{a+x}$ has no atom, it is  supported on $\ol{\caB}_{0}$, and
		 absolutely continuous on $\caB_{0}$ with a strictly positive real analytic density therein. 
		 Together with~\eqref{eq:deformed_circular} this implies  (i), i.e. that  $\mbox{supp} (\rho_{a+x})=\ol{\caB}_{0}$
		 covers the bulk of $\rho$, i.e. the regime where $\rho$ is  typically positive.
		 
		 Furthermore, \cite[Eq. (4.2)]{Zhong2021} (first proved for normal $A$ in \cite[Theorem 1.4]{Bordenave-Caputo-Chafai2014}) gives an implicit formula for the density of $\rho_{a+x}$, which implies the following quantitative estimate on $\caB_{\tau}$:
		\beq\label{eq:DOS_lb}
			\frac{\tau^{2}}{(\norm{A-z}^{2}+\tau^{2})^{2}}\leq \pi\frac{\dd \rho_{a+x}}{\dd m}(z)\leq \norm{A-z}\left(\frac{\norm{A-z}^{2}+\tau^{2}}{\tau^{2}}\right)^{2}+\frac{1}{\tau^{2}},\qquad \forall z\in\caB_{\tau},
		\eeq
		where $m$ denotes the Lebesgue measure on $\C$. Taking $f$ in \eqref{eq:deformed_circular} to be a smoothed indicator
		function  $\lone_{\caD}$ and combining with \eqref{eq:DOS_lb} yields the second statement (ii) and \eqref{eq:DOS_macro}.
	\end{rem}
\subsection{Eigenvalues and eigenvector overlaps of $X+A$}\label{sec:O}

In this section, we 
consider  $X+A$ where $A$ 
is a general deterministic matrix with $\norm{A}=O(1)$ and $X$ is a regular i.i.d. matrix. 

The first result is our Wegner-type estimate for the eigenvalues in the bulk:  it 
asserts that the expected density of $\sigma_{i}$'s, or equivalently, the one point correlation function, is almost bounded in the bulk.
\begin{thm}\label{thm:main}
	Fix $\gamma\in(0,1)$, $\frK>0$, and (small) $\delta,\tau>0$. Let $X$ be a real or complex regular i.i.d. matrix, and
	 $A\in\C^{N\times N}$ be deterministic with $\norm{A}\leq \frK$. 
	 Then there exists a constant $C\equiv C(\gamma,\delta,\tau,\frb,\bsfrm,\frK)>0$ such that the following hold for all $N\in\N$:
	\begin{itemize}
		\item[(i)] If $X$ is complex, then for all $r\in[0,N^{-1/2}]$ and $z\in\caB_{\tau}$ we have
		\beq\label{eq:e_C}
		\prob{\caN_{z,r}\geq 1}\leq \E \caN_{z,r}\leq CN^{\delta}(Nr^{2})^{1-\gamma}.
		\eeq
		\item[(ii)] If $X$ is real, then for all $r\in[0,N^{-1/2}]$, $z\in\caB_{\tau}$, and $\lambda_{1}(\Im[A-z])>2r$ we have
		\beq\label{eq:e_R}
		\prob{\caN_{z,r}\geq 1}\leq \E\caN_{z,r}\leq CN^{\delta}\frac{(Nr^{2})^{1-\gamma}}{\lambda_{1}(\Im[A-z])},
		\eeq
		where we recall that $\lambda_{1}(\Im[B])$ denotes the smallest singular value of $\Im B=(B-\ol{B})/2\ii$ for $B\in\C^{N\times N}$.
	\end{itemize}
\end{thm}
Note that the quotient $\E\caN_{z,r}/(N\pi r^{2})$ tends to the averaged density of states in the limit $r\to 0$, hence Theorem \ref{thm:main} proves the Wegner estimate up to an exponent $\gamma>0$. As a corollary, we show that $\brkt{(X+A-z)^{-1}}$, the normalized trace of the resolvent, is essentially bounded in (probabilistic) \Cb $L^{2-o(1)}$\nc. This is a truly random effect since $z$ lies in the support of the limiting density of states, hence  
\Cb $\brkt{(X+A-z)^{-1}}$ \nc may be unbounded, but its \Cb $(2-o(1))$-th moment \nc is finite. 
\begin{cor}\label{cor:resolv}
	Fix $\frK,\tau>0$, \Cb$\delta_{1}\in (0,1]$ \nc  and let  $\delta>0$ be sufficiently small.
	 Let $X$ be a real or complex regular i.i.d. matrix and $A\in\C^{N\times N}$ with $\norm{A}\leq \frK$. 
	\begin{itemize}
		\item[(i)] If $X$ is complex, then there exists a constant {\Cb $C\equiv C(\delta_{1},\tau,\frb,\bsfrm,\frK)>0$}  such that the following hold for all $z\in\caB_{\tau}$ and $N\in\N$:
		\beq\label{eq:resolv_C}
		\E\absv{\brkt{(X+A-z)^{-1}}}^{\Cb 2-\delta_{1}}\leq \Cb C \nc.
		\eeq
		\item[(ii)] If $X$ is real, then there exists a constant {\Cb $C\equiv C(\delta_{1},\delta,\tau,\frb,\bsfrm,\frK)>0$}  such that the following hold for all $z\in\caB_{\tau}$ and $N\in\N$:
		\beq\label{eq:resolv_R}
		\E\absv{\brkt{(X+A-z)^{-1}}}^{\Cb 2-\delta_{1}}\leq C\left(1+\Cb\frac{N^{-\delta_{1}/2+\delta}}{y}(1\wedge (Ny)^{\delta_{1}/2-\delta})\right)\nc
		\eeq
		\Cb where we abbreviated $y:=\lambda_{1}(\Im[A-z])$. \nc
	\end{itemize}
\end{cor}
\Cb A similar result as \eqref{eq:resolv_C} appeared in \cite[Lemma 2.2]{Rider-Silverstein2006}, with an extra factor of $N^{2}$ on the right-hand side. \nc

\Cb Notice that we take $\delta_{1}\in(0,1]$ in Corollary \ref{cor:resolv}, instead of $(0,2]$. While we use this choice in the proof (see \eqref{eq:resolv_prf}), Corollary \ref{cor:resolv} trivially implies an upper bound for $\E\absv{\brkt{X+A-z}}^{2-\delta_{1}}$ even when $\delta_{1}\in(1,2]$ by H\"{o}lder's inequality. In particular, \eqref{eq:resolv_C} is true for all $\delta_{1}\in(0,2]$, and the same holds true in the real case if $\lambda_{1}(\Im[A-z])$ is of order one.\nc
\begin{rem}\label{rem:main}
	In light of \eqref{eq:Gini} for the Ginibre ensemble, we believe that Theorem \ref{thm:main} remain true for $\delta=0=\gamma$ even in the general case. In our current Theorem \ref{thm:main} the factor $N^{\delta}$ is due to the error in the local laws for the singular values of $X+A-z$, and the exponent $\gamma>0$ is the cost for neglecting the correlation between singular values and vectors. We defer the details to Remark \ref{rem:main_1} after the proof of Theorem \ref{thm:main}. {\Cb Also note that $\delta_{1}>0$ in Corollary \ref{cor:resolv} is necessary since the contribution from a single eigenvalue to the second moment of $\brkt{(X+A-z)^{-1}}$, i.e. $N^{-2}\E\absv{\sigma_{i}-z}^{-2}$, is already logarithmically divergent.}
\end{rem}

\begin{rem}\label{rem:e_R}
	When $A=0$ and $X$ is real Ginibre ensemble, \eqref{eq:Gini} and \eqref{eq:Gini_R_R} imply that $\E \caN_{z,r}$ remains bounded as long as $\absv{\im z}\gtrsim N^{-1/2}$. Taking $A=0$ in our results, \eqref{eq:e_R} and \eqref{eq:resolv_R}, they do not feature this sharp $\absv{\im z}$-dependence due to the factor $\absv{\im z}^{-1}$ in \eqref{eq:e_R}. This factor originates from Theorem \ref{thm:s_real_weak}, where we prove a lower tail estimate for the smallest singular value of $(X+A-z)$ that also carries the same $\absv{\im z}^{-1}$ factor. In Remark \ref{rem:s_R_C} following Theorem \ref{thm:s_real_weak}, we explain more details on its source and a possibly optimal $\absv{\im z}$-dependence.
\end{rem}

Recall that when the spectrum of $(X+A)$ is simple we have the spectral decomposition
\beq
X+A=\sum_{i}\sigma_{i}\bsr_{i}\bsl_{i}\adj,\qquad \bsl_{i}\adj\bsr_{j}=\delta_{ij}
\eeq
and that the overlaps between eigenvectors $\bsl_{i}$ and $\bsr_{j}$ are defined as
\beq\label{eq:def_O}
\caO_{ij}\deq \bsl_{i}\adj\bsl_{j}\bsr_{j}\adj\bsr_{i}.
\eeq
Before presenting our results on the diagonal overlaps $\caO_{ii}$, we briefly pause to show that $X+A$ has simple spectrum with probability $1$ when $X$ is regular. To see this, notice that $X+A$ has a repeated eigenvalue if and only if the characteristic polynomial and its derivative has a common root. This is in turn equivalent to the resultant of these two polynomials being identically zero. Thus $X+A$ is not simple when its entries satisfy a polynomial equation, so that $X$ is in an algebraic submanifold of $\bbK^{N\times N}$ with positive codimension. Since this manifold has Lebesgue measure zero and $X$ has a density, $X$ falls into this set with probability zero. 

In the next two theorems, we prove upper bounds for the expectation of $\caO_{ii}$. 
\Cb The first result holds anywhere in the spectrum but only on a set with very high probability.  In the  second  
result we remove this exceptional set in the bulk. \nc
\begin{thm}\label{thm:overlap}
	Fix {\Cb $\frK_{0}>0$} and (large) $D,K>0$. Let $X$ be a real or complex regular i.i.d. matrix and $A\in\C^{N\times N}$ with {\Cb $\norm{A}\leq N^{\frK_{0}}$.} Then there exists a constant {\Cb $C\equiv C(\frb,\bsfrm,\frK_{0},D,K)>0$} such that the following hold true:
	\begin{itemize}
		\item[(i)]If $X$ is complex, for any square $\caD\subset\C$ with $\absv{\caD}\geq N^{-2K}$ we have
		\beq\label{eq:overlap_strong}
		\E\lone_{\Xi_{\caD}}\sum_{i:\sigma_{i}\in\caD}\caO_{ii}\leq C N(\log N)^{2}(N\absv{\caD}),
		\eeq
		for some event $\Xi_{\caD}$ with $\P[\Xi_{\caD}^{c}]\leq N^{-D}$.
		\item[(ii)]If $X$ is real, for any square $\caD\subset\C$ with $\absv{\caD}\geq N^{-2K}$ and $\min_{z\in\caD}\lambda_{1}(\Im[A-z])\geq N^{-K}$ we have
		\beq\label{eq:overlap_strong_R}
		\E\lone_{\Xi_{\caD}}\sum_{i:\sigma_{i}\in\caD}\caO_{ii}\leq C N(\log N)^{2} \frac{{\Cb(1+\norm{A}^{2})}(N\absv{\caD})}{\min_{z\in\caD}\lambda_{1}(\Im[A-z])},
		\eeq 
		for some event $\Xi_{\caD}$ with $\P[\Xi_{\caD}^{c}]\leq N^{-D}$.
	\end{itemize}
	In particular, if $X$ is complex, for each fixed $0<\epsilon'<\epsilon$ and $K>0$ we have
	\beq\label{eq:overlap_weak}
	\P\left[\sum_{i:\sigma_{i}\in \caD}\caO_{ii}\geq N^{1+\epsilon}N\absv{\caD}\right]\leq N^{-\epsilon'}
	\eeq 
	uniformly over squares $\caD\subset\C$ with $\absv{\caD}>N^{-2K}$. The same holds true for real $X$ 
	for all $\caD\subset\C$ with $\absv{\caD}>N^{-2K}$ as long as $\min_{z\in\caD}\lambda_{1}(\Im[A-z])\sim1$.
\end{thm}

\begin{rem}
	{\Cb Notice that in Theorem \ref{thm:overlap} we only assume $\norm{A}$ is at most polynomially large in $N$, allowing $A$ to be much larger than $X$. A closely related canonical model is $\gamma A+X$, 
	with $\norm{A}=O(1)$, where $\gamma>0$ is a large control parameter. In fact, controlling the effect 
	of large $\gamma$ on the condition number $\caO_{ii}$ has been a central task in quantitative linear algebra, e.g. in \cite{Banks-Kulkarni-Mukherjee-Srivastava2021,Banks-Vargas-Kulkarni-Srivastava2020,Jain-Sah-Sawhney2021}
	after the trivial rescaling $\gamma A+X = \gamma ( A+ \gamma^{-1}X)$ and
	viewing $\gamma^{-1}X$ as a small random perturbation of the deterministic $A$.
	 Since the constant $C>0$ in Theorem \ref{thm:overlap} depends only on $\log \norm{A}$ via $\frK_{0}$, our result can be used for such a purpose. In fact, if we replace $A$ with $\gamma A$ in Theorem \ref{thm:overlap} (i) and (ii), we recover the same $\gamma$-dependence for $\caO_{ii}$ as in respectively \cite[Theorem 1.5]{Banks-Kulkarni-Mukherjee-Srivastava2021} and \cite[Proposition 6.4]{Banks-Vargas-Kulkarni-Srivastava2020} modulo $\log\gamma$ factors.}
	
	We also remark that our proof of Theorem \ref{thm:overlap} easily extends the result to regular matrices
	 under the additional assumptions that $\caD$ is contained in a  bounded 
	 set, moreover, if $X$ is real, we also need to assume that $\E\norm{X}^{4}=O(1)$.	
\end{rem}

Since \eqref{eq:overlap_strong} and \eqref{eq:overlap_strong_R} are true for all fixed $K>0$ and additive in $\caD$, 
we may easily deduce the same result for many other domains that can be approximately tiled by small squares. 
Indeed, in Appendix \ref{append:Minkowski} we prove that Theorem~\ref{thm:overlap} holds true not only for
 square domains but also for any Borel measurable domain $\caD$ satisfying a
natural geometric regularity condition of the form $\absv{\caD+[-N^{-K},N^{-K}]^{2}}\leq C\absv{\caD}$ 
for some constants $C,K>0$. Here $+$ refers to the Minkowski sum of $\caD$ and a small square. 
In particular, we may take $\caD$ to be a disk or a ``polynomially thin'' rectangle. 

In the next theorem, we remove the probabilistic factor $\lone_{\Xi_{\caD}}$ \Cb from \eqref{eq:overlap_strong} assuming that $\caD$ 
is in the bulk $\caB_{\tau}$, and  extend the result to $\caO_{ii}$'s at real eigenvalues when $A$ and $X$ are both real.\nc
\begin{thm}\label{thm:overlap_strong}
	Fix $\frK>0$ and (small) $\delta,\tau>0$. Let $X$ be a real or complex regular i.i.d. matrix and $A\in\C^{N\times N}$ with $\norm{A}\leq \frK$. Then there exists a constant $C\equiv C(\delta,\tau,\frb,\frm,\frK)$ such that the following hold:
	\begin{itemize}
		\item[(i)] If $X$ is complex, for any Borel set $\caD\subset\caB_{\tau}$ we have
		\beq
			\E\sum_{i:\sigma_{i}\in\caD}\caO_{ii}\leq CN^{1+\delta}(N\absv{\caD}).
		\eeq
		\item[(ii)] If $X$ and $A$ are real, for any Borel set $\caI\subset\caB_{\tau}\cap\R$ we have
		\beq
			\E\sum_{i:\sigma_{i}\in\caI}\sqrt{\caO_{ii}}\leq CN^{1/2+\delta}(\sqrt{N}\absv{\caI}),
		\eeq
		where $\absv{\caI}$ is the Lebesgue measure of $\caI$ in $\R$.
	\end{itemize}
\end{thm}

The most relevant choice for $\caD$ in both of Theorems \ref{thm:overlap} and \ref{thm:overlap_strong} is a square of area $N^{-1+\epsilon}$ in the bulk, so that typically there is at least one eigenvalue in $\caD$ with high probability. More precisely, since the local law in~\cite[Theorem 2.6]{Cipolloni-Erdos-Henheik-Schroder2023arXiv} covers all mesoscopic scales, we can extend the convergence in \eqref{eq:deformed_circular} to test functions $f$ on mesoscopic scales supported in $\caB_{\tau}$; see e.g.~\cite[Section 5.2]{Alt-Erdos-Kruger2018} for a completely analogous proof in a slightly different setup. In particular \eqref{eq:DOS_macro} extends to squares $\caD\subset\caB_{\tau}$ with $\absv{\caD}=N^{-1+\epsilon}$, so that there is at least one eigenvalue in $\caD$ with high probability. For such sets $\caD$, Theorem \ref{thm:overlap_strong} and \eqref{eq:ovlp_true} are on the same footing; our result shows that the expectation of $\caO_{ii}$ is bounded by $N^{1+\epsilon}$. 

In an unrelated context, we  remark that Theorem~\ref{thm:overlap_strong}(i) can also be used to give an upper bound on  the diffusivity of the complex Dyson-type eigenvalue dynamics defined via \eqref{eq:DBM}.  This matches the analogous lower bound given in~\cite[Eq. (2.14)]{Cipolloni-Erdos-Henheik-Schroder2023arXiv}.

\subsection{Small singular values of shifted regular matrices}\label{sec:s_result}
As mentioned in the introduction, our proofs of Theorems \ref{thm:main}--\ref{thm:overlap_strong} translate the problems to understanding the singular values of $X+A-z$ for any shift parameter $z\in\C$. In this section we present results on the singular values of $X+A$, where $X$ is a general regular matrix in Definition \ref{defn:model} and $A\in\C^{N\times N}$ is a deterministic matrix. By replacing $A$ with $A-z$, we use the results in this section as inputs for those in Section~\ref{sec:O}. These results are of independent interest, so we list them in this section separately. 

We stress that all results in this section except Theorem \ref{thm:s_strong} will be uniform in $A$, in particular no norm bound on $A$ is assumed unlike in the previous Section~\ref{sec:O}. This means that by a simple rescaling $\gamma X +A = \gamma(X+ A/\gamma)$, $\gamma>0$, these results also hold for the singular values of the matrix $\gamma X +A$ with a rescaled noise, as often presented in numerical applications, e.g.~\cite{Banks-Kulkarni-Mukherjee-Srivastava2021,Banks-Vargas-Kulkarni-Srivastava2020,Jain-Sah-Sawhney2021}.

We denote the ordered singular values of $X+A$ by
\beq\label{eq:def_sing}
0\leq \lambda_{1}\leq\cdots\leq\lambda_{N}.
\eeq
As eigenvalues of a Hermitian random matrix problem, $\lambda_{k}$'s are subject to \emph{level repulsion} that in turn forces an especially small lower tail probability for them. The next theorems contain such results. Their optimality will be discussed afterwards.

\begin{thm}[Regular matrices]\label{thm:s_weak}
	Let $k\in\N$ be fixed, $A\in\C^{N\times N}$, and $X$ be a \emph{regular} matrix. Then  for 
	the singular values \eqref{eq:def_sing} of $(X+A)$ we have the following:
	\begin{itemize}
		\item[(i)]{\rm[Complex case]} If $X$ is complex, there exists a constant $C\equiv C(k,\frb)>0$ such that for all $N\geq k\vee 2$ and $s\in[0,1]$
		\beq\label{eq:s_C_weak}
		\P[N\lambda_{k}\leq s]\leq C\left(\absv{\log s}+\log N\right)^{k}s^{2k^{2}}.
		\eeq
		\item[(ii)]{\rm[Real case]} If $X$ is real, there exists a constant $C\equiv C(k,\frb)>0$ such that for all {$\Cb 2\leq k\leq N$} and $s\in[0,1]$
		\beq\label{eq:s_R_weak}
		\P[N\lambda_{k}\leq s]\leq C(\absv{\log s}+\log N)^{k}s^{k^{2}}.
		\eeq
		\end{itemize}
\end{thm}

{\Cb Notice that the result for real $X$, as formulated in \eqref{eq:s_R_weak}, requires $k\geq 2$. The origin of this restriction will be explained in Section \ref{sec:s_weak}, along the proof of Proposition \ref{prop:repel_weak}. Nevertheless, 
we  still have results that are optimal in $s$ for $k=1$ and real $X$, see the following  Theorems \ref{thm:s_strong} and \ref{thm:s_real_weak}, respectively for real $A$ and genuinely complex $A$. }

Note that the constant $C$ in Theorem \ref{thm:s_weak} does not depend on $A$. Also note that the only assumption on $X$ is the regularity, and in particular some moments\footnote{Still one should think of the second moment of $\sqrt{N}X_{ij}$ to be finite, for otherwise $\lambda_{k}\sim N^{-1}$ might not be true; see \cite{Belinschi-Dembo-Guionnet2009} for the case when $h_{ij}(x)\sim 1/\absv{x}^{1+\alpha}$ with $\alpha\in(0,2)$. Our result still holds for such matrices, but our $N\lambda_{k}$ scaling may not be adequate.} of the entries of $X$ may be infinite. Although Theorem \ref{thm:s_weak} has such robustness, it is really relevant when the (non-Hermitian) spectrum of $X+A$ indeed reaches zero, otherwise typically even the lowest singular value $\lambda_{1}$ is far away from zero and its $1/N$ scaling indicated by \eqref{eq:s_C_weak}--\eqref{eq:s_R_weak} is completely off. In particular, recalling the definition of $\caB_{\tau}$ in \eqref{eq:Btau_def}, this is the case for $X+A-z$ if $X$ is a regular i.i.d. matrix and $z\in\caB_{\tau}$. To be precise, denoting the singular values of $X+A-z$ by
\beq\label{eq:def_sing_z}
	0\leq \lambda_{1}^{z}\leq\cdots\leq\lambda_{N}^{z},
\eeq
we remove the $\absv{\log s}$ factors from Theorem \ref{thm:s_weak} under this setting.
\begin{thm}[Regular i.i.d. matrices]\label{thm:s_strong}
	Fix $k\in\N$ and $\delta,\tau,{\Cb\frK}>0$. Let $X$ be a \emph{regular i.i.d.} matrix and $A\in\C^{N\times N}$ {\Cb with $\norm{A}\leq\frK$}. Then we have the following for the singular values \eqref{eq:def_sing_z} of $X+A-z$.
	\begin{itemize}
		\item[(i)]{\rm [Complex case]} If $X$ is complex, there exist constants $C_{1}\equiv C_{1}(k,\frb)$ and $C_{2}\equiv C_{2}(k,\tau,\delta,\bsfrm,{\Cb\frK})$ such that for all $N\geq 3k$, $s\in[0,1]$, and $z\in\caB_{\tau}$ we have
		\beq\label{eq:s_C_strong}
		\P[N\lambda_{k}^{z}\leq s]\leq C_{1}s^{2k^{2}}\E\left[(1+N\lambda_{3k}^{z})^{2k^{2}}\right]\leq C_{1}C_{2}N^{\delta}s^{2k^{2}}.
		\eeq
		
		\item[(ii)]{\rm [Real case]} If $X$ and $A$ are both real, there exist constants $C_{1}\equiv C_{1}(k,\frb)$ and $C_{2}\equiv C_{2}(k,\tau,\delta,\bsfrm,{\Cb\frK})$ such that for all $N\geq 3k$, $s\in[0,1]$, and $z\in\caB_{\tau}\cap\R$ we have
		\beq\label{eq:s_R_strong}
		\P[N\lambda_{k}^{z}\leq s]\leq C_{1}s^{k^{2}}\E\left[(1+N\lambda_{3k}^{z})^{k^{2}}\right]\leq C_{1}C_{2}N^{\delta}s^{k^{2}}.
		\eeq
	\end{itemize}
\end{thm}

As easily seen from \eqref{eq:s_C_strong} and \eqref{eq:s_R_strong}, the only suboptimal factor $N^{\delta}$ comes from the estimate for $\lambda_{3k}^{z}$. This is due to local law (see Lemma \ref{lem:ll} for its statement), which carries a small power of $N$. \Cb Note that 
for real $X$, Theorem \ref{thm:s_strong} (ii) applies also to $k=1$, but requires $A$ to be real. \nc

\begin{rem}\label{rem:s_C}
	For complex Gaussian $X$ and general $A$, $\lambda_{k}$'s are exactly the square roots of eigenvalues of deformed Laguerre unitary ensemble, whose joint density was computed in \cite[Proposition 3.1]{BenArous-Peche2005}. In particular one can easily deduce
	 that the joint density of unordered eigenvalues $(x_{1},\cdots,x_{N})$ of $(X-z)(X-z)\adj$ is proportional (ignoring $N$-factors) to
	\beq\label{eq:dLUE}
	V(\bsx)\exp\left(-N^{2}\absv{z}^{2}-N\sum_{i\in\bbrktt{N}}x_{i}\right)\det \left( x_{j}^{i-1}F^{(i-1)}(N^{2}\absv{z}^{2}x_{j})\right)_{i,j\in\bbrktt{N}},
	\eeq
	where $V(\bsx)\deq \prod_{i<j}(x_{i}-x_{j})$ is the Vandermonde determinant, 
	and $F$ is a real analytic function directly related to the zeroth modified Bessel function $I_{0}$.
	After some algebra, one can easily see that the  second determinant in
	 \eqref{eq:dLUE} is $O(V(\bsx))$ up to an $N$-dependent factor. Thus we obtain that (ignoring $N$-factors) 
	\beq
	\P[\lambda_{k}(X-z)\leq s]=\P[\absv{\{x_{j}:x_{j}\leq s^{2}\}}\geq k]\lesssim (s^{2})^{k}\cdot (s^{2})^{k(k-1)}=s^{2k^{2}},
	\eeq
	where the first factor comes from the volume of $[0,s^{2}]^{k}$ and the second 
	from $V(\bsx)^{2}$. In particular \eqref{eq:dLUE} as well as \eqref{eq:Szarek} show 
	that Theorem~\ref{thm:s_strong} is optimal for general $k$ as far as the $s$-power is concerned in the  small $s$ regime.
	\end{rem}

\Cb In the next result, we show that if the entrywise imaginary part $\Im A$ of the shift matrix $A$ is positive definite, Theorem \ref{thm:s_weak} (ii) extends to $k=1$ with quadratic decay $s^{2}$; in terms of Theorem \ref{thm:s_strong}, the linear $s$-dependence in \eqref{eq:s_R_strong} for $k=1$ can be improved to $s^{2}$.\nc
\begin{thm}[Real case, complex shift, improved]\label{thm:s_real_weak}
	Let $X$ be a \emph{real} regular matrix, $A\in\C^{N\times N}$, and $\lambda_{1}$ be the smallest singular value of $(X+A)$. Then there exists a constant $C\equiv C(\frb)>0$ such that the following holds for all $N\geq 4$ and $s\in[0,1]$;
	\beq\label{eq:s_real_weak}
	\P[N\lambda_{1}\leq s]\leq Cs^{2}\left(1+\frac{(\E\norm{X}^{4})^{1/2}+\norm{A}^{2}}{\lambda_{1}(\Im A)}(\absv{\log s}+\log N)\right).
	\eeq
	Consequently, if $X$ is a real regular \emph{i.i.d.} matrix, $A$ is real, and $\frK>0$, then there exists a constant $C\equiv C(\frb,\bsfrm,\frK)$ such that the following holds whenever $\norm{A}\leq\frK$, $z\in\C$, $N\geq 4$, and $s\in[0,1]$;
	\beq\label{eq:s_real}
	\P[N\lambda_{1}^{z}\leq s]\leq Cs^{2}\left(\frac{1+\absv{z}^{2}}{\absv{\im z}}(\absv{\log s}+\log N)\right).
	\eeq
\end{thm}

\begin{rem}\label{rem:s_R_C}
	In particular when $\absv{\im z}\sim 1$, \eqref{eq:s_real} shows that $\P[\lambda_{1}^{z}\leq s]=O(s^{2})$ up to a logarithmic factor, similarly to the complex case. The same improvement due to (genuinely) complex shift was previously proved for real Ginibre ensemble $X$ and $A=0$ in \cite[Theorem 2.1]{Cipolloni-Erdos-Schroder2022SIAM}; for $z$ in the bulk, the result therein reads
	\beq\label{eq:Gini_R}
		\P[N\lambda_{1}^{z}\leq s]\lesssim (1+\absv{\log s})s^{2}+\e{-(\sqrt{N}\absv{\im z})^{2}}\left(s\wedge\frac{s^{2}}{\sqrt{N}\absv{\im z}}\right).
	\eeq
	From \eqref{eq:Gini_R} it is clear that the critical scale of $\absv{\im z}$ is $N^{-1/2}$, and that the right-hand side is roughly of the same size as $s^{2}$ or $s$ depending on whether $\absv{\im z}$ exceeds that scale or not. Our result fails to capture this scale of transition due to the suboptimality of Lemma \ref{lem:real_1}, where we estimate to what extent are the singular vectors of $(X-z)$ genuinely complex when $z$ is. We refer to Remark \ref{rem:real_1} for more details.
\end{rem}

\section{Wegner-type estimate for $X$: Proof of Theorem \ref{thm:main}}\label{sec:Wegner}
	Throughout the rest of the paper, we fix parameters $\tau,\frb,\bsfrm$, and $\frK$; since they affect Theorems~\ref{thm:main} -- \ref{thm:s_real_weak} only implicitly via the constant $C$, considering them to be fixed does no harm to the proof.
	
	For each $A\in\C^{N\times N}$ and $z\in\C$, we define $H_{z}\equiv H_{z,A}\in \C^{2N\times 2N}\cong\C^{N\times N}\otimes \C^{2\times 2}$ as 
	\begin{align}\label{eq:Hermit}
	H_{z}&\deq \begin{pmatrix}
		0 & X+A-z \\ (X+A-z)\adj & 0
	\end{pmatrix},
	\end{align}
	which is also known as the \emph{Hermitization} of $X+A$. As in \eqref{eq:Hermit}, we often omit the dependence on $A$ for simplicity. We may write the spectral decomposition of $H_{z}$ as
	\beq
		H_{z}=\sum_{\absv{i}\in\bbrktt{N}}\lambda_{i}^{z}{\bf w}_{i}^{z}({\bf w}_{i}^{z})\adj=\sum_{\absv{i}\in\bbrktt{N}}\lambda_{i}^{z}\begin{pmatrix}
			\bsu_{i}^{z} \\ \bsv_{i}^{z}
		\end{pmatrix}\begin{pmatrix}
		\bsu_{i}^{z} \\ \bsv_{i}^{z}
	\end{pmatrix}\adj ,\qquad \norm{{\bf w}_{i}^{z}}^{2}=\norm{\bsu_{i}^{z}}^{2}+\norm{\bsv_{i}^{z}}^{2}=1,
	\eeq
	where the eigenvalues $\{\lambda_{i}^{z}:\absv{i}\in\bbrktt{N}\}$ are increasingly ordered and we decomposed the eigenvectors 
	${\bf w}^z_i\in \C^{2N}$ into two parts with $\bsu_{i}^{z},\bsv_{i}^{z}\in\C^{N}$. Note that $\lambda_{i}^z$ and 
	$\bsu_{i}^z,\bsv_{i}^z$ for $i\geq 1$ are exactly the singular values and 
	left and right singular vectors of $(X-z)$, respectively. We denote the resolvent 
	of $H_{z}$ by 
	\beqs
	G_{z}(w)\deq (H_{z}-w)^{-1}
	\eeqs
	for each $w\in\C_{+}$. To simplify the notation, we further omit the dependence on $z$ to write, for example, $G\equiv G_{z}$ and $\lambda_{i}\equiv \lambda_{i}^{z}$ when there is no confusion.

	Since $H_{z}$ has only block off-diagonal component, it is clear that $\lambda_{-i}=-\lambda_{i}$ and $\brkt{G(\ii\eta)}=\ii\brkt{\im G(\ii\eta)}$ hold true. Again due to the block structure of $H_{z}$, we may take $(\bsu_{i},\bsv_{i})$ that satisfies $\bsu_{-i}=\bsu_{i}$ and $\bsv_{-i}=-\bsv_{i}$, which in turn implies $\norm{\bsu_{i}}^{2}=\norm{\bsv_{i}}^{2}=1/2$ whenever $\lambda_{i}\neq 0$. Since the event $[\lambda_{1}>0]$ has probability $1$, we will work on this event in what follows. 

For $w\in\C_{+}$, it is known that $G_{z}(w)$ concentrates around a deterministic block
 constant matrix $M_{z}(w)\in (\C^{N\times N})^{2\times 2}$; such results
 are commonly called \emph{local laws}. The matrix $M_{z}(w)\equiv M$ is the unique solution 
 of the matrix Dyson equation
\beq\label{eq:MDE}
	-M^{-1}=-\begin{pmatrix}
		-w & A-z \\ (A-z)\adj & -w
	\end{pmatrix}
	+\brkt{M},
\eeq
with the side condition $\im M=(M-M\adj)/(2\ii)\geq0$. Note that the \emph{self consistent density of states (scDos)}
corresponding to $H_z$, given by $\rho^{z}(x)=\pi^{-1}\brkt{\im M(x+\ii 0)}$, is exactly the density of the measure $\mu_{\mathrm{sc}}\boxplus \mu^{\mathrm{symm}}_{\absv{A-z}}$ that we used to define the bulk in \eqref{eq:Btau_def}.

More concretely, we have the following local law for $H_{z}$ from 
\cite[Theorem 2.6]{Cipolloni-Erdos-Henheik-Schroder2023arXiv}:
\begin{lem}[{\cite[Theorem 2.6]{Cipolloni-Erdos-Henheik-Schroder2023arXiv}}]\label{lem:ll}
	Let $\tau,\frK>0$ be fixed.	Then for each (small) $\epsilon,\xi>0$ and (large) $D>0$, there exists $N_{0}\in\N$ such that the following holds uniformly over $N\geq N_{0}$, $\norm{A}\leq\frK$, $z\in\caB_{\tau}$, and $\eta\in[N^{-1+\epsilon},1]$;
	\beq\label{eq:ll}
	\prob{\absv{\brkt{G_{z}(\ii\eta)-M_{z}(\ii\eta)}}\geq \frac{N^{\xi}}{N\eta}}\leq N^{-D}.
	\eeq
	Consequently, for any fixed (small) $\epsilon,\xi>0$ and (large) $D>0$, the following holds uniformly over $\absv{z}\leq\tau$;
	\beq\label{eq:rigid}
		\P[\Xi_{z}(\epsilon,\xi)^{c}]\leq N^{-D},\qquad \Xi_{z}(\epsilon,\xi)\deq\bigcap_{\eta\in[N^{-1+\epsilon},1]}\left[\absv{\{i\in\bbrktt{N}:\lambda_{i}^{z}\leq \eta\}}\leq N^{1+\xi}\eta\right].
	\eeq
\end{lem}
The second result \eqref{eq:rigid} is a direct consequence of the local law \eqref{eq:ll} via the inequality
\beqs
	\absv{\{i\in\bbrktt{N}:\lambda_{i}^{z}\leq \eta\}}\leq 2N\eta\brkt{\im G_{z}(\ii \eta)}
\eeqs
together with the fact that $\norm{M}\lesssim 1$ from \cite[Appendix A]{Cipolloni-Erdos-Henheik-Schroder2023arXiv}.

Next, we present the two main technical inputs for our proof of Theorem \ref{thm:main}, (i) an upper bound for $\im\brkt{G_{z}(\ii\eta)}$ with small $\eta>0$ and (ii) a high-moment bound for overlaps between the left and right eigenvectors $\bsu_{i}$ and $\bsv_{i}$.
\begin{lem}\label{lem:gap}
	Let $X$ be a real or complex regular i.i.d. matrix and $A\in\C^{N\times N}$ with $\norm{A}\leq\frK$. Then for each fixed $\delta,\tau,\xi>0$ and $\epsilon\in[0,1]$, there exists a constant $C>0$ such that the following holds uniformly over all $z\in\caB_{\tau}$ and $\eta\in(0,1)$;
	\beq\label{eq:DoF}
	\E\absv{\brkt{G_{z}(\ii\eta)}}^{2+\epsilon}\leq
	\begin{cases}
		1+CN^{\delta}(N\eta)^{-\epsilon-\xi}\left(\absv{\log \eta}+\log N\right) & \text{if $X$ is complex},	\\
		1+CN^{\delta}(N\eta)^{-\epsilon-\xi}\dfrac{\absv{\log \eta}+\log N}{\absv{\im z}} & \text{if $X$ is real}.
	\end{cases}
	\eeq
\end{lem}
\begin{prop}\label{prop:overlap}{\cite[Theorem 2.2]{Cipolloni-Erdos-Henheik-Schroder2023arXiv}}
	 Let $X$ and $A$ be as in Lemma \ref{lem:gap}, and fix $\delta,\tau>0$ and $p\in\N$. Then there exists a constant $c\in (0,1)$ such that the following holds uniformly over $z\in\caB_{\tau}$;
	\begin{align}\label{eq:sing_ovlp}
		N^{p}\E \sup_{i,j\in\bbrktt{cN}}
		\absv{\brkt{\bsu_{j}^{z},\bsv_{i}^{z}}-\delta_{i,j}q_{i}(z)}^{2p}&\lesssim N^{\delta},
	\end{align}
	where $\{q_{i}(z):i\in\bbrktt{cN}\}$ is the collection of deterministic functions of $z$ defined by
	\beq\label{eq:qi_def}
		q_{i}(z)\deq \frac{\brkt{\im[M(\gamma_{i})]F}}{\brkt{\im M(\gamma_{i})}},\qquad F\deq \begin{pmatrix} 0 & 0 \\ 1 & 0 \end{pmatrix}
	\eeq
	and $\gamma_{i}$  is the $i$-th $N$-quantile of the scDOS of $H_{z}$, i.e.
	\beq
		\gamma_{i}=\inf\left\{x\in\R:\int_{0}^{x}\rho^{z}(t)\dd t\geq \frac{i}{2N}\right\}.
	\eeq
	Furthermore, there exists a constant $C>0$ such that
	\beq\label{eq:qi}
		\absv{q_{i}(z)}\leq C\frac{i}{N},\qquad \forall i\in\bbrktt{cN},\forall z\in\caB_{\tau}.
	\eeq
\end{prop}
The proof of Lemma \ref{lem:gap} is postponed to the end of this section. It heavily relies on the tail estimate of the lowest singular value $\lambda_{1}^{z}$ from Theorem \ref{thm:s_weak}(i) for the complex case and Theorem \ref{thm:s_real_weak} for the real case. These theorems are proven separately in Sections~\ref{sec:s_strong} and~\ref{sec:s_weak}. Proposition \ref{prop:overlap} is a direct consequence of the thermalisation result in~\cite[Theorem 2.2]{Cipolloni-Erdos-Henheik-Schroder2023arXiv} which computes any quadratic form $({\bsu}^z_i, B {\bsv}^z_j)$ of the bulk singular vectors of $X+A-z$ in terms of a leading deterministic quantity up to an error or order $N^{-1/2}$. The upper bound for $q_{i}(z)$ in \eqref{eq:qi} follows by combining Lemma A.1 (b) and (c) therein, which shows Lipschitz continuity of $M$ (in operator norm) and $\im [M(\ii \eta)]F\equiv 0$ for $\eta\in\R$, respectively. 

\begin{proof}[Proof of Theorem \ref{thm:main}]
	We only prove the complex case, and the real case follows from the exact same proof. Fix $z_{0}\in\D$ with $\absv{z_{0}}\leq 1-\tau$, and let $f_{0}\in C_{c}^{\infty}(\C)$ be a fixed cut-off function with $\lone_{\D}\leq f_{0}\leq \lone_{2\D}$ and $f\deq f_{0}(\frac{\cdot-z_{0}}{r})$. Since $\lone(\absv{\cdot-z_{0}}\leq r)\leq f$, Girko's formula \eqref{eq:Girko} and $\im\brkt{G_{z}(\ii\eta)}=-\ii\brkt{G_{z}(\ii\eta)}$ gives
	\begin{align}
		\E\caN_{z_{0},r}\leq& N\E\int_{\C}f(z)\dd\rho(z)
		=-\frac{N}{2\pi}\int_{\C}\Delta f(z)\int_{0}^{T} \E\im\brkt{G_{z}(\ii\eta)}\dd\eta\dd^{2} z+\frac{1}{4\pi}\int_{\C}\Delta f(z)\E\log\absv{\det (H_{z}-\ii T)}\dd^{2}z	\nonumber\\
		=&\frac{N}{2\pi}\int_{\C}\Delta f(z)\left(\int_{0}^{\eta_{0}}+\int_{\eta_{0}}^{\eta_{1}}+\int_{\eta_{1}}^{T}\right)\ii \E\brkt{G_{z}(\ii\eta)}\dd\eta\dd^{2}z+\frac{1}{4\pi}\int_{\C}\Delta f(z)\E\log\absv{\det (H_{z}-\ii T)}\dd^{2}z	\nonumber\\
		=&:I_{0}+I_{1}+I_{T}+I_{\infty},	\label{eq:split}
	\end{align}
	where we defined
	\begin{align}\label{eq:eta_defn}
		&\eta_{0}\deq N^{-1}\cdot(Nr^{2})^{D}, &  &\eta_{1}\deq 1, & &T\deq N^{D}(Nr^{2})^{-D}
	\end{align}
	for a fixed (large) constant $D>0$. Since the spectral parameters in $G_{z}(\ii\eta)$ are always $z$ and $\ii\eta$, we abbreviate $G\equiv G_{z}(\ii\eta)$ throughout the rest of the proof.
	
	The bound for the first term $I_{0}$ follows by taking $\epsilon=0$ and small enough $\xi>0$ in Lemma \ref{lem:gap};
	\beq\begin{aligned}\label{eq:I0}
		\absv{I_{0}}
		\leq& \frac{N}{2\pi}\int_{\C}\Absv{\Delta f(z)}\int_{0}^{\eta_{0}}\Norm{\brkt{G}}_{2}\dd\eta\dd^{2}z \\
		\lesssim& N^{\delta/2}\int_{0}^{N\eta_{0}}t^{-\xi/2}\sqrt{\absv{\log t}}\dd t
		\lesssim N^{\delta/2}(N\eta_{0})^{1-\xi}\lesssim N^{\delta/2}(Nr^{2})^{D/2},
	\end{aligned}\eeq
	where we used $\norm{\Delta f}_{L^{1}}\lesssim 1$.
	
	To handle $I_{1}$ and $I_{T}$ in \eqref{eq:split}, we use integration by parts, that is, for any $f\in C_{c}^{\infty}(\C)$ and $0<A<B<\infty$
	\beq\label{eq:IP}
	\E\int_{\C}\int_{A}^{B}\Delta f(z)\brkt{G}\dd\eta\dd^{2} z	
	=\int_{A}^{B}\int_{\C}f(z)\E\brkt{\Delta_{z}G_{z}(\ii\eta)}\dd^{2} z\dd\eta.
	\eeq
	The derivative $\Delta_{z}G$ can be calculated directly; for $F\in \C^{2N\times 2N}$ defined in \eqref{eq:qi_def},
	\beq\label{eq:DeltaG}
	\frac{1}{4}\Delta_{z}G	=\partial_{z}\ol{\partial}_{z}G=-\partial_{z}GFG=GF\adj GFG+GFGF\adj G.
	\eeq
	Now we can estimate $I_{T}$ in \eqref{eq:split}. Applying the trivial deterministic bound $\norm{G}\leq \eta^{-1}$ gives
	\beqs
	\E\absv{\brkt{GFGF\adj G+GF\adj GFG}}\leq 2\norm{F}^{2}\E\norm{G}^{3}\leq 2\eta^{-3},
	\eeqs
	which yields
	\beq\label{eq:IT}
	\absv{I_{T}}\lesssim Nr^{2}\int_{1}^{T}\eta^{-3}\dd\eta \leq Nr^{2}.
	\eeq
	
	The estimate for $\E\absv{\brkt{\Delta G}}$ when $\eta<1$ is much trickier than $\eta>1$. We first state the result here as a lemma, use it to conclude Theorem \ref{thm:main}, and then prove the lemma.
	\begin{lem}\label{lem:DelG}
		Let $X$ be a complex regular i.i.d. matrix and $\epsilon,\delta,\tau>0$ be fixed. Then the following holds uniformly over $\absv{z}\leq 1-\tau$ and $\eta\in(0,1)$;
		\beq\label{eq:DelG}
			\E\absv{\brkt{\Delta_{z}G_{z}(\ii\eta)}}\lesssim \frac{N^{1+\delta}}{N\eta}(N\eta\wedge 1)^{-\epsilon}.
		\eeq
	\end{lem}	
	Given Lemma \ref{lem:DelG}, we substitute \eqref{eq:DelG} into \eqref{eq:IP} with $A=\eta_{0}$ and $B=\eta_{1}$ so that
	\beq\begin{aligned}\label{eq:I1}
		\absv{I_{1}}&\leq N\int_{\eta_{0}}^{\eta_{1}}\int_{\C}\absv{f(z)}\absv{\E\brkt{\Delta G}}\dd^{2}z\dd\eta 
		\lesssim N^{\delta/2} Nr^{2}\int_{\eta_{0}}^{\eta_{1}}\frac{(N\eta\wedge 1)^{-\epsilon}}{N\eta}(N\dd\eta)\\
		&\lesssim N^{\delta/2} (Nr^{2})((N\eta_{0})^{-\epsilon}+\log N)= N^{\delta}(Nr^{2})^{1-D\epsilon} \leq N^{2\delta}(Nr^{2})^{(1-\gamma)},
	\end{aligned}\eeq
	where in the last step we took $\epsilon<\gamma/D$.

Finally, the last term $I_{\infty}$ is estimated using $\log (1+x^{2})\leq x^{2}$ and $\E\brkt{H_{z}^{2}}\lesssim 1$ as
\beq\label{eq:Iinf}
\absv{I_{\infty}}=\frac{1}{2}\Absv{\int_{\C}\Delta f(z)\E\Tr \log (1+T^{-2}H_{z}^{2})\dd^{2} z}\leq \frac{N}{T^{2}}\int_{\C}\absv{\Delta f(z)} \E\brkt{H_{z}^{2}}\dd^{2} z
\lesssim T^{-2}N\lesssim Nr^{2}.
\eeq
Combining \eqref{eq:I0}, \eqref{eq:IT}, \eqref{eq:I1}, and \eqref{eq:Iinf}, we have proved
\beq
\E\caN_{z_{0},r}\lesssim N^{1+2\delta}(Nr^{2})^{1-\gamma}.
\eeq
This completes the proof of Theorem \ref{thm:main} modulo Lemma \ref{lem:DelG}.
\end{proof}

\begin{proof}[Proof of Lemma \ref{lem:DelG}]
	We start with expressing the trace of the right-hand side of \eqref{eq:DeltaG} in terms of the eigenvalues and eigenvectors of $H_{z}$. First we write
	\beq\label{eq:DeltaG_1}
	\brkt{GF\adj GFG}=-\ii\frac{\eta}{2}\left\langle \frac{1}{\eta^{2}+(X+A-z)\adj(X+A-z)}\frac{\eta^{2}-(X+A-z)(X+A-z)\adj}{(\eta^{2}+(X+A-z)(X+A-z)\adj)^{2}}\right\rangle,
	\eeq
	and the second term of \eqref{eq:DeltaG}, $\brkt{GF\adj GFG}$, is given by the same quantity as \eqref{eq:DeltaG_1} with roles of $X$ and $X\adj$ interchanged. Here we used the Schur complement form of $G$, that is,
	\beq\label{eq:G_Schur}
		G=G(\ii\eta)=\begin{pmatrix}
			\dfrac{\ii\eta}{\eta^{2}+(X+A-z)(X+A-z)\adj} & \dfrac{1}{\eta^{2}+(X+A-z)(X+A-z)\adj}(X+A-z)\\
			(X+A-z)\adj\dfrac{1}{\eta^{2}+(X+A-z)(X+A-z)\adj}& \dfrac{\ii\eta}{\eta^{2}+(X+A-z)\adj(X+A-z)}
		\end{pmatrix}.
	\eeq
	Recalling the definition of $\bsu_{i}$ and $\bsv_{i}$ and \eqref{eq:DeltaG_1}, we may further write 
	\beq\begin{aligned}\label{eq:Holder_prep_1}
		\absv{\brkt{GF\adj GFG}} &\leq\frac{1}{2N}\sum_{i,j\in\bbrktt{N}}\frac{\eta}{\eta^{2}+\lambda_{i}^{2}}\frac{\absv{\eta^{2}-\lambda_{j}^{2}}}{(\eta^{2}+\lambda_{j}^{2})^{2}}\absv{\brkt{\bsu_{j},\bsv_{i}}}^{2}	\\
		&\leq \frac{1}{2N\eta}\left(\sum_{i,j\in\bbrktt{cN}}+\sum_{i\in\bbrktt{cN,N},j\in\bbrktt{N}}+\sum_{i\in\bbrktt{N},j\in\bbrktt{cN,N}}\right)\frac{\eta}{\eta^{2}+\lambda_{i}^{2}}\frac{\eta}{\eta^{2}+\lambda_{j}^{2}}\absv{\brkt{\bsu_{j},\bsv_{i}}}^{2}
		\\
		&\leq\frac{1}{2N\eta}\sum_{i,j\in\bbrktt{cN}}\frac{\eta^{2}}{(\eta^{2}+\lambda_{i}^{2})(\eta^{2}+\lambda_{j}^{2})}\absv{\brkt{\bsu_{j},\bsv_{i}}}^{2}+\frac{1}{\eta^{2}+\lambda_{\lceil cN\rceil}^{2}}\frac{1}{N}\sum_{i}\frac{\eta}{\eta^{2}+\lambda_{i}^{2}},
	\end{aligned}\eeq
	where the constant $c>0$ is from Proposition \ref{prop:overlap} and in the last inequality we used 
	\beqs
		\sum_{i\in\bbrktt{N},j\in\bbrktt{cN}}\frac{\eta}{\eta^{2}+\lambda_{i}^{2}}\frac{\eta}{\eta^{2}+\lambda_{j}^{2}}\absv{\brkt{\bsu_{j},\bsv_{i}}}^{2}\leq \frac{\eta}{\eta^{2}+\lambda_{\lceil cN\rceil}^{2}}\sum_{i\in\bbrktt{N}}\frac{\eta}{\eta^{2}+\lambda_{i}^{2}}\left(\sum_{j\in\bbrktt{N}}\absv{\brkt{\bsu_{j},\bsv_{i}}}^{2}\right),
	\eeqs
	and that $\bsu_{j}$'s are orthogonal to perform the $j$-summation. 
	For the first term on the rightmost side of \eqref{eq:Holder_prep_1}, we recall the deterministic function $q_{i}(z)$ from Proposition \ref{prop:overlap} and write
	\beq\begin{aligned}
		&\frac{1}{2N\eta}\sum_{i,j\in\bbrktt{cN}}\frac{\eta^{2}}{(\eta^{2}+\lambda_{i}^{2})(\eta^{2}+\lambda_{j}^{2})}\absv{\brkt{\bsu_{j},\bsv_{i}}}^{2}	\\
		\lesssim& \frac{1}{\eta}\brkt{\im G}^{2} \left(N\sup_{i,j\in\bbrktt{cN}}\absv{\brkt{\bsu_{j},\bsv_{i}}-\delta_{ij}q_{i}(z)}^{2}\right)+\frac{1}{N\eta}\sum_{i\in\bbrktt{cN}}\frac{\eta^{2}}{(\eta^{2}+\lambda_{i}^{2})^{2}}\frac{i^{2}}{N^{2}} \label{eq:Holder_prep_2}.
	\end{aligned}\eeq 
	Plugging in \eqref{eq:Holder_prep_2} to \eqref{eq:Holder_prep_1} and using the same set of inequalities to the second term of \eqref{eq:DeltaG}, we find that
	\begin{align}
		\absv{\brkt{\Delta_{z}G}}\lesssim &\frac{1}{\eta}\brkt{\im G}^{2}\left(N\sup_{i,j\in\bbrktt{cN}}\absv{\brkt{\bsu_{j},\bsv_{i}}-\delta_{ij}q_{i}(z)}^{2}\right)	\nonumber\\
		&+\frac{1}{N\eta}\sum_{i\in\bbrktt{cN}}\frac{\eta^{2}}{(\eta^{2}+\lambda_{i}^{2})^{2}}\frac{i^{2}}{N^{2}}+\frac{1}{\eta^{2}+\lambda_{\lceil cN\rceil}^{2}}\brkt{\im G}.\label{eq:Holder_prep}
	\end{align}

	Next, we estimate the expectation of the right-hand side of \eqref{eq:Holder_prep} for $\eta\in(0,1)$. For the first term in \eqref{eq:Holder_prep}, we apply H\"{o}lder's inequality and $\brkt{G}=\ii\im\brkt{G}$ to get for all $\eta\in(0,1)$ that
	\begin{multline}\label{eq:Holder}
		\frac{1}{\eta}\E\brkt{\im G}^{2}\left(N\sup_{i,j\in\bbrktt{cN}}\absv{\brkt{\bsu_{j},\bsv_{i}}-\delta_{ij}q_{i}(z)}^{2}\right)	\\\leq \frac{1}{\eta}\norm{\brkt{G}^{2}}_{1+c'\epsilon}\Norm{\sup_{i,j\in\bbrktt{cN}}N\absv{\brkt{\bsu_{j},\bsv_{i}}-\delta_{ij}q_{i}(z)}^{2}}_{1+1/(c'\epsilon)},
	\end{multline}
	where $c'>0$ is a small constant, say $c'=1/100$. Then we use Lemma \ref{lem:gap} and Proposition \ref{prop:overlap} with a union bound, respectively, to the first and second factors on the right-hand side of \eqref{eq:Holder} to get
	\beqs\begin{aligned}
		\norm{\absv{\brkt{G}}^{2}}_{1+c'\epsilon}=\norm{\brkt{G}}^{2}_{2(1+c'\epsilon)}\lesssim \left(1+N^{\delta/3}(N\eta)^{-2c'\epsilon}(\absv{\log \eta}+\log N)\right)^{1/(1+c'\epsilon)}\lesssim N^{\delta/2}(N\eta\wedge 1)^{-2c'\epsilon}, \nonumber\\
		\norm{N\sup_{i,j}\absv{\brkt{\bsu_{j},\bsv_{i}}-\delta_{ij}q_{i}(z)}^{2}}_{1+1/(c'\epsilon)}\leq \left(\E \sup_{i,j\in\bbrktt{cN}} (N\absv{\brkt{\bsu_{j},\bsv_{i}}-\delta_{ij}q_{i}(z)}^{2})^{(1+1/(c'\epsilon))}\right)^{c'\epsilon}\lesssim N^{\delta/2}.\label{eq:Holder_1}
	\end{aligned}\eeqs
	Thus we conclude
	\beq\label{eq:Holder_2}
		\frac{1}{\eta}\E\brkt{\im G}^{2}\left(N\sup_{i,j\in\bbrktt{cN}}\absv{\brkt{\bsu_{j},\bsv_{i}}}^{2}\right)\lesssim \frac{1}{\eta}N^{\delta}(1+(N\eta)^{-2c'\epsilon})=\frac{N^{1+\delta}}{N\eta}(N\eta\wedge 1)^{-2c'\epsilon}.
	\eeq
	
	To handle the expectation of the second term in \eqref{eq:Holder_prep}, recall the definition of the event $\Xi_{z}(\epsilon,\xi)$ from \eqref{eq:rigid}. On the event $\Xi\deq \Xi_{z}(c'\delta/2,c'\delta/2)$ we have
	\beq
		\lambda_{i}\geq N^{-c'\delta/2}\frac{i}{N} \qquad \forall i\in\bbrktt{N^{c'\delta},N},
	\eeq
	which follows from \eqref{eq:rigid} by putting in $\eta=N^{-c'\delta/2}(i/N)$. We divide the second term of \eqref{eq:Holder_prep} as
	\beqs\begin{aligned}
		&\sum_{i\in\bbrktt{cN}}\frac{\eta^{2}}{(\eta^{2}+\lambda_{i}^{2})^{2}}\frac{i^{2}}{N^{2}}\leq \left(\sum_{i\in\bbrktt{N^{c'\delta}}}+\lone_{\Xi^{c}}\sum_{i\in\bbrktt{cN}}+\lone_{\Xi}\sum_{i\in\bbrktt{N^{c'\delta},cN}}\right)\frac{\eta^{2}}{(\eta^{2}+\lambda_{i}^{2})^{2}}\frac{i^{2}}{N^{2}}	\\
		\leq&(N^{2c'\delta}+N^{2}\lone_{\Xi^{c}})\brkt{\im G}^{2}+ N^{2c'\delta}\sum_{i\in\bbrktt{N}}\frac{\eta^{2}\cdot (N^{-1}i)^{2}}{(\eta^{2}+(N^{-1}i)^{2})^{2}}
		\lesssim (N^{2c'\delta}+N^{2}\lone_{\Xi^{c}})\brkt{\im G}^{2}+ N^{1+2c'\delta}\eta.
		\end{aligned}\eeqs
	Then we follow the same argument as in \eqref{eq:Holder} to conclude for any fixed $D>0$ that
	\beq\label{eq:Holder_3}\begin{aligned}
		\frac{1}{N\eta}\E \sum_{i\in\bbrktt{cN}}\frac{\eta^{2}}{(\eta^{2}+\lambda_{i}^{2})^{2}}\frac{i^{2}}{N^{2}}&\leq \frac{1}{N\eta}\Norm{N^{2c'\delta}+N^{2}\lone_{\Xi^{c}}}_{1+1/(c'\epsilon)}\Norm{\brkt{G}^{2}}_{1+c'\epsilon}+\frac{N^{2c'\delta}}{\eta}\\
		&\lesssim \frac{N^{2c'\delta}}{\eta} +\frac{N^{-D+\delta/2}}{\eta}(N\eta\wedge 1)^{-2c'\epsilon}\leq \frac{N^{1+\delta}}{N\eta}(N\eta\wedge 1)^{-\epsilon}.
	\end{aligned}\eeq
	
	For the expectation of the last term of \eqref{eq:Holder_prep}, we again use the same event $\Xi=\Xi_{z}(c'\delta/2,c'\delta/2)$ but we only require that $\lambda_{\lceil cN\rceil}\geq cN^{-c'\delta/2}$ on the event. Then we write
	\beq
		\E\lone_{\Xi}\frac{1}{\eta^{2}+\lambda_{\lceil cN\rceil}^{2}}\brkt{\im G}\lesssim N^{c'\delta}\E\brkt{\im G}\leq N^{c'\delta}\norm{\brkt{G}}_{2}\lesssim N^{\delta}(1+(N\eta)^{-\epsilon})
		\leq \frac{N^{1+\delta}}{N\eta}(N\eta\wedge 1)^{-\epsilon}.
	\eeq
	The complementary event $\Xi^{c}$ can be dealt with the same reasoning as in \eqref{eq:Holder_3}. We thus conclude
	\beq\label{eq:Holder_result}
		\E\absv{\brkt{\Delta_{z}G}}\lesssim N^{1+\delta}(N\eta\wedge 1)^{-1-\epsilon}
	\eeq
	as desired.
\end{proof}

\begin{rem}\label{rem:main_1}
	As easily seen from the proof, the factor of $N^{\delta}$ in Theorem \ref{thm:main} is due to the same factors in Lemma \ref{lem:gap} and Proposition \ref{prop:overlap}. In the proof of Lemma \ref{lem:gap}, we use the local laws for $H_{z}$, Lemma \ref{lem:ll}, to show that $\lambda_{1}^{z}$ determines the size of $\brkt{G_{z}}$ up to a factor of $N^{\delta}$. If we can show that the random variable $\absv{\{i:\lambda_{i}^{z}\leq K/N\}}$ has finite high-moment for any fixed $K$, then we may use this as an alternative input and remove $N^{\delta}$ from Lemma \ref{lem:gap}. Likewise, Proposition \ref{prop:overlap} uses the two-resolvent local laws in \cite[Theorem 4.4]{Cipolloni-Erdos-Henheik-Schroder2023arXiv} that is designed for mesoscopic scales. If $h$ has a better decay rate, we expect Proposition \ref{prop:overlap} to be true without the factor of $N^{\delta}$ on the right-hand side of \eqref{eq:sing_ovlp}, which would reduce $N^{\delta}$ in Theorem \ref{thm:main} to a logarithmic correction. 
	
	On the other hand, changing the power $(1-\gamma)$ to the optimal first power of $Nr^{2}$ seems harder. It is due to a H\"{o}lder inequality in \eqref{eq:Holder}, whose purpose is to separate the singular values and the overlaps since we do not know how to estimate their joint distribution effectively.
\end{rem}

\begin{proof}[Proof of Corollary \ref{cor:resolv}]	
	We first prove the complex case and later show how to modify the proof for the real case. Fix $z\in\caB_{\tau}$. We decompose the special test function $w\mapsto 1/w$ into three parts by inserting cutoff functions as follows.
	\beq
		\frac{1}{w}=\frac{1}{w}(1-f_{0}(w/\kappa))+\frac{1}{w}f_{0}({\Cb N^{1/2}\nc}w)+\frac{1}{w}f_{0}(w/\kappa)(1-f_{0}({\Cb N^{1/2}\nc}w))=:f_{1}(w)+f_{2}(w)+f_{3}(w),
	\eeq
	{\Cb where $\kappa>0$ is such that $D(z,2\kappa)\subset B_{\tau/2}$} and the cutoff function $f_{0}$ was defined in the proof of Theorem~\ref{thm:main}. Thus we may write
	\beq\label{eq:cutoff_cor}
	\E\absv{\brkt{(X-z)^{-1}}}^{2-\delta_{1}}=\E\Absv{\int_{\C}(f_{1}(w-z)+f_{2}(w-z)+f_{3}(w-z))\dd\rho(w)}^{2-\delta_{1}}
	\eeq
	where $\rho$ is the spectral distribution of $X$ defined in \eqref{eq:rho}.
	
	The first integral is the easiest as $\absv{f_{1}(w)}\leq \kappa^{-1}$ for all $w\in\C$, so that
	\beq\label{eq:f1_conc}
	\E\Absv{\int_{\C}f_{1}(w-z)\dd\rho(w)}^{2-\delta_{1}}\leq \kappa^{-2+\delta_{1}}.
	\eeq
	{\Cb For the second integral, we use Jensen's inequality (recall $\delta_{1}\in(0,1]$) and $\supp f_{0}\subset D(0,2)$ to write
	\beq\label{eq:resolv_prf}
	\Absv{\int_{\C}f_{2}(w-z)\dd\rho(w)}^{2-\delta_{1}}\leq \rho(D(z,2N^{-1/2}))^{1-\delta_{1}}\int_{\C}\frac{1}{\absv{w-z}^{2-\delta_{1}}}\lone_{D(z,2N^{-1/2})}(w)\dd\rho(w).
	\eeq
	We then consider the following dyadic decomposition of $D(z,2N^{-1/2})$ according to $\absv{w-z}$ ;
	\beq\label{eq:cor_dyadic_decomp}
		D(z,2N^{-1/2})=\bigcup_{k=0}^{\infty}(D(z,2^{-k+1}N^{-1/2})\setminus D(z,2^{-k}N^{-1/2}))=:\bigcup_{k=0}^{\infty}D_{k}.
	\eeq
	Applying Theorem \ref{thm:main} (i) to each domain $D_{k}$ with 
	some $\delta<\delta_{1}/2$ and $\gamma<\delta_{1}/2$, we have
	\beq\begin{aligned}\label{eq:cor_dyadic}
		\E\int_{D_{k}}\frac{1}{\absv{w-z}^{2-\delta_{1}}}\dd\rho(w)
		\leq &(2^{k}N^{1/2})^{2-\delta_{1}}\E \int_{D_{k}}\dd\rho(w)	
		\leq N^{-1}(2^{k}N^{1/2})^{2-\delta_{1}}\E\caN_{z,N^{-1/2}2^{-k+1}}\\
		\leq &CN^{\delta-\delta_{1}/2}2^{k(2-\delta_{1})-2(k-1)(1-\gamma)}
		\leq CN^{\delta-\delta_{1}/2}2^{k(2\gamma-\delta_{1})}.
	\end{aligned}\eeq
	Therefore, summing over $k$ and using $\rho(D(z,2N^{-1/2})) \le \rho(\C) =1$, we obtain
	\beq\label{eq:f2_conc}
		\E\Absv{\int_{\C}f_{2}(w-z)\dd\rho(w)}^{2-\delta_{1}}\leq \sum_{k=0}^{\infty}\E\int_{D_{k}}\frac{1}{\absv{w-z}^{2-\delta_{1}}}\dd\rho(w)\leq CN^{\delta-\delta_{1}/2}\leq C.
	\eeq}
	
	{\Cb The last integral in \eqref{eq:cutoff_cor} can be dealt with the local law for $X+A$. This is the natural extension of the local circular law to $X+A$ and it asserts that for all (possibly $N$-dependent) $C_{c}^{2}$ test function $g$ supported in $\caB_{\tau/2}$ and fixed positive $\epsilon$ and $D$ we have
	\beq\label{eq:cor_ll}
	\P\left[\Absv{\int_{\C}g(w)\dd\rho(w)-\int_{\C}g(w)\dd\rho_{a+x}(w)}\geq
	\frac{N^{\epsilon}}{N}\norm{\Delta g}_{L^{1}}\right]\lesssim N^{-D}.
	\eeq}
	We omit its proof as it is standard, given the optimal local law \cite[Theorem 2.6]{Cipolloni-Erdos-Henheik-Schroder2023arXiv} as an input, following the analogous argument in the proof of~\cite[Theorem 2.5]{Alt-Erdos-Kruger2018}. {\Cb We apply \eqref{eq:cor_ll} to $g(w)=f_{3}(w-z)$, which is supported on $\caB_{\tau/2}$ by definition. We denote the `good' event in \eqref{eq:cor_ll} for this choice of $g$ by $\Xi_\epsilon$, i.e.
	\beq
		\Xi_{\epsilon}\deq\left[\Absv{\int_{\C}f_{3}(w-z)\dd\rho(w)-\int_{\C}f_{3}(w-z)\dd\rho_{a+x}(w)}\leq
		\frac{N^{\epsilon}}{N}\norm{\Delta f_{3}}_{L^{1}}\right],
	\eeq
	so that $\P[\Xi_{\epsilon}^{c}]\lesssim N^{-D}$ for all fixed $D>0$. Now notice that
	\beq\label{eq:Deltaf3}
	\norm{\Delta f_{3}}_{L^{1}}\lesssim N^{1/2}.
	\eeq
	On the other hand, since $\rho_{a+x}$ has a bounded density in $\caB_{\tau}$ by \eqref{eq:DOS_lb}, we have
	\beq\label{eq:f3_center}
		\Absv{\int_{\C}f_{3}(w-z)\dd\rho_{a+x}}\lesssim\int_{D(0,\kappa)}\frac{1}{\absv{z}}\dd^{2}z\lesssim 1.
	\eeq
	Hence on the event $\Xi_{\epsilon}$ we have
	\beq\label{eq:f3_1}
		\E\lone_{\Xi_{\epsilon}}\Absv{\int_{\C}f_{3}(w-z)\dd\rho_{a+x}(w)}^{2-\delta_{1}}\leq 1+ \left(\frac{N^{\epsilon}}{N^{1/2}}\right)^{2-\delta_{1}}\lesssim 1.
	\eeq
	On the complementary event $\Xi_{\epsilon}^{c}$, we recall  $\absv{f_{3}(w)}\leq N^{1/2}$ from the definition of $f_{3}$, which gives
	\beq\label{eq:f3_2}
		\E\lone_{\Xi_{\epsilon}^{c}}\Absv{\int_{\C}f_{3}(w-z)\dd\rho(w)}^{2-\delta_{1}}\leq \P[\Xi_{\epsilon}^{c}]N^{1-\delta_{1}/2}\lesssim N^{1-D}.
	\eeq
	Combining \eqref{eq:f3_1} and \eqref{eq:f3_2} we arrive at
	\beq\label{eq:f3_conc}
		\E\Absv{\int_{\C}f_{3}(w-z)\dd\rho(w)}^{2-\delta_{1}}\lesssim 1.
	\eeq
	Plugging in \eqref{eq:f1_conc}, \eqref{eq:f2_conc}, and \eqref{eq:f3_conc} to \eqref{eq:cutoff_cor} concludes the proof of Corollary \ref{cor:resolv} in the complex case.
	
	Finally, we show how to modify the proof for real $X$. The estimates for $f_{1}$ and $f_{3}$ are completely analogous to the complex case, and their contributions are both bounded by a constant. Henceforth we focus on estimating the contribution of $f_{2}$.  We consider the cases $y=\lambda_{1}(\Im[A-z])\geq N^{-1}$
	and $y\le N^{-1}$ separately.
	
	When $y\geq N^{-1}$, 
	we define $\wt{y}:=\min(N^{-1/2},y/4)$ and further divide $f_{2}$ into two parts, 
	\beq
		f_{2}(w)=\frac{1}{w}f_{0}(\wt{y}^{-1}w)+\frac{1}{w}f_{0}(N^{1/2}w)(1-f_{0}(\wt{y}^{-1}w))=:f_{21}(w)+f_{22}(w).
	\eeq
	The contribution of $f_{21}$ can be estimated in the exact same fashion as  $f_{2}$ for the complex case, using a dyadic decomposition. The only difference is that we use radii $2^{-k+1}\wt{y}$ instead of $2^{-k+1}N^{-1/2}$, and apply Theorem \ref{thm:main} (ii) in place of (i). As a result we obtain
	\beq\label{eq:cor_real_1}
		\E\Absv{\int_{\C}f_{21}(w-z)\dd\rho(w)}^{2-\delta_{1}}\leq C\frac{N^{\delta-\delta_{1}/2}}{y}(N\wt{y}^{2})^{\delta_{1}/2-\gamma}\leq C\frac{N^{\delta-\delta_{1}/2}}{y}.
	\eeq
	For $f_{22}$, we use the crude bound $\absv{f_{22}(w)}\leq 2\wt{y}^{-1}f_{0}(N^{1/2}(w))$ to get
	\begin{multline}\label{eq:cor_real_2}
		\E\Absv{\int_{\C}f_{22}(w-z)\dd\rho(w)}^{2-\delta_{1}}\leq C\wt{y}^{\delta_{1}-2}\E\int_{\C}f_{0}(N^{1/2}(w-z))\dd\rho(w)	
		\leq CN^{-1+\delta}\wt{y}^{\delta_{1}-2}	\\
		\leq CN^{\delta-\delta_{1}/2}+CN^{-1+\delta}y^{\delta_{1}-2}.
	\end{multline}
	where in the first line we replaced $ \dd\rho(w)$ with the bounded density $\dd\rho_{a+x}(w)$
	by applying
	 \eqref{eq:cor_ll} to $g(w)=f_{0}(N^{1/2}(w-z))$, using that 
	$\norm{\Delta g}_{L^{1}}=\norm{ \Delta f_{0}}_{L^{1}} =O(1)$ for this choice of $g$. Adding \eqref{eq:cor_real_1} and \eqref{eq:cor_real_2} proves 
	\beq
		\E\Absv{\int_{\C}f_{2}(w-z)\dd\rho(w)}^{2-\delta_{1}}\leq C+C\frac{N^{\delta}}{y}\left(N^{-\delta_{1}/2}+\frac{1}{Ny^{1-\delta_{1}}}\right)\leq C+\frac{N^{-\delta_{1}/2+\delta}}{y},
	\eeq
	where we used $y\geq N^{-1}$ in the last inequality.
	
	Now it only remains to consider the case $y\leq1/N$ and prove
	\beq\label{eq:cor_real_3_0}
		\E\Absv{\int_{\C}f_{2}(w-z)\dd\rho(w)}^{2-\delta_{1}}\leq C\frac{N^{\delta}}{y^{1-\delta_{1}/2+\delta}}.
	\eeq
	Here we use that $\absv{f_{2}(w-z)}\leq f_{0}(N^{1/2}(w-z))/\lambda_{1}^{z}$, which gives
	\beq\label{eq:cor_real_3}
		\E\Absv{\int_{\C}f_{2}(w-z)\dd\rho(w)}^{2-\delta_{1}}\leq \E (N\lambda_{1}^{z})^{\delta_{1}-2}\Absv{N\int_{\C}f_{0}(N^{1/2}(w-z))\dd\rho(w)}^{2-\delta_{1}}.
	\eeq
	Then we apply H\"{o}lder inequality in \eqref{eq:cor_real_3} to obtain for all fixed $\epsilon''\in(0,\delta_{1})$ that
	\beq\label{eq:Holder_resolv}
	\E\Absv{\int_{\C}f_{2}(w-z)\dd\rho(w)}^{2-\delta_1}\lesssim N^{\delta/2}(\E(N\lambda_{1}^{z})^{-2+\epsilon''})^{(2-\delta_{1})/(2-\epsilon'')},
	\eeq
	where we used local laws to the second factor in \eqref{eq:cor_real_3} to bound its high moment by $N^{\delta/2}$, as in \eqref{eq:cor_real_2}. Now by  \eqref{eq:s_real_weak} in
	Theorem \ref{thm:s_real_weak} and an integration by parts we can compute the right-hand side of \eqref{eq:Holder_resolv} as
	\beqs
		\E(N\lambda_{1}^{z})^{-2+\epsilon''}
		=(2-\epsilon'')\int_{0}^{\infty}t^{-3+\epsilon''}\P[N\lambda_{1}^{z}\leq t]\dd t
		\lesssim 1+ \frac{N^{\delta/2}}{y}\int_{0}^{1}t^{-1+\epsilon''}\log t\dd t\lesssim \frac{N^{\delta/2}}{y}.
		\eeqs
	Thus we have
	\beqs
	\E \Absv{\int_{\C}f_{2}(w-z)\dd\rho(w)}^{2-\delta_1}\lesssim N^{\delta}y^{-(2-\delta_{1})/(2-\epsilon'')},
	\eeqs
	and choosing suitable $\epsilon''\in(0,\delta_{1})$ concludes \eqref{eq:cor_real_3_0}. This completes the proof of Corollary \ref{cor:resolv}.
}
\end{proof}

\begin{proof}[Proof of Lemma \ref{lem:gap}]
	The only input for the proof, other than Lemma \ref{lem:ll}, is Theorem \ref{thm:s_weak}.(i) for the complex case and Theorem \ref{thm:s_real_weak} for the real case. We restrict ourselves to the complex case, since the real case follows directly by replacing the input accordingly.
	
	First, we separate the contribution of the singular values $\lambda_{i}^{z}$ below $N^{-1+\delta/2}$ from the rest;
	\beq\label{eq:s_real_G_cutoff_0}
	\E\absv{\brkt{G}}^{2+\epsilon}\lesssim
	\E n_{0}^{2+\epsilon}\left(\frac{N\eta}{(N\lambda_{1})^{2}+(N\eta)^{2}}\right)^{2+\epsilon}+\E\left(\frac{1}{N}\sum_{i:\lambda_{i}^{z}>N^{-1+\delta/2}}\frac{\eta}{\lambda_{i}^{2}+\eta^{2}}\right)^{2+\epsilon},
	\eeq
	where we denoted by $n_{0}$ the number of singular values below $N^{-1+\delta/2}$;
	\beq
		n_{0}\deq \absv{\{i\in\bbrktt{N}:\lambda_{i}\leq N^{-1+\delta/2}\}}.
	\eeq
	By \eqref{eq:rigid}, the random variable $n_{0}$ satisfies $\norm{n_{0}}_{p}\lesssim N^{\delta}$ for any fixed $p\in\N$. Thus applying H\"{o}lder's inequality to the first term of \eqref{eq:s_real_G_cutoff_0} gives that for each $\xi>0$
	\beq\label{eq:s_real_G_cutoff}
		\E\absv{\brkt{G}}^{2+\epsilon}\lesssim N^{\delta}\Norm{\frac{N\eta}{(N\lambda_{1})^{2}+(N\eta)^{2}}}_{2+\epsilon+\xi}^{2+\epsilon}+\E\brkt{\im G_{z}(\ii N^{-1+\delta/2})}^{2+\epsilon}.
	\eeq
	Denoting $F_{1}(\cdot)=\P[N\lambda_{1}\leq\cdot]$, the first term of \eqref{eq:s_real_G_cutoff} is estimated using \eqref{eq:s_real} as
	\beq\begin{aligned}\label{eq:s_real_G_cutoff_1}
		&\E\left(\frac{N\eta}{(N\lambda_{1})^{2}+(N\eta)^{2}}\right)^{2+\epsilon+\xi}
		\lesssim\int_{0}^{\infty}\int_{x}^{\infty}\frac{(N\eta)^{2+\epsilon+\xi}t}{(t^{2}+(N\eta)^{2})^{3+\epsilon+\xi}}\dd t \dd F_{1}(x)	\\
		=&\int_{0}^{\infty}\frac{(N\eta)^{2+\epsilon+\xi}t}{(t^{2}+(N\eta)^{2})^{3+\epsilon+\xi}}F_{1}(t)\dd t
		\lesssim \int_{0}^{\infty}\frac{(N\eta)^{2+\epsilon+\xi}t^{3}}{(t^{2}+(N\eta)^{2})^{3+\epsilon+\xi}}\left(\absv{\log t}+\log N\right)\dd t	\\
		=&(N\eta)^{-\epsilon-\xi}\int_{0}^{\infty}\frac{s^{3}}{(s^{2}+1)^{3+\epsilon+\xi}}\left(\absv{\log (N\eta s)}+\log N\right)\dd s	
		\lesssim (N\eta)^{-\epsilon-\xi}\left(\absv{\log \eta}+\log N\right).
	\end{aligned}\eeq
	For the second expectation in \eqref{eq:s_real_G_cutoff}, we directly use \eqref{eq:ll} and $\norm{M}\lesssim 1$ from \cite[Appendix A]{Cipolloni-Erdos-Henheik-Schroder2023arXiv} to get
	\beq\label{eq:s_real_G_cutoff_2}
		\E\brkt{\im G_{z}(\ii N^{-1+\delta/2})}^{2+\epsilon}\lesssim 1+\frac{1}{N^{\delta/2}}\lesssim 1.
	\eeq
	Substituting \eqref{eq:s_real_G_cutoff_1} and \eqref{eq:s_real_G_cutoff_2} into \eqref{eq:s_real_G_cutoff} completes the proof of Lemma \ref{lem:gap} for the complex case.
\end{proof}

\section{Upper bound for the overlap: Proof of Theorems \ref{thm:overlap} and \ref{thm:overlap_strong}}\label{sec:O_proof}
We start with the proof of Theorem \ref{thm:overlap_strong} for it motivates that of Theorem \ref{thm:overlap} while being much shorter.
\begin{proof}[Proof of Theorem \ref{thm:overlap_strong}]
Recall the inequalities in \eqref{eq:ovlp_s}. In fact, they follow from more general, elementary deterministic inequalities (see e.g. \cite[Lemma 3.2]{Banks-Kulkarni-Mukherjee-Srivastava2021} for a proof)
\beq\begin{aligned}\label{eq:ovlp_triv}
	\pi\sum_{i:\sigma_{i}\in\caD}\caO_{ii}\leq& \liminf_{\epsilon\to0}\frac{\Absv{\{z\in\caD:\lambda_{1}^{z}\leq \epsilon}}{\epsilon^{2}},	\\
	2\sum_{i:\sigma_{i}\in\caI}\sqrt{\caO_{ii}}\leq &\liminf_{\epsilon\to 0}\frac{\Absv{\{x\in\caI:\lambda_{1}^{x}\leq\epsilon\}}}{\epsilon},
\end{aligned}\eeq
that hold true for any matrix with simple spectrum and any Borel sets $\caD\subset\C$ and $\caI\subset\R$. Simply taking the expectation of these two inequalities and using Fatou's lemma gives
\beq\label{eq:ovlp_triv1}\begin{aligned}
	\E\sum_{i:\sigma_{i}\in\caD}\caO_{ii}\leq \liminf_{\epsilon\to 0}\frac{1}{\pi\epsilon^{2}}\E\absv{\{z\in\caD:\lambda_{1}^{z}\leq\epsilon}
	=\frac{N^{2}}{\pi}\liminf_{\epsilon\to 0}\int_{\caD}\frac{\P[N\lambda_{1}^{z}\leq N\epsilon]}{(N\epsilon)^{2}}\dd^{2} z, \\
	\E\sum_{i:\sigma_{i}\in\caI}\sqrt{\caO_{ii}}\leq \liminf_{\epsilon\to 0}\frac{1}{2\epsilon}\E\absv{\{x\in\caI:\lambda_{1}^{x}\leq\epsilon\}}
	=\frac{N}{2}\liminf_{\epsilon\to 0}\int_{\caI}\frac{\P[N\lambda_{1}^{x}\leq N\epsilon]}{N\epsilon}\dd x.
\end{aligned}\eeq
Then plugging in the results of Theorem \ref{thm:s_strong} into \eqref{eq:ovlp_triv1} proves Theorem \ref{thm:overlap_strong}.
\end{proof}

One can immediately see that it is vital to have the exact rate $\P[N\lambda_{1}\leq s]\lesssim s^{2}$ or $s$ for the proof above. Any suboptimal factor in $s$, even $\absv{\log s}$ as in Theorems \ref{thm:s_weak} and \ref{thm:s_real_weak}, completely ruins the proof, due to the fact that the limit $\epsilon\to0$ in \eqref{eq:ovlp_triv} is not quantitative. Inspecting the original proof of \eqref{eq:ovlp_triv} in \cite{Banks-Kulkarni-Mukherjee-Srivastava2021}, one finds that in order for an inequality like \eqref{eq:ovlp_triv} to be true, $\epsilon$ has to be smaller than, say, (i) minimal gap between eigenvalues in $\caD$ and (ii) $1/\lambda_{1}^{z}$. Roughly, our proof of Theorem \ref{thm:overlap} shows that a suitable deterministic choice of $\epsilon$ is smaller than both of (i) and (ii) with high probability, and quantifies the limit in \eqref{eq:ovlp_triv} via contour integrals. 

\begin{proof}[Proof of Theorem \ref{thm:overlap}]
	First of all, notice that \eqref{eq:overlap_weak} follows immediately from \eqref{eq:overlap_strong} after an application of Markov's inequality, so we will focus on \eqref{eq:overlap_strong} and \eqref{eq:overlap_strong_R}.
	
	Next we prove that it suffices to assume {\Cb $\caD\subset[-4-\norm{A},4+\norm{A}]^{2}$.} Suppose that we have the result for squares in {\Cb $[-4-\norm{A},4+\norm{A}]^{2}$,} and take a general square $\caD$ with $\absv{\caD}\geq N^{-2K}$. Define 
	\beqs
		\caD_{0}\deq\begin{cases}
			\emptyset &\text{ if }\caD\cap{\Cb(-3-\norm{A},3+\norm{A})^{2}}=\emptyset, \\
			\caD\cap{\Cb[-4-\norm{A},4+\norm{A}]^{2}} & \text{ otherwise}.
		\end{cases}
	\eeqs 
	Then $\caD_{0}$ is either empty or a rectangle whose side lengths are at least $N^{-K}$. Since \eqref{eq:overlap_strong} and \eqref{eq:overlap_strong_R} are additive with respect to the domain, the conclusion is true for $\caD_{0}$ with an exceptional event $\Xi_{\caD_{0}}$ such that $\P[\Xi_{\caD_{0}}^{c}]\leq N^{-D}$. We then consider the event {\Cb$\Xi_{0}\deq[\norm{X+A}\leq 3+\norm{A}]$}, which satisfies $\P[\Xi_{0}^{c}]\lesssim N^{-D}$ for all fixed $D>0$ by, e.g. \cite[Theorem 3.1]{Pillai-Yin2014}. Defining $\Xi_{\caD}=\Xi_{\caD_{0}}\cap\Xi_{0}$ we have $\P[\Xi_{\caD}^{c}]\lesssim N^{-D}$. Since there is no eigenvalue in {\Cb$[-3-\norm{A},3+\norm{A}]^{c}\supset\caD\setminus\caD_{0}$} on the event $\Xi_{0}$, \eqref{eq:overlap_strong} and \eqref{eq:overlap_strong_R} hold true for the general square $\caD$.
	
	Thus from now on we assume {\Cb$\caD\subset[-4-\norm{A},4+\norm{A}]^{2}$.}  For a parameter $r<\sqrt{\absv{\caD}}$ to be chosen later, we define $\caC_{r}$ to be a partition of $\caD$ consisting of solid, open squares with side length $r$ so that $\absv{\caC_{r}}\sim \absv{\caD}r^{-2}$ where $\absv{\caC_{r}}$ denotes the number of elements in the partition. Now for this partition, we define the events
	\begin{align}\label{eq:O_Xi}
		\Xi_{1}\deq \bigcap_{C\in\caC_{r}}\left[\absv{\lambda_{1}^{z}}>\fra, \forall z\in \partial C\right],	\qquad 
		\Xi_{2}\deq \bigcap_{C\in\caC_{r}}\left[\absv{\{i:\sigma_{i}\in C\}}\leq 1\right],\AND \Xi\deq \Xi_{1}\cap\Xi_{2}
	\end{align}
	where $\fra>0$ is an $N$-dependent parameter that will also be chosen later. Under these notations, we estimate $\P[\Xi^{c}]$ and the expectation in \eqref{eq:overlap_strong} and \eqref{eq:overlap_strong_R} with the following lemma.
	\begin{lem}\label{lem:est}
		Let {\Cb $\caD\subset\C$} be a square and define $\Xi_{1}$,$\Xi_{2}$, and $\Xi$ as in \eqref{eq:O_Xi}.
		\begin{itemize}
			\item If $X$ is complex, we have
					\begin{gather}
				\E\lone_{\Xi}\sum_{i:\sigma_{i}\in \caD}\caO_{ii}\lesssim N(N\absv{\caD}) \absv{\log \fra}^{2},\label{eq:est1}\\
				\P[\Xi_{1}^{c}] \lesssim N^{2}r^{-1}\fra\absv{\log\fra}{\Cb\absv{\caD}}, \label{eq:est2}\\
				\P[\Xi_{2}^{c}]\lesssim N^{16/5}r^{6/5}\absv{\log r}^{2}{\Cb\absv{\caD}}.\label{eq:est3}
			\end{gather}
			uniformly over $0<\fra\leq r\leq\min((4N)^{-1},\sqrt{\absv{\caD}})$. 
			\item If $X$ is real, we have
			\begin{gather}
				\E\lone_{\Xi}\sum_{i:\sigma_{i}\in \caD}\caO_{ii}\lesssim N\frac{{\Cb(1+\norm{A}^{2})}N\absv{\caD}}{y}\absv{\log \fra}^{2},\label{eq:est1_R}\\
				\P[\Xi_{1}^{c}] \lesssim N^{2}r^{-1}\fra\frac{\absv{\log\fra}}{y}{\Cb\absv{\caD}}, \label{eq:est2_R}\\
				\P[\Xi_{2}^{c}]\lesssim N^{8/3}r^{2/3}y^{-2/3}{\Cb(\absv{\log r}+\absv{\log y})^{2}}{\Cb\absv{\caD}}.\label{eq:est3_R}
			\end{gather}
			uniformly over $0<\fra\leq r\leq \min\{(4N)^{-1},\sqrt{\absv{\caD}},y/2\}$, where we denoted $y\deq\min_{z\in\caD}\lambda_{1}(\Im[A-z])$.
		\end{itemize}
	\end{lem}
	
	We postpone the proof of Lemma \ref{lem:est} and first deduce the complex case, \eqref{eq:overlap_strong}, from the lemma. Writing $\fra=N^{-k_{1}}$ and $r=N^{-k_{2}}$ with $k_{1},k_{2}>1$, we will show that we can choose constant parameters $k_{1}$ and $k_{2}$ so that \eqref{eq:overlap_strong} is true and the right-hand sides of \eqref{eq:est2}, \eqref{eq:est3} are $O(N^{-D})$.
	
	First of all, taking $\fra=N^{-k_{1}}$ in \eqref{eq:est1} with constant $k_{1}$ immediately proves \eqref{eq:overlap_strong}. Then we plug in $\fra=N^{-k_{1}}$ and $r=N^{-k_{2}}$ into \eqref{eq:est2} and \eqref{eq:est3}, so that
	\beq
	\P[\Xi_{1}^{c}]\lesssim N^{k_{2}-k_{1}+2{\Cb +2\frK_{0}}}\log N \AND \P[\Xi_{2}^{c}]\lesssim N^{\frac{16-6k_{2}}{5}{\Cb +2\frK_{0}}}(\log N)^{2},
	\eeq
	{\Cb where we used $\absv{\caD}\leq (N^{\frK_{0}}+4)^{2}$}.
	Therefore choosing $k_{2}=K+D{\Cb+2\frK_{0}}+100$ and $k_{1}=k_{2}+D{\Cb +2\frK_{0}}+100$ proves $P[\Xi^{c}]\lesssim N^{-D}$. We omit the proof for the real case since it is completely analogous except that we use {\Cb$N^{-K}\leq y\leq 2N^{\frK_{0}}+4$}. This finishes the proof of Theorem \ref{thm:overlap} modulo Lemma \ref{lem:est}.
\end{proof}

\begin{proof}[Proof of Lemma \ref{lem:est}]
	We first prove the complex case, and start with the proof of \eqref{eq:est1}. The major problem is that we cannot pull the sum over $\{i:\sigma_{i}\in\caD\}$ out of the expectation, since the set of indices itself is random. To circumvent this, we use the partition $\caC_{r}$ and the fact that each $C\in\caC_{r}$ contains at most one eigenvalue on the event $\Xi\subset\Xi_{2}$, hence
	\begin{align}\label{eq:Xi2used}
		\E\lone_{\Xi}\sum_{i:\sigma_{i}\in \caD}\caO_{ii} =\E\lone_{\Xi}\sum_{C\in\caC_{r}}\sum_{i:\sigma_{i}\in C}\caO_{ii}=\E\lone_{\Xi}\sum_{C\in\caC_{r}}\lone(\exists\sigma_{i}\in C)\caO_{ii}
		=\sum_{C\in\caC_{r}}\E\lone_{\Xi}\lone(\exists\sigma_{i}\in C)\caO_{ii}.
	\end{align}
	Note that the deterministic sum over $C\in\caC_{r}$ in effect replaces the sum over $i$, and in the last equality of \eqref{eq:Xi2used} we interchanged it with the expectation. Furthermore, for each $C\in\caC_{r}$ we have 
	\beq
	\frac{1}{2\pi\ii}\oint_{\partial C}\frac{1}{X+A-z}\dd z
	=\frac{1}{2\pi\ii}\oint_{\partial C}\sum_{i}\frac{\dd z}{\sigma_{i}-z}\bsr_{i}\bsl_{i}\adj	=-\sum_{i:\sigma_{i}\in C}\bsr_{i}\bsl_{i}\adj,
	\eeq
	which is valid as $\P[\exists \sigma_{i}\in\partial C]=0$ and $X+A$ is simple almost surely. Recalling the definition of $\caO_{ij}$ from \eqref{eq:def_O}, this in turn implies
	\begin{align}\label{eq:contour}
		\frac{1}{4\pi^{2}}\lone_{\Xi}\Norm{\oint_{\partial C}\frac{1}{X+A-z}\dd z}^{2}
		=\lone_{\Xi}\Norm{\sum_{i:\sigma_{i}\in C}\bsr_{i}\bsl_{i}\adj}^{2}
		=\lone_{\Xi}\lone (\exists \sigma_{i}\in C)\caO_{ii},
	\end{align}
	where we used that there is at most one eigenvalue in each $C$ on the event $\Xi$. Since the equality in \eqref{eq:contour} is true for every $C\in\caC_{r}$, we take the expectation and plug it into \eqref{eq:Xi2used} to obtain
	\beq\label{eq:O_caC}
	\E\lone_{\Xi}\sum_{i:\sigma_{i}\in \caD}\caO_{ii} \leq \frac{1}{4\pi^{2}}\sum_{C\in \caC_{r}}\E\lone_{\Xi}\Norm{\oint_{\partial C}\frac{1}{(X+A-z)}\dd z}^{2}.
	\eeq
	
	We next estimate the contour integral for each $C$. First, we use $\norm{Y}^{2}=\norm{YY\adj}$ to write
	\beq\begin{aligned}\label{eq:O_C}
		&\E\lone_{\Xi}\Norm{\oint_{\partial C}\frac{1}{(X+A-z)}\dd z}^{2}	
		=\E\lone_{\Xi}\Norm{\oint_{\partial C}\oint_{\partial C}\frac{1}{(X+A-z)}\frac{1}{(X+A-w)\adj}\dd z\dd\ol{w}}\\
		\leq& \E\lone_{\Xi}\int_{\partial C}\int_{\partial C}\Norm{ \frac{1}{(X+A-z)}\frac{1}{(X+A-w)\adj}}\absv{\dd z}\absv{\dd w}	
		=\int_{\partial C}\int_{\partial C}\E\lone_{\Xi}\Norm{ \frac{1}{(X+A-z)}\frac{1}{(X+A-w)\adj}}\absv{\dd z}\absv{\dd w}\\
		\leq& 16r^{2}\sup_{z,w\in\partial C} \E\lone_{\Xi_{1}\cap\Xi_{2}}\Norm{\frac{1}{X+A-z}\frac{1}{(X+A-w)\adj}}
		\leq 16r^{2}\sup_{z\in\partial C}\E\lone_{\Xi_{1}}\Norm{\frac{1}{X+A-z}}^{2},
	\end{aligned}\eeq
	where in last line we used Cauchy-Schwarz inequality and also dropped $\lone_{\Xi_{2}}$. Since $\lambda_{1}^{z}\geq\fra$ on the event $\Xi_{1}$, for all $i\in\bbrktt{N}$ and $z\in\partial C$ we have from Theorem \ref{thm:s_weak} that
	\beq\label{eq:O_lambda}\begin{aligned}
		&\E\lone_{\Xi_{1}}\Norm{\frac{1}{X+A-z}}^{2}
		=N^{2}\E\lone_{\Xi_{1}}\frac{1}{(N\lambda_{1}^{z})^{2}}
		=2N^{2}\int_{N\fra}^{\infty}\frac{1}{s^{3}}\P[N\lambda_{1}^{z}\leq s]\dd s\\
		\leq& N^{2}+CN^{2}\int_{N\fra}^{1}\frac{1}{s}(\absv{\log s}+\log N)\dd s
		\leq N^{2}+CN^{2}\absv{\log (N\fra)}\left(\log N+\absv{\log (N\fra)}\right)
		\leq CN^{2}\absv{\log\fra}^{2},
	\end{aligned}\eeq
	where we used $\fra\leq (4N)^{-1}$ in the second line. Combining \eqref{eq:O_caC} -- \eqref{eq:O_lambda}, we conclude that
	\beq
	\E\lone_{\Xi}\sum_{i:\sigma_{i}\in \caD}\caO_{ii}
	\lesssim \absv{\caC_{r}}r^{2}N^{2}\absv{\log \fra}^{2}\sim N^{2}\absv{\caD}\absv{\log \fra}^{2}.
	\eeq
	This completes the proof of \eqref{eq:est1}.
	
	Next, we prove the second estimate \eqref{eq:est2}. We take a regular $\fra$-grid $\caL$ of points in $\bigcup_{C\in\caC_{r}}\partial C$ so that 
	\begin{align}
		\max_{C\in C_{r}}\max_{z\in \partial C}\min_{w\in \caL}\absv{z-w}\leq \fra \AND \absv{\caL}\sim \absv{\caC_{r}}r\fra^{-1}\sim \absv{\caD}r^{-1}\fra^{-1}.
	\end{align}
	We here remark that it is crucial to have the first negative power $\fra^{-1}$ on the right-hand side, not the second; this comes from the fact that $\caL$ is a grid on $\partial C$, not $C$. Then, on the event
	\beq
	\bigcap_{w\in\caL}[\lambda_{1}^{w}\geq 2\fra],
	\eeq
	for any $z\in \bigcup_{C\in \caC_{r}}\partial C$ we can find a point $w\in \caL$ with $\absv{z-w}\leq \fra$ so that
	\beq
	\lambda_{1}^{z}\geq \lambda_{1}^{w}-\norm{H_{z}-H_{w}}=\lambda_{1}^{w}-\absv{z-w}\geq \fra.
	\eeq
	Thus, by Theorem \ref{thm:s_weak}(i) with $k=1$ and $\fra\leq r\leq (4N)^{-1}$, we obtain
	\beq\label{eq:ovlp_prf1}
	\begin{aligned}
		\P[\Xi_{1}^{c}]=&\P\left[\bigcup_{C\in\caC_{r}}[\exists z\in \partial C \text{ such that }\lambda_{i}^{z}\leq \fra]\right]\leq \P\left[\bigcup_{w\in\caL}[\lambda_{1}^{w}\leq 2\fra]\right]	\\
	\lesssim& \absv{\caL} (N\fra)^{2}(\log N+\absv{\log\fra}) \lesssim N^{2}\absv{\caD}r^{-1}\fra \absv{\log\fra}.
	\end{aligned}\eeq
	This proves \eqref{eq:est2}.
	
	We finally prove \eqref{eq:est3}. By a union bound we trivially have
	\beq\label{eq:est3_1}
	\P[\Xi_{2}^{c}]=\P\left[\bigcup_{C\in\caC_{r}}[\absv{\{i:\sigma_{i}\in C\}}\geq 2]\right]\leq \absv{\caC_{r}}\max_{C\in\caC_{r}}\P[\absv{\{i:\sigma_{i}\in C\}}\geq 2].
	\eeq
	To bound the probability on the right-hand side, we fix a square $C\in\caC_{r}$, take $z_{C}$ to be the center of $C$, and label the eigenvalues so that $\absv{\sigma_{i}-z_{C}}$ increases in $i$. Then we have
	\beq
	\P[\absv{\{i:\sigma_{i}\in C\}}\geq 2]\leq \P[\absv{\sigma_{2}-z_{C}}\leq 2r]\leq \P[\absv{\sigma_{1}-z_{C}}\cdot\absv{\sigma_{2}-z_{C}}\leq 4r^{2}]\leq \P[\lambda_{1}^{z_{C}}\lambda_{2}^{z_{C}}\leq 4r^{2}],
	\eeq
	where the last inequality is due to Weyl;
	\beq
	\prod_{i=1}^{k}\absv{\sigma_{i}-z_{C}}\geq \prod_{i=1}^{k}\lambda_{i}^{z_{C}},\qquad \forall k\in\bbrktt{N}.
	\eeq
	Now for a threshold $\kappa>0$ to be optimized later, by Theorem \ref{thm:s_weak}.(i) with $k=1,2$ we have
	\beq\label{eq:ovlp_prf2}
	\begin{aligned}
		\P[\lambda_{1}^{z_{C}}\lambda_{2}^{z_{C}}\leq 4r^{2}]	
		=& \P[\lambda_{1}^{z_{C}}\lambda_{2}^{z_{C}}\leq 4r^{2}, \lambda_{2}^{z_{C}}\leq\kappa]+\P[\lambda_{1}^{z_{C}}\lambda_{2}^{z_{C}}\leq 4r^{2},\lambda_{2}^{z_{C}}\geq\kappa]	\\
		\leq&\P[\lambda_{2}^{z_{C}}\leq \kappa]+\P[\lambda_{1}^{z_{C}}\leq 4r^{2}/\kappa]	\\
		\lesssim& \left((N\kappa)^{8}+ N^{2}\kappa^{-2}r^{4}\right)\cdot\left(\log N+\absv{\log \kappa}+\absv{\log r}\right)^{2}.
	\end{aligned}
	\eeq
	Note that here we used that $s\in[0,1]$ in Theorem \ref{thm:s_weak} can be extended to $s\in[0,\infty)$ by increasing $C$ slightly. Taking the optimal choice $\kappa=N^{-3/5}r^{2/5}$, we have $\P[\lambda_{1}^{z_{C}}\lambda_{2}^{z_{C}}\leq 4r^{2}]\lesssim (Nr)^{16/5}\absv{\log r}^{2}$. Therefore we conclude from \eqref{eq:est3_1} that
	\beq
	\P[\Xi_{2}^{c}]\lesssim \absv{\caC_{r}}(Nr)^{16/5}\absv{\log r}^{2}\sim N^{16/5}\absv{\caD}r^{6/5}\absv{\log r}^{2},
	\eeq
	completing the proof of Lemma \ref{lem:est} in the complex case.
	
	We next show how to modify the proof for the real case. Firstly in \eqref{eq:O_lambda}, we use Theorem \ref{thm:s_real_weak} instead of Theorem \ref{thm:s_weak} to get
			\beqs
			\E\lone_{\Xi_{1}}\Norm{\frac{1}{X+A-z}}^{2}\lesssim N^{2}\absv{\log \fra}^{2}\frac{\norm{A}^{2}+1}{\lambda_{1}(\im[A-z])},
			\eeqs
			which leads to \eqref{eq:est1_R} via the same argument as in the complex case.
			
			Secondly in \eqref{eq:ovlp_prf1}, we again use the same replacement to prove \eqref{eq:est2_R};
			\beq
				\P[\Xi_{1}^{c}]\leq \P\left[\bigcup_{w\in\caL}[\lambda_{1}^{w}\leq 2\fra]\right]\lesssim \absv{\caL}(N\fra)^{2}\frac{(\absv{\log \fra}+\log N)}{y}\lesssim N^{2}r^{-1}\fra\frac{\absv{\log\fra}}{y}{\Cb \absv{\caD}}.
			\eeq
			Lastly in \eqref{eq:ovlp_prf2}, we {\Cb use Theorems \ref{thm:s_weak} (ii) and \ref{thm:s_real_weak} to get}
			\beq
				\P[\lambda_{1}^{z_{C}}\lambda_{2}^{z_{C}}\leq4r^{2}]\lesssim  \left({\Cb(N\kappa)^{4}}+\frac{1}{y}N^{2}\kappa^{-2}r^{4}\right)(\log N+\absv{\log \kappa}+\absv{\log r})^{2}.
			\eeq
			After optimizing $\kappa$ we obtain
			\beq
				\P[\lambda_{1}^{z_{C}}\lambda_{2}^{z_{C}}\leq 4r^{2}]\lesssim \frac{(Nr)^{8/3}}{y^{2/3}}(\log N+\absv{\log r}+\absv{\log y})^{2},
			\eeq
			which gives
			\beq
				\P[\Xi_{2}^{c}]\lesssim y^{-2/3}N^{8/3}\absv{\caD}r^{2/3}(\absv{\log r}+\absv{\log y})^{2}.
			\eeq
			This concludes the proof of Lemma \ref{lem:est}.
\end{proof}

\section{Proof of Theorem \ref{thm:s_strong}}\label{sec:s_strong}
In the rest of the paper we prove results in Section \ref{sec:s_result}. We start with the proof of Theorem \ref{thm:s_strong}, since it best represents the common core ideas. 

Recall that all of Theorems \ref{thm:s_weak}--\ref{thm:s_real_weak} concern singular values of $X+A$, where $X$ is a regular matrix and $A$ is a deterministic shift sometimes with additional restrictions. We denote the Hermitization of $X+A$ by $H$ and the resolvent $(H-w)^{-1}$ by $G(w)$ as in \eqref{eq:Hermit}. The spectral parameter of $G$ is always taken to be $w=\ii\eta$, where
\beq\label{eq:def_eta}
	\eta\equiv \eta(s)\deq \frac{s}{N}
\eeq
with the same parameter $s>0$ in Theorems \ref{thm:s_weak}--\ref{thm:s_real_weak}. Henceforth we often abbreviate $G\equiv G(\ii\eta)$. As seen below, all arguments and inequalities for the complex case is also used for the real case with some changes to exponents. In order to unify the presentation, we introduce the exponent $\beta$ defined by
\beq\label{eq:def_beta}
	\beta\deq \begin{cases}
		1 & \text{if $X$ is real},\\
		2 & \text{if $X$ is complex}.
	\end{cases}
\eeq

Before commencing the proofs, we introduce a few notations and observations regarding minors of matrices, that are used throughout the rest of the paper. Firstly, for each subset\footnote{The index set $\caI$ is unrelated to the domain $\caI\subset\R$ in Theorem \ref{thm:overlap_strong}; $\caI$ always denote an index set in what follows.} $\caI\subset\bbrktt{N}$ we define the matrix $J^{(\caI)}\in \C^{(\bbrktt{N}\setminus\caI)\times \bbrktt{N}}$ by
\beq\label{eq:def_J}
(J^{(\caI)})_{ij}\deq \delta_{ij}\lone(j\notin\caI)
\eeq
Note that $J^{(\caI)}$ acts as a $\caI$-shift operator; for each $A\in\C^{N\times n}$, taking product $J^{(\caI)}A\in \C^{(N-\absv{\caI})\times n}$ removes the $j$-th rows of $A$ for $j\in\caI$. Secondly, we write $H^{(\caI)}$ and $G^{(\caI)}$ for the Hermitization of $J^{(\caI)}(X+A)\in\C^{(\bbrktt{N}\setminus\caI)\times \bbrktt{N}}$ and its resolvent, that is,
\beq
	H^{(\caI)}\deq \begin{pmatrix} 0	& J^{(\caI)}(X+A) \\ (X+A)\adj (J^{(\caI)})\adj & 0 \end{pmatrix}\in\C^{((\bbrktt{N}\setminus\caI)\cup\bbrktt{N+1,2N})^{2}},\qquad G^{(\caI)}\equiv G^{(\caI)}(\ii\eta)\deq (H^{(\caI)}-\ii\eta)^{-1}.
\eeq
Note that $H^{(\caI)}$ is obtained from $H^{(\emptyset)}=H$ by removing the $j$-th rows and columns for $j\in\caI$. Thirdly, we write the spectral decomposition of $H^{(\caI)}$ as follows;
\beq
	H^{(\caI)}=\sum_{i\in\bbrktt{-(N-\absv{\caI}),-1}\cup\bbrktt{1,N}}\lambda^{(\caI)}_{i}\begin{pmatrix}\bsu_{i}^{(\caI)} \\ \bsv_{i}^{(\caI)}\end{pmatrix}\begin{pmatrix}\bsu_{i}^{(\caI)} \\ \bsv_{i}^{(\caI)}\end{pmatrix}\adj ,\qquad \bsu_{i}^{(\caI)}\in\C^{N-\absv{\caI}},\,\bsv_{i}^{(\caI)}\in\C^{N}
\eeq
where the eigenvalues are ordered increasingly. Note that the superscript $(\emptyset)$ is superfluous, for example $\lambda_{i}^{(\emptyset)}=\lambda_{i}$.

Now we list the key observations under these notations. Almost surely, the matrix $H^{(\caI)}$ has rank $2(N-\absv{\caI})$, and by the block structure of $H^{(\caI)}$ we have
\beq
	0=\lambda_{1}^{(\caI)}=\cdots=\lambda_{\absv{\caI}}^{(\caI)}<\lambda^{(\caI)}_{\absv{\caI}+1}<\cdots<\lambda^{(\caI)}_{N} \AND \lambda^{(\caI)}_{-i}=-\lambda_{i+\absv{\caI}}^{(\caI)}\quad\forall i\in\bbrktt{N-\absv{\caI}},
\eeq
where the strict inequalities are due to the fact that $H^{(\caI)}$ has simple spectrum. Furthermore, for any $\caI\subset\bbrktt{N}$ and $i\in\bbrktt{N}\setminus\caI$, since $H^{(\caI\cup\{i\})}$ is a principal minor of $H^{(\caI)}$, Cauchy interlacing theorem implies that
\beq\label{eq:interlace}
	\lambda_{i}^{(\caI\cup\{i\})}\leq\lambda_{i}^{(\caI)}\leq\lambda_{i+1}^{(\caI\cup\{i\})}\qquad \forall i\in\bbrktt{N}\setminus\caI.
\eeq
For the eigenvectors, we have
\beq\begin{aligned}
	\bsv^{(\caI)}_{-i}=\bsv^{(\caI)}_{i+\absv{\caI}},\quad \bsu^{(\caI)}_{-i}=-\bsu^{(\caI)}_{i+\absv{\caI}},&\AND\norm{\bsv^{(\caI)}_{i}}^{2}=\frac{1}{2}=\norm{\bsu^{(\caI)}_{i}}^{2}& \text{for }&i\in\bbrktt{1,N-\absv{\caI}},\\
	\bsu^{(\caI)}_{i}=0&\AND \norm{\bsv_{i}^{(\caI)}}^{2}=1 &\text{for }&i\in\bbrktt{\absv{\caI}}.
\end{aligned}\eeq
In particular, we can write the spectral decompositions
\beq\label{eq:minor_spec}\begin{aligned}
	(X+A)\adj (J^{(\caI)})\adj J^{(\caI)}(X+A)
	=&\sum_{i\in\bbrktt{\absv{\caI}}}0\cdot \bsv^{(\caI)}_{i}(\bsv^{(\caI)}_{i})\adj+
	\sum_{i\in\bbrktt{\absv{\caI}+1,N}}(\lambda^{(\caI)}_{i})^{2}(\sqrt{2}\bsv^{(\caI)}_{i})(\sqrt{2}\bsv^{(\caI)}_{i})\adj,	\\
	J^{(\caI)}(X+A)(X+A)\adj (J^{(\caI)})\adj=&\sum_{i\in\bbrktt{\absv{\caI}+1,N}}(\lambda^{(\caI)}_{i})^{2}(\sqrt{2}\bsu^{(\caI)}_{i})(\sqrt{2}\bsu^{(\caI)}_{i})\adj.
\end{aligned}\eeq
Finally, since the entries of $X$ are independent, the minor $H^{(\caI)}$ and the $j$-th rows of $X$ for $j\in\caI$ are independent. To be precise, we define the $\sigma$-algebras and the corresponding conditional expectations
\beq\label{eq:def_salg}
	\caF^{(\caI)}\deq\sigma(\{X_{ab}:a\in\bbrktt{N}\setminus\caI,b\in\bbrktt{N}\}),\qquad \E^{(\caI)}[\cdot]\deq \E[\cdot\vert\caF^{(\caI)}].
\eeq
Then all of $H^{(\caI)}$, $G^{(\caI)}$, and $(\lambda_{i}^{(\caI)},\bsu_{i}^{(\caI)},\bsv_{i}^{(\caI)})_{i\in\bbrktt{N}}$ are measurable with respect to $\caF^{(\caI)}$ and thus independent of the $\caI$-rows $\{X_{ab}:a\in\caI,b\in\bbrktt{N}\}$ of $X$.

Proofs of Theorems \ref{thm:s_weak} an \ref{thm:s_strong} both involve an induction over  minors $H^{(\caI)}$ as $\absv{\caI}$ runs through $\bbrktt{k}$, with the same $k$ as in their statements. To make this rigorous, we define for each $\caI\subset\bbrktt{N}$ with $\absv{\caI}\leq k$ the events\footnote{The events $\Xi_{k}$'s here are completely unrelated to those in \eqref{eq:O_Xi}.}
\beq\label{eq:def_Xi}
\Xi_{k}\deq [\lambda_{k}\leq\eta],\qquad \Xi_{k}^{(\caI)}\deq [\lambda_{k}^{(\caI)}\leq \eta],
\eeq
where we recall that $\eta=s/N$. Thus the goal of these two theorems is to bound $\P[\Xi_{k}]$ by $(N\eta)^{\beta k^{2}}$ with the exponent $\beta$ from \eqref{eq:def_beta}. Note that $\Xi_{k}^{(\caI)}\in\caF^{(\caI)}$ and that, due to $\lambda_{k}^{(\caI\cup\{i\})}\leq \lambda_{k}^{(\caI)}$ from \eqref{eq:interlace},
\beq\label{eq:ind_incl}
\Xi_{k}\equiv\Xi_{k}^{(\emptyset)}\subset \Xi^{(\caI_{1})}_{k}\subset \cdots\subset \Xi^{(\caI_{k})}_{k}
\eeq
holds true for any increasing sequence of index sets $(\caI_{j})_{j\in\bbrktt{k}}$. Also note that if $\absv{\caI_{k}}=k$, then the last event $\Xi_{k}^{(\caI_{k})}$ has probability one since $\lambda_{k}^{(\caI_{k})}\equiv 0$.

As we will see shortly, Theorems \ref{thm:s_weak} and \ref{thm:s_strong} both follow from an induction over the chain \eqref{eq:ind_incl} for sequences $(\caI_{j})_{j\in\bbrktt{k}}$ with $\absv{\caI_{j}}=j$, gaining a factor of $(N\eta)^{\beta k}$ at each step. To be precise, we prove
\beq\label{eq:ind}
	\E[\frX_{\caI}\lone_{\Xi_{k}^{(\caI)}}]\leq(N\eta)^{\beta k}\frac{1}{N-\absv{\caI}}\sum_{i\in\bbrktt{N}\setminus \caI}\E[\frX_{\caI\cup\{i\}}\lone_{\Xi_{k}^{(\caI\cup\{i\})}}], \quad \forall \eta\in[0,N^{-1}], \,\,\forall \caI\subset\bbrktt{N},\,\absv{\caI}\leq k-1,
\eeq
for a collection of positive random variables $(\frX_{\caI})_{\caI\subset\bbrktt{N}}$ such that $\frX_{\caI}\in\caF^{(\caI)}$ and $\frX_{\emptyset}\equiv1$. Once we have \eqref{eq:ind}, then it immediately follows that
\begin{align}
	\P[\Xi_{k}]=\E[\frX_{\emptyset}\lone_{\Xi_{k}^{(0)}}]\leq &(N\eta)^{\beta k}\frac{1}{N}\sum_{i\in\bbrktt{N}}\E[\frX_{\{i\}}\lone_{\Xi_{k}^{(\{i\})}}]\leq 
	(N\eta)^{2\beta k}\frac{2}{N(N-1)}\sum_{i_{1}\neq i_{2}\in\bbrktt{N}}\E\frX_{\{i_{1},i_{2}\}}\lone_{\Xi_{k}^{(\{i_{1},i_{2}\})}}\leq \cdots	\nonumber\\
	\leq&(N\eta)^{\beta k^{2}}\frac{(N-k)!k!}{N!}\sum_{\caI\subset\bbrktt{N},\absv{\caI}=k}\E[\frX_{\caI}\lone_{\Xi_{k}^{(\caI)}}]\leq(N\eta)^{\beta k^{2}}\max_{\caI\subset\bbrktt{N},\absv{\caI}=k}\E[\frX_{\caI}],\label{eq:ind_appl}
\end{align}
so that the only remaining thing is to estimate $\E[\frX_{\caI}]$ for $\absv{\caI}=k$, which becomes the correction. The exact form of $\frX_{\caI}$ in \eqref{eq:ind} varies depending on the goal, and in particular $\frX_{\caI}$'s are deterministic in the proof of Theorem \ref{thm:s_weak}.

Although all discussions above are written in terms of a general index set $\caI$ for the sake of rigor, in the actual proof we often choose $\caI=\bbrktt{j}\subset\bbrktt{k}$ for clarity. To further simplify the notation, we write the superscript $(j)$ instead of $(\bbrktt{j})$ for an integer $j$, for example $\lambda^{(j)}\equiv\lambda^{(\bbrktt{j})}$, $\Xi_{k}^{(0)}\equiv\Xi_{k}^{(\bbrktt{0})}=\Xi_{k}^{(\emptyset)}$, et cetera.

We now present the two inputs for the proof of Theorem \ref{thm:s_strong}, whose proofs are postponed to the end of this section. Note that the first input, Proposition \ref{prop:repel}, applies to a general regular matrix $X$ but assumes that $A$ is real when $X$ is.
\begin{prop}\label{prop:repel}
	Let $k,N\in\N$ with $N\geq 3k$, $\bbK=\R$ or $\C$, $X\in\K^{N\times N}$ be a regular matrix, and $A\in\bbK^{N\times N}$ be deterministic. Then there exists a constant $C\equiv C(k,\frb)$ such that 
		\beq\label{eq:repel}
		\E (1+N\lambda^{(\caI)}_{m})^{n}\lone_{\Xi_{k}^{(\caI)}}\leq C(N\eta)^{\beta k}\frac{1}{N-\absv{\caI}}\sum_{i\in\bbrktt{N}\setminus\caI}\E(1+N\lambda_{m+1}^{(\caI\cup\{i\})})^{n+\beta k}\lone_{\Xi_{k}^{(\caI\cup\{i\})}}\quad  \forall \eta\in[0,N^{-1}]
		\eeq
		holds for any $\caI\subset\bbrktt{N}$ with $\absv{\caI}\leq k-1$, $m\in\bbrktt{2k,N-1}$, and $n\geq0$.
\end{prop}
\begin{lem}\label{lem:lambdak}
	Let $X$ be a regular i.i.d. matrix, $A\in\C^{N\times N}$ with $\norm{A}\leq\frK$, and for each $z\in\C$ let $\lambda_{1}^{z}\leq\cdots\leq\lambda_{N}^{z}$ be singular values of $X+A-z$. For each fixed $\delta,\tau>0$, there exists a constant $C_{2}\equiv C_{2}(k,\delta,\tau,\bsfrm,\frK)$ such that
	\beq\label{eq:lambda_k}
	\E[(1+N\lambda_{3k}^{z})^{2k^{2}}]\leq C_{2}N^{\delta}
	\eeq
	holds for all $N\geq 3k$ and $\absv{z}\leq 1-\tau$.
\end{lem}

We finally prove Theorem \ref{thm:s_strong} using these inputs.
\begin{proof}[Proof of Theorem \ref{thm:s_strong}]
	Recall the definitions of $\eta$ and $\beta$ respectively from \eqref{eq:def_eta} and \eqref{eq:def_beta}. Define
	\beq
		\frX_{\caI}\deq C^{\absv{\caI}}(1+N\lambda_{2k+\absv{\caI}}^{(\caI)})^{\beta \absv{\caI}k}
	\eeq
	where $C$ is the constant from Proposition \ref{prop:repel}. Then, since we assumed $z\in\R$ in the real case, we may apply \eqref{eq:repel} with $A$ replaced by $A-z$. This immediately proves \eqref{eq:ind} and hence, by \eqref{eq:ind_appl} we have
	\beq\label{eq:s_strong_1}
		\P[N\lambda_{k}^{z}\leq s]\leq s^{\beta k^{2}}\max_{\caI\subset\bbrktt{N},\absv{\caI}=k}\E[\frX_{\caI}]\leq C^{k}s^{k^{2}}\E[(1+N\lambda_{3k}^{z})^{\beta k^{2}}],
	\eeq 
	where we used $\lambda_{3k}^{(\caI)}\leq \lambda_{3k}$ from \eqref{eq:interlace} in the last inequality. In particular, by taking $C_{1}=C^{k}$, this proves the first inequalities in \eqref{eq:s_C_strong} and \eqref{eq:s_R_strong}.
	
	Finally, since 
	\beqs
		\E(1+N\lambda_{3k})^{k^{2}}\leq (\E(1+N\lambda_{3k})^{2k^{2}})^{1/2},
	\eeqs
	substituting \eqref{eq:lambda_k} into \eqref{eq:s_strong_1} concludes the proof of Theorem \ref{thm:s_strong}.
	\end{proof}

\begin{proof}[Proof of Proposition~\ref{prop:repel}]
	First of all, we prove that one may take $\caI=\bbrktt{j-1}$ for some $j\in\bbrktt{k}$ without loss of generality. To this end, we take a general $\caI\subset\bbrktt{N}$ with $\absv{\caI}\leq k-1$ and assume that Proposition \ref{prop:repel} is valid for $\bbrktt{\absv{\caI}}$. Consider a permutation matrix $P_{\caI}$ that maps $\caI$ into $\bbrktt{\absv{\caI}}$, so that $J^{(\absv{\caI})}P_{\caI}=J^{(\caI)}$. Then it is easy to see that $P_{\caI}X$ is regular with the same density bound $\frb$ if and only if $X$ is, so that we may apply Proposition \ref{prop:repel} to $P_{\caI}(X+A)$. Noticing that the constant $C_{1}$ in \eqref{eq:repel} depends only on $(k,\frb)$ and that
	\beq\label{eq:permut}
		\lambda_{m}^{(\absv{\caI})}(P_{\caI}(X+A))=\lambda_{m}^{(\caI)},
	\eeq
	where the left-hand side stands for the $m$-th smallest singular value of $J^{(\absv{\caI})}P_{\caI}(X+A)$, the result for general $\caI$ immediately follows.
	
	Therefore in what follows we take $\caI=\bbrktt{j-1}$ for some $j\in\bbrktt{k}$, and all superscripts $(\caI)$ are replaced by $(j-1)$. Next, we write
	\beq
		\im\sum_{i\in\bbrktt{j,N}}G^{(j-1)}_{ii}(\ii\eta)=\im \Tr \frac{\ii\eta}{J^{(j-1)}(X+A)(X+A)\adj J^{(j-1)*}+\eta^{2}}	=\sum_{i\in\bbrktt{j,N}}\frac{\eta}{\absv{\lambda^{(j-1)}_{i}}^{2}+\eta^{2}},
	\eeq
	where we recall that $H^{(j-1)}$ and $G^{(j-1)}$ are indexed by $\bbrktt{j,2N}^{2}$. Recalling the definition of $\Xi^{(j-1)}_{k}$ from \eqref{eq:def_Xi}, we have
	\beq\label{eq:repel_G_det}
		\lone_{\Xi^{(j-1)}_{k}}\im\sum_{i\in\bbrktt{j,N}}G^{(j-1)}_{ii}\geq\lone_{\Xi^{(j-1)}_{k}}\sum_{i=j}^{k}\frac{\eta}{\absv{\lambda^{(j-1)}_{i}}^{2}+\eta^{2}}\geq \frac{1}{2\eta}(k-j+1)\lone_{\Xi^{(j-1)}_{k}}\geq \frac{1}{2\eta}\lone_{\Xi^{(j-1)}_{k}},
	\eeq
	and hence, raising it to the $(\beta k)$-th power (recall that $\beta=1$ or $2$ depending on real or complex $X$),
	\beq\label{eq:repel_G}
	\lone_{\Xi^{(j-1)}_{k}}\leq (2N\eta)^{\beta k}\left(\frac{1}{N-j+1}\sum_{i\in\bbrktt{j,N}}\im G_{ii}^{(j-1)}\right)^{\beta k}\lone_{\Xi_{k}^{(j-1)}}.
	\eeq
	Multiplying both sides of \eqref{eq:repel_G} by the factor $(1+N\lambda_{m}^{(j-1)})^{n}$ and using Jensen's inequality, we have
	\beq\label{eq:Jensen}
		\E\lone_{\Xi^{(j-1)}_{k}}(1+N\lambda_{m}^{(j-1)})^{n}\leq \frac{1}{N-j+1}(2N\eta)^{\beta k} \E\lone_{\Xi^{(j-1)}_{k}} (1+N\lambda_{m}^{(j-1)})^{n}\sum_{i\in\bbrktt{j,N}} \absv{G^{(j-1)}_{ii}}^{\beta k}.
	\eeq
	Then, for each $i\in\bbrktt{j,N}$ we use that $\lambda_{m}^{(j-1)}\leq \lambda_{m+1}^{(\bbrktt{j-1}\cup\{i\})}$ from \eqref{eq:interlace} and $\Xi^{(j-1)}_{k}\subset\Xi^{(\bbrktt{j-1}\cup\{i\})}_{k}$ to estimate the $i$-th summand in the right-hand side of \eqref{eq:Jensen} as
	\beq\label{eq:Jensen_1}
		\E\lone_{\Xi^{(j-1)}_{k}} (1+N\lambda_{m}^{(j-1)})^{n}\absv{G^{(j-1)}_{ii}}^{\beta k}
		\leq \E\lone_{\Xi^{(\bbrktt{j-1}\cup\{i\})}_{k}}(1+N\lambda_{m+1}^{(\bbrktt{j-1}\cup\{i\})})^{n}\E^{(\bbrktt{j-1}\cup\{i\})}\absv{G^{(j-1)}_{ii}}^{\beta k},
	\eeq
	where we recall the definition of $\E^{(\caI)}$ from \eqref{eq:def_salg} and that $\lambda_{m+1}^{(\caI)},\Xi^{(\caI)}_{k}$ are measurable with respect to $\caF^{(\caI)}$. Substituting \eqref{eq:Jensen_1} into \eqref{eq:Jensen} and then comparing with the final goal \eqref{eq:repel}, we find that it suffices to prove for all $i\in\bbrktt{j,N}$ that
	\beq
		\lone_{\Xi^{(\bbrktt{j-1}\cup\{i\})}_{k}}\E^{(\bbrktt{j-1}\cup\{i\})}\absv{G^{(j-1)}_{ii}}^{\beta k}\leq C\lone_{\Xi^{(\bbrktt{j-1}\cup\{i\})}_{k}}(1+N\lambda_{2k+1}^{(\bbrktt{j-1}\cup\{i\})})^{\beta k},
	\eeq
	for a constant $C$ depending only on $k$, $\frb$. At this point we may further assume $i=j$ without loss of generality, using the same argument as in \eqref{eq:permut} involving permutation matrices. Finally, recalling the assumption $m\geq 2k$, it only remains to prove
	\beq\label{eq:Jensen_result}
		\lone_{\Xi^{(j)}_{k}}\E^{(j)}\absv{G_{jj}^{(j-1)}}^{\beta k}\leq C\lone_{\Xi^{(j)}_{k}}(1+N\lambda_{2k+1}^{(j)})^{\beta k}.
	\eeq
	
	Next, we decouple the $j$-th row of $X$ from $G_{jj}^{(j-1)}$. Applying Schur complement formula with respect to the $(j,j)$-th entry of $H^{(j-1)}$ gives
	\beq
		\frac{1}{G^{(j-1)}_{jj}}=-\ii\eta -\sum_{k,\ell\in\bbrktt{j,2N}}H^{(j-1)}_{jk}G^{(j)}_{k\ell}H^{(j-1)}_{\ell j}	
		=-\ii\eta-\sum_{k,\ell\in\bbrktt{N}}(X+A)_{jk}G^{(j)}_{k+N,\ell+N}(\ol{X}+\ol{A})_{j\ell}.
	\eeq
	On the other hand, another application of Schur's complement with respect to the bottom right $(N\times N)$ block of $(H^{(j)}-\ii\eta)$ gives
	\beq\label{eq:Schur_G22}
	G^{(j)}_{k+N,\ell+N}=\ii\eta\left((X+A)\adj J^{(j)*}J^{(j)}(X+A)+\eta^{2}\right)^{-1}_{k\ell}, \qquad k,\ell\in\bbrktt{N},
	\eeq
	so that \eqref{eq:minor_spec} gives
	\beq\label{eq:Schur_minor}
		\frac{1}{G^{(j-1)}_{jj}}
		=-\ii\eta-\ii\eta\sum_{i=1}^{N}\frac{(2-\lone(i\leq j))}{(\lambda^{(j)}_{i})^{2}+\eta^{2}}\absv{\bse_{j}\adj(X+A)\bsv_{i}^{(j)}}^{2}=-\ii\eta-\ii\sum_{i=1}^{N}c_{i}^{(j)}\absv{w_{i}^{(j)}}^{2},
	\eeq
where we defined (recall that $\norm{\bsv_{i}^{(j)}}^{2}$ is equal to $1$ for $i\leq j$ and $1/2$ for $i>j$)
\beq
\begin{aligned}\label{eq:repel_def}
	&c_{i}^{(j)}\deq \frac{N\eta}{(N\lambda_{i}^{(j)})^{2}+(N\eta)^{2}},&\qquad 
	&\bsw^{(j)}\deq (w^{(j)}_{1},\cdots,w^{(j)}_{N})\deq U^{(j)}(\bsb^{(j)}+\bsa^{(j)}), 	
	&\bsb^{(j)}\deq \sqrt{N}X\adj\bse_{j}, & 	\\
	&\bsa^{(j)}\deq \sqrt{N}A\adj\bse_{j}, &
	&U^{(j)}\deq\begin{pmatrix} \bsv^{(j)}_{1} & \cdots 
		&\bsv^{(j)}_{j} & \sqrt{2}\bsv^{(j)}_{j+1} & \cdots \sqrt{2}\bsv^{(j)}_{N}\end{pmatrix}\adj. &
\end{aligned}
\eeq
While \eqref{eq:Schur_minor} is an identity, we do not need all summands on the right-hand side but only that
\beq\label{eq:Schur_minor_strong}
		\absv{G_{ii}^{(j-1)}}=\left(\eta+\sum_{i=1}^{N}c_{i}^{(j)}\absv{w_{i}^{(j)}}^{2}\right)^{-1}\leq \left(\sum_{i\in\bbrktt{2k+1}}c_{i}^{(j)}\absv{w_{i}^{(j)}}^{2}\right)^{-1}.
\eeq
Note that $c_{i}^{(j)}$ and $U^{(j)}$ are measurable with respect to $\caF^{(j)}$ defined in \eqref{eq:def_salg} and that $\bsb^{(j)}$ is exactly the $j$-th row of $X$ with a deterministic shift, hence $c_{i}^{(j)}$ and $U^{(j)}$ are independent of $\bsb^{(j)}$. Also the matrix $(U^{(j)})\adj$, and hence $U^{(j)}$, is unitary due to \eqref{eq:minor_spec}. Here we emphasize that, on the event $\Xi^{(j)}_{k}=[\lambda^{(j)}_{k}\leq\eta]$, sizes of $c_{i}^{(j)}$'s are roughly
\beq\label{eq:c_size}
\frac{1}{2N\eta}\leq c_{i}^{(j)}\leq \frac{1}{N\eta}\quad \text{for }i\leq k, \AND c_{i}^{(j)}\sim (N\eta) \quad \text{for } i>k
\eeq
since $N\lambda^{(j)}_{i}\sim1$ for $i>k$. Notice that the first inequality in \eqref{eq:c_size} is not a heuristic, but deterministically true on $\Xi_{k}^{(j)}$.

Substituting \eqref{eq:Schur_minor_strong} into the left-hand side of \eqref{eq:Jensen_result}, we find that it suffices to prove
\beq\label{eq:repel_tech}
	\lone_{\Xi^{(j)}_{k}}\E^{(j)}\left(\sum_{i\in\bbrktt{2k+1}}c_{i}^{(j)}\absv{w_{i}^{(j)}}\right)^{-\beta k}\leq C\lone_{\Xi^{(j)}_{k}}(1+N\lambda_{2k+1}^{(j)})^{\beta k}
\eeq
for a constant $C$ depending only on $k$ and $\frb$. The inequality \eqref{eq:repel_tech} follows from the following technical lemma, whose proof is presented after completing that of Theorem \ref{thm:s_strong}.
\begin{lem}\label{lem:tech}
	Let $k,N\in\N$ with $2k+1\leq N$, $\bbK=\R$ or $\C$, $\frb>0$, and $\bsb\in\bbK^{N}$ be a random vector with independent components. Assume that each component of $\bsb$ ($\re\bsb$ and $\im\bsb$, resp.) has a density bounded by $\frb$ if $\bbK=\R$ (if $\bbK=\C$, resp.). Then there exists a constant $\frC_{1}\equiv \frC_{1}(k,\frb)>0$ such that the following holds for any positive sequence $(c_{i})_{i\in\bbrktt{N}}$ and a unitary matrix $U\in\bbK^{N\times N}$;
	\beq\label{eq:tech}
		\E\left(\sum_{i\in\bbrktt{2k+1}}c_{i}\absv{\bse_{i}\adj U\bsb}^{2}\right)^{-\beta k}\leq 1+ \frC_{1}\prod_{i=1}^{k}c_{i}^{-\beta/2}\cdot \prod_{i=k+1}^{2k+1}c_{i}^{-\beta k/(2(k+1))}.
	\eeq
\end{lem}

Notice that Lemma \ref{lem:tech} assumes that $U$ is real orthogonal when $\bbK=\R$. Assuming Lemma \ref{lem:tech} is valid, we complete the proof of Theorem \ref{thm:s_strong}. We aim at applying Lemma \ref{lem:tech} with the choices $(\bsb,c_{i},U)=(\bsb^{(j)}+\bsa^{(j)},c_{i}^{(j)},U^{(j)})$. First of all, recalling that $A$ is real when $X$ is, the matrix $(X+A)\adj\in\bbK^{N\times N}$ is regular so that each component of the vector $(\bsb^{(j)}+\bsa^{(j)})\in\bbK^{N}$ has a density bounded by $\frb$. Secondly, again due to $A\in\bbK^{N\times N}$, the singular vectors $(\bsv_{i}^{(j)})_{i\in\bbrktt{N}}$ are in $\bbK^{N}$ so that $U^{(j)}\in\bbK^{N\times N}$. Lastly, since the constant $\frC_{1}$ in \eqref{eq:tech} is uniform over all $(c_{i})$ and $U$, we may as well take them to be random as long as they are independent of $\bsb$. More precisely, for any $\sigma$-algebra $\caF$ independent of $\bsb$ such that $\sigma((c_{i}),U)\subset \caF$, we may also replace $\E$ on the left-hand side of \eqref{eq:tech} with the conditional expectation $\E[\cdot\vert\caF]$. In particular for $\bsb=\bsb^{(j)}$ we may replace $\E$ with $\E^{(j)}$.

Thus we may apply Lemma \ref{lem:tech} to the left-hand side of \eqref{eq:repel_tech}, so that 
\beq
\begin{aligned}
	\E^{(j)}\left(\sum_{i\in\bbrktt{2k+1}}c_{i}^{(j)}\absv{w_{i}^{(j)}}^{2}\right)^{-\beta k}\leq& 1+C\prod_{i=1}^{k}(c_{i}^{(j)})^{-\beta/2}\prod_{i=k+1}^{2k+1}(c_{i}^{(j)})^{-\beta k/(2(k+1))}	\\
	\leq& 1+C(c_{k}^{(j)}c_{2k+1}^{(j)})^{-\beta k/2},
	\end{aligned}
\eeq
where in the second line we used that $c_{i}^{(j)}$ is decreasing in $i\in\bbrktt{N}$. Then recalling $\Xi_{k}^{(j)}=[\lambda_{k}^{(j)}\leq \eta]$ we have
\beq
	\lone_{\Xi^{(j)}_{k}}\frac{1}{c_{k}^{(j)}c_{2k+1}^{(j)}}=\lone_{\Xi^{(j)}_{k}}\left((N\eta)^{2}+(N\lambda_{2k+1}^{(j)})^{2}\right)\cdot \frac{(N\lambda_{k}^{(j)})^{2}+(N\eta)^{2}}{(N\eta)^{2}}\leq 2\lone_{\Xi^{(j)}_{k}}(1+(N\lambda_{2k+1}^{(j)})^{2})
\eeq
where we used $N\eta=s\leq 1$. Therefore we conclude
\beqs
	\lone_{\Xi_{k}^{(j)}}\E^{(j)}\left(\sum_{i\in\bbrktt{N}}c_{i}^{(j)}\absv{w_{i}^{(j)}}^{2}\right)^{-\beta k}
	\leq \lone_{\Xi^{(j)}_{k}}\left(1+C\left(1+(N\lambda_{2k+1}^{(j)})^{2}\right)^{\beta k/2}\right)
	\leq C'\lone_{\Xi_{k}^{(j)}}\left(1+N\lambda_{2k+1}^{(j)}\right)^{\beta k},
\eeqs
by taking suitable constant $C'$ that still depends only on $(k,\frb)$. This finishes the proof of \eqref{eq:repel_tech}, concluding that of Proposition~\ref{prop:repel}.
\end{proof}

\begin{proof}[Proof of Lemma \ref{lem:tech}]
We define probability measures  $\mu$ and $\mu_{p}$ on $\K^{N}$ and $\K^{p}$ to be the laws of $U\bsb$ and its first $p$ coordinates. More specifically, for any suitable test functions $f:\K^{N}\to\C$ and $f_{p}:\K^{p}\to\C$ we define
\beq\label{eq:dens}
	\int_{\K^{N}}f\dd\mu\deq \int_{\K^{N}}f(U\bsb)h(\bsb)\dd^{\beta N}\bsb,\qquad \int_{\K^{p}}f_{p}\dd\mu_{p}\deq \int_{\K^{N}}f_{p}(w_{1},\cdots,w_{p})\dd\mu(\bsw)
\eeq
where we abbreviated $h(\bsb)\deq\prod_{i=1}^{N}h_{i}(b_{i})$. With $\mu_{2k+1}$, we may rewrite the left-hand side of \eqref{eq:tech} as
\beq\label{eq:repel_Schur_result}
	\E\left(\sum_{i\in\bbrktt{2k+1}}c_{i}\absv{\bse_{i}\adj U\bsb}^{2}\right)^{-\beta k}=\int_{\K^{2k+1}}\left(\sum_{i=1}^{2k+1}c_{i}\absv{w_{i}}^{2}\right)^{-\beta k}\dd\mu_{2k+1}(\bsw).
\eeq

For a threshold $\kappa>0$ that will be optimized afterwards, we divide the integration in \eqref{eq:repel_Schur_result} by inserting the following factor;
\begin{gather}\label{eq:repel_kappa}
	1=\lone_{S_{1}^{c}}+\lone_{S_{1}\cap S_{2}^{c}}+\lone_{S_{1}\cap S_{2}},	\\
	S_{1}\deq \left\{\bsw\in\K^{2k+1}:\sum_{i=1}^{k}c_{i}\absv{w_{i}}^{2}\leq 1\right\},\qquad 
	S_{2}\deq \left\{\bsw\in\K^{2k+1}:\sum_{i=k+1}^{2k+1}c_{i}\absv{w_{i}}^{2}\leq \kappa^{2}\right\}.\nonumber
 \end{gather}
The integral in \eqref{eq:repel_Schur_result} corresponding to the first regime $S_{1}^{c}$ is simply bounded by $1$ as $\mu_{2k+1}$ is a probability measure and the integrand is bounded by one. For the second domain $S_{1}\cap S_{2}^{c}$, denoting the $L^{\infty}$ norm of the Lebesgue density of a measure $\nu$ by $\norm{\nu}_{\infty}$, we have
\beq\label{eq:S1_S2c}\begin{aligned}
	&\int_{\K^{2k+1}}\frac{\lone_{S_{1}\cap S_{2}^{c}}(\bsw)}{\left(\sum_{i=1}^{2k+1}c_{i}\absv{w_{i}}^{2}\right)^{\beta k}}\dd\mu_{2k+1}(\bsw)
	\leq \int_{\K^{k}} \frac{\lone(\sum_{i=1}^{k}c_{i}\absv{w_{i}}^{2}\leq 1)}{\left(\sum_{i=1}^{k}c_{i}\absv{w_{i}}^{2}+\kappa^{2}\right)^{\beta k}}\dd\mu_{k}(\bsw)	\\
	\leq &\norm{\mu_{k}}_{\infty}\prod_{i=1}^{k}c_{i}^{-\beta/2}\int_{\K^{k}}\frac{\lone(\norm{\bsx}\leq 1)}{(\norm{\bsx}^{2}+\kappa^{2})^{\beta k}}\dd^{\beta k}\bsx	
	\leq C\norm{\mu_{k}}_{\infty}\prod_{i=1}^{k}c_{i}^{-\beta/2}\int_{0}^{1}\frac{s^{\beta k-1}}{(s^{2}+\kappa^{2})^{\beta k}}\dd s \\
	\leq &C\norm{\mu_{k}}_{\infty}\kappa^{-\beta k}\prod_{i=1}^{k}c_{i}^{-\beta/2},
\end{aligned}
\eeq
where we used the change of variables $x_{i}=\sqrt{c_{i}}w_{i}$ in the second inequality. Here the constant $C>0$ is a number that depends solely on $\beta$. Likewise, for the integral over the third domain $S_{1}\cap S_{2}$ we get
\beq\label{eq:S1_S2}\begin{aligned}
	&\int_{\K^{2k+1}}\frac{\lone_{S_{1}\cap S_{2}}(\bsw)}{\left(\sum_{i=1}^{2k+1}c_{i}\absv{w_{i}}^{2}\right)^{\beta k}}\dd\mu_{2k+1}(\bsw)	\\
	\leq& C\norm{\mu_{2k+1}}_{\infty}\prod_{i=1}^{2k+1}c_{i}^{-\beta/2}\int_{\K^{k+1}}\int_{\K^{k}}\frac{\lone(\norm{\bsx}\leq 1)\lone(\norm{\bsy}\leq \kappa)}{(\norm{\bsx}^{2}+\norm{\bsy}^{2})^{\beta k}}\dd^{\beta k}\bsx\dd^{\beta(k+1)}\bsy	\\
	\leq&C\norm{\mu_{2k+1}}_{\infty} \prod_{i=1}^{2k+1}c_{i}^{-\beta/2}\int_{0}^{\kappa}\int_{0}^{1}\frac{s^{\beta k-1} t^{\beta(k+1)-1}}{(s^{2}+t^{2})^{\beta k}}\dd s\dd t	
	\leq C\norm{\mu_{2k+1}}_{\infty}\kappa^{\beta}\prod_{i=1}^{2k+1}c_{i}^{-\beta/2}.
\end{aligned}
\eeq

	The fact that $\norm{\mu_{p}}_{\infty}\leq C(p,\frb)$ for any $p\leq N$ is a direct consequence of the following lemma.
	\begin{lem}[{\cite[Theorem 1.1]{Rudelson-Vershynin2015}}]\label{lem:dens_proj}
		Let $\wt{p}\leq \wt{N}$ be positive integers, $\frb>0$, and $\wt\bsb=(\wt{b}_{i})_{i\in\bbrktt{N}}\in\bbR^{\wt{N}}$ be a random vector with independent components. If the densities of $\wt{b}_{i}$ are bounded by $\frb$, then the density of $\wt{P}\wt{\bsb}$ on $\wt{P}\R^{N}$ is bounded by $(C\frb)^{\wt{p}}$ for any orthogonal projection $\wt{P}$ on $\bbR^{N}$ with $\mathrm{rank}\,\wt{P}=\wt{p}$ where $C$ is an absolute constant.
	\end{lem}
	In the real case, we apply Lemma \ref{lem:dens_proj} with $\wt{p}=p$, $\wt{N}=N$, $\wt{\bsb}=\bsb,$ and
		\beq
		\wt{P}=U\adj \begin{pmatrix} I_{p} & O \\ O & O\end{pmatrix} U,
	\eeq
	to find that the density of $\wt{P}\wt{\bsb}$ is bounded on $\wt{P}\R^{N}$. Since $\mu_{p}$ is the law of $U\wt{P}\wt{\bsb}$ on $U\wt{P}\R^{N}=(\R^{p},\boldsymbol{0})$, we immediately have $\norm{\mu_{p}}_{\infty}\leq (C\frb)^{p}$.
	To apply Lemma \ref{lem:dens_proj} in the complex case, we employ the following notations: For matrices $A\in\C^{d_{1}\times d_{2}}$, we define
	\beq\label{eq:def_caJ}
	\caJ[A]\deq\begin{pmatrix}
		\Re [A] &-\Im [A] \\\Im [A] & \Re [A]
	\end{pmatrix}\in\R^{2d_{1}\times 2d_{2}}.
	\eeq
	where $\Re [A]$ and $\Im[A]$ are the entrywise real and imaginary parts of $A$ defined in \eqref{eq:def_re}.
	With a slight abuse of notation, when $\bsv\in\C^{d}$ is a vector, we define $\caJ[\bsv]\deq(\re[\bsv]\tp,\im[\bsv]\tp)\tp\in\R^{2d}$. Note that $\caJ$ is an $\R$-linear algebra homomorphism in the sense that, for all $A\in\C^{d_{1}\times d_{2}},B\in\C^{d_{2}\times d_{3}}$, and $\bsv\in\C^{d_{2}}$,
	\beq
		\caJ[AB]=\caJ[A]\caJ[B],\qquad \caJ[A\bsv]=\caJ[A]\caJ[\bsv],\qquad \caJ[A\adj]=\caJ[A]\tp.
	\eeq
	Now we apply Lemma \ref{lem:dens_proj} with the choices $\wt{p}=2p$, $\wt{N}=2N$, $\wt{\bsb}=\caJ[\bsb]$, and
	\beq
		\wt{P}=\caJ[U\adj]\caJ\left[\begin{pmatrix} I_{p} & O \\ O&O\end{pmatrix}\right]\caJ[U].
	\eeq
	Noticing that $\caJ[U]\in\R^{2N\times 2N}$ is a real orthogonal matrix, we immediately find $\norm{\mu_{p}}_{\infty}\leq (C\frb)^{2p}$.

Combining \eqref{eq:S1_S2c}, \eqref{eq:S1_S2}, $\norm{\mu_{k}}_{\infty},\norm{\mu_{2k+1}}_{\infty}=O(1)$, and optimizing for $\kappa$ gives
\beqs
	\E\left(\sum_{i=1}^{2k+1}c_{i}\absv{\bse_{i}\adj U\bsb}^{2}\right)^{-\beta k}\leq 1+C\left(\prod_{i=1}^{k}c_{i}^{-\beta/2}\right)\left(\kappa^{-\beta k}+\kappa^{\beta}\prod_{i=k+1}^{2k+1}c_{i}^{-\beta/2}\right)
	\leq 1+C\prod_{i=1}^{k}c_{i}^{-\beta/2}\prod_{i=k+1}^{2k+1}c_{i}^{-k\beta/(2(k+1))}.
\eeqs
This completes the proof of Lemma \ref{lem:tech}.
\end{proof}

\begin{proof}[Proof of Lemma \ref{lem:lambdak}]
	Recall that the goal is to prove
	\beq\label{eq:lambda_k_1}
	\E[(1+N\lambda_{3k})^{2k^{2}}]\leq C(k,\delta,\tau,\bsfrm)N^{\delta}.
	\eeq
	Observe that we may assume that $N$ is sufficiently large, as long as the threshold depends on the same parameters as $C$. Recall the definition of $\Xi_{z}$ from \eqref{eq:rigid}. We apply Lemma \ref{lem:ll} with the choices $\epsilon=\xi=\delta/(100k^{2})$ and $D=100k^{2}$, so that $[\Xi_{z}(\epsilon,\xi)^{c}]\leq N^{-D}$ holds for all $\absv{z}\leq 1-\tau$. Then, on the event $\Xi_{z}(\epsilon,\xi)$ we have
	\beqs
	N\lambda_{3k}^{z}\leq N\lambda_{\lfloor N^{\epsilon+\xi}\rfloor}^{z}\leq N^{1+\epsilon}\eta\leq N^{1+2\epsilon},
	\eeqs
	hence
	\beq\label{eq:lambda_k_crude_1}
	\E\lone_{\Xi_{z}(\epsilon,\xi)}(1+N\lambda_{3k}^{z})^{2k^{2}}\leq N^{5k^{2}\epsilon}.
	\eeq
	On the complementary event $\Xi_{z}(\epsilon,\xi)^{c}$ we use Cauchy-Schwarz and $\lambda_{N}^{z}=\norm{A+X-z}\leq\norm{X}+\absv{z}+\norm{\frK}$ to get
	\beq
	\E\lone_{\Xi_{z}(\epsilon,\xi)^{c}}(1+N\lambda_{3k}^{z})^{2k^{2}} \leq \P[\Xi_{z}(\epsilon,\xi)^{c}]^{1/2} \E[(1+N\lambda_{N}^{z})^{4k^{2}}]^{1/2}\leq N^{2k^{2}-D/2}(1+\norm{\norm{X}}_{4k^{2}})^{2k^{2}}.
	\eeq
	Then we finally use
	\beq
	\E[\norm{X}^{4k^{2}}]\leq \E (\Tr XX\adj)^{2k^{2}}\leq N^{2k^{2}}\frm_{4k^{2}}
	\eeq
	so that
	\beq\label{eq:lambda_k_crude_2}
	\E\lone_{\Xi_{z}(\epsilon,\xi)^{c}}(1+N\lambda_{3k}^{z})^{2k^{2}}\leq N^{3k^{2}-D/2} \frm_{4k^{2}}^{1/2}.
	\eeq
	Substituting the definitions of $\epsilon,\xi,D$ into \eqref{eq:lambda_k_crude_1} and \eqref{eq:lambda_k_crude_2} proves \eqref{eq:lambda_k}, concluding the proof of Lemma \ref{lem:lambdak}.
\end{proof}

\section{Proof of Theorem \ref{thm:s_weak}}\label{sec:s_weak}
The proof of Theorem \ref{thm:s_weak} follows a similar strategy as in Section \ref{sec:s_strong} for the proof of Theorem \ref{thm:s_strong} but with some nontrivial modifications that will be summarized in Remark \ref{rem:strong_vs_weak}. The key step again is to prove \eqref{eq:ind} but this time with the deterministic choice $\frX_{\caI}=C^{\absv{\caI}}\absv{\log \eta}^{\absv{\caI}}$; we state this in the next proposition. Recall the definition of $\beta$ from \eqref{eq:def_beta}.
\begin{prop}\label{prop:repel_weak}
	Let $k,N\in\N$ with ${\Cb 2\leq k\leq N}$, $X$ be an $(N\times N)$ regular 
	\Cb complex or real \nc matrix, and $A\in\C^{N\times N}$. 
	Then there exists a constant $C\equiv C(k,\frb)>0$ such that 
	\beq\label{eq:repel_weak}
		\P[\Xi_{k}^{(\caI)}]\leq C\absv{\log \eta}(N\eta)^{\beta k}\frac{1}{N-\absv{\caI}}\sum_{i\in\bbrktt{N}\setminus\caI}\P[\Xi_{k}^{(\caI\cup\{i\})}] \qquad \forall \eta\in[0,N^{-1}],
	\eeq
	holds for all $\caI\subset\bbrktt{N}$ with $\absv{\caI}\leq k-1$. {\Cb  If $X$ is complex, the same result also holds
	for  $k=1$.}
\end{prop}
Note that Proposition \ref{prop:repel_weak}, in contrast to Proposition \ref{prop:repel}, does not assume that $A$ is real even if $X$ is. As we will see below, complex $A$ poses additional technicalities in its proof when $X$ is real. Given Proposition \ref{prop:repel_weak}, Theorem \ref{thm:s_weak} follows directly by substituting $\eta= s/N$ and  $\frX_{\caI}=C^{\absv{\caI}}\absv{\log \eta}^{\absv{\caI}}$ into \eqref{eq:ind_appl} where $C$ is the constant from \eqref{eq:repel_weak}. Thus we move on to the proof of Proposition \ref{prop:repel_weak}.

\begin{proof}[Proof of Proposition \ref{prop:repel_weak}]
	We make two modifications to the proof Proposition~\ref{prop:repel} in order to prove \eqref{eq:repel_weak}. Firstly in \eqref{eq:repel_G}, we raise \eqref{eq:repel_G_det} to the $(\beta k/2)$-th power instead of $(\beta k)$, and we do not smuggle in the factor $(1+N\lambda_{m}^{(j-1)})^{n}$ in the following steps. {\Cb Notice here that we need to 
	take $k\geq 2$ in the real case, so that we may apply Jensen's inequality in \eqref{eq:Jensen} with the power $k/2$.} Secondly in \eqref{eq:Schur_minor_strong}, we keep the first term $\eta$ and only use the first $k$ coordinates of $\bsw^{(j)}$ instead of the first $(2k+1)$. After these modifications, it suffices to prove
	\beq\label{eq:repel_tech_weak_0}
	\lone_{\Xi_{k}^{(j)}}\E^{(j)}\left(\eta+\sum_{i\in\bbrktt{k}}c_{i}^{(j)}\absv{w_{i}^{(j)}}^{2}\right)^{-\beta k/2}\leq C(N\eta)^{\beta k/2}\absv{\log\eta}
	\eeq
	for a constant $C\equiv C(k,\frb)>0$.
	
	Recall from \eqref{eq:c_size} that $c_{i}^{(j)}\geq 1/(2N\eta)$ for all $i\in\bbrktt{k}$ on the event $\Xi_{k}^{(j)}$. Thus we have
	\beq\label{eq:repel_tech_weak_1}
		\lone_{\Xi_{k}^{(j)}}\E^{(j)}\left(\eta+\sum_{i\in\bbrktt{N}}c_{i}^{(j)}\absv{w_{i}^{(j)}}^{2}\right)^{-\beta k/2}\leq C(N\eta)^{\beta k/2} \lone_{\Xi_{k}^{(j)}}\E^{(j)}\left(N\eta^{2}+\sum_{i\in\bbrktt{k}}\absv{w_{i}^{(j)}}^{2}\right)^{-\beta k/2},
	\eeq
	where $C>0$ depends only on $k$. Therefore \eqref{eq:repel_tech_weak_0} follows once we prove
	\beq\label{eq:repel_tech_weak}
		\E^{(j)}\left(N\eta^{2}+\sum_{i\in\bbrktt{k}}\absv{w_{i}^{(j)}}^{2}\right)^{-\beta k/2}\leq C \absv{\log \eta}
	\eeq
	for a constant $C$ depending only on $k$ and $\frb$. We state this estimate as the next lemma. To simplify the presentation we define $P_{k}:\C^{N}\to\C^{k}$ to be the projection onto the first $k$ components, that is,
	\beq\label{eq:def_P}
	P_{k}\deq \begin{pmatrix} I_{k} & \vert & O \end{pmatrix}\in\C^{(k\times N)}.
	\eeq
	\begin{lem}\label{lem:tech_weak}
		Let $k,N\in\N$ with $k\leq N$, $\bbK=\R$ or $\C$, and $\bsb\in\K^{N}$ be a random vector with independent components satisfying the same assumptions as in Lemma \ref{lem:tech}. Then there exists a constant $\frC_{2}\equiv \frC_{2}(k,\frb)>0$ such that the following holds for any $c_{0}\in(0,1/2)$, $\bsa\in \C^{N}$, and a unitary matrix $U\in\C^{N\times N}$;
		\beq\label{eq:tech_weak}
			\E\left(c_{0}+\norm{P_{k}U(\bsb+\bsa)}^{2}\right)^{-\beta k/2}\leq \frC_{2}(1+\absv{\log c_{0}}).
		\eeq
	\end{lem}
	Note that when $\bbK=\R$, the vector $PU(\bsb+\bsa)$ in \eqref{eq:tech_weak} involves both real and complex matrices (or vectors) in contrast to \eqref{eq:tech}. Hereafter, any binary operation concerning both real and complex matrices treats $\R$ as canonically embedded in $\C$.
	
	As in Section \ref{sec:s_strong}, we postpone the proof of Lemma \ref{lem:tech_weak} and first use this lemma to conclude Proposition \ref{prop:repel_weak}. We choose $c_{0}=N\eta^{2}$, $U=U^{(j)}$, $\bsb=\sqrt{N}X\adj \bse_{j}$, and $\bsa=\sqrt{N}A\adj\bse_{j}$. With these choices we have
	\beq
		P_{k}U(\bsb+\bsa)=\sqrt{N}P_{k}U^{(j)}(X+A)\adj\bse_{j}=PU^{(j)}\bsb^{(j)}=P\bsw^{(j)},
	\eeq
	so that the left-hand side of \eqref{eq:tech_weak} is exactly that of \eqref{eq:repel_tech_weak}. Following similar arguments as in the proof of Proposition \ref{prop:repel_weak}, these choices satisfy all assumptions of Lemma \ref{lem:tech_weak} so that
	\beq
		\E^{(j)}\left(N\eta^{2}+\sum_{i\in\bbrktt{k}}\absv{w_{i}^{(j)}}^{2}\right)^{-\beta k/2}\leq \frC_{2}(1+\absv{\log (N\eta^{2})})\leq 10\frC_{2}\absv{\log \eta},
	\eeq
	where we used $\eta\in[0,N^{-1}]$. This proves \eqref{eq:repel_tech_weak}, concluding the proof of Proposition \ref{prop:repel_weak}.
\end{proof}

Next we will present the proof of Lemma \ref{lem:tech_weak}. Recall that in Lemma \ref{lem:tech_weak} the unitary matrix $U$ is not real orthogonal even if $X$ is real, in contrast to Lemma \ref{lem:tech}. Hence the proof of \eqref{eq:repel_tech_weak} for the real case does not simply follow from the complex case just after some changes in the exponents. Nonetheless, we still expect that the vector $P_{k}U\bsb$ carries at least half the degrees of freedom compared to the complex case, regardless of $U$. As before, these will be used to regularize the potentially singular expectation in \eqref{eq:tech_weak} even if $c_{0}$ is very small. In other words, we will construct a `projection' $R:\C^{k}\to\R^{k}$ so that $RP_{k}U\bsb\in\R^{k}$ has a continuous distribution. The next lemma is used to construct such an $R$.

	\begin{lem}\label{lem:s_R_C_R}
		Let $U$ be an $N\times N$ complex unitary matrix, $k\in\bbrktt{N}$, $P_{k}$ be given by \eqref{eq:def_P}, and define
		\beq\label{eq:def_Q}
			Q\deq \caJ[P_{k}]\caJ[U]\begin{pmatrix} I_{N} \\ O\end{pmatrix}\in\R^{2k\times N}.
		\eeq
		Then the singular values $m_{1}\leq\cdots\leq m_{2k}=\norm{Q}$ of $Q$ satisfy 
		\beq
			m_{2k}\leq 1,\qquad m_{k+1}\geq \frac{1}{\sqrt{k+1}}.
		\eeq
	\end{lem}
The proof of Lemma \ref{lem:s_R_C_R} is postponed to the end of this section, and we proceed to prove Lemma \ref{lem:tech_weak}.
\begin{proof}[Proof of Lemma \ref{lem:tech_weak}]
	We first prove the complex case, $\beta=2$. One can easily see that the vector $\wt{\bsb}\deq \bsb+\bsa$ has independent components whose densities are bounded by $\frb$, exactly when $\bsb$ does. Since $\norm{PU(\bsb+\bsa)}=\norm{PU\wt{\bsb}}$, we may replace $PU(\bsb+\bsa)$ by $PU\bsb$ in the left-hand side of \eqref{eq:tech_weak}, that is, it suffices to prove
	\beq\label{eq:tech_weak_C}
		\E(c_{0}+\norm{PU\bsb}^{2})^{-k}\leq C(1+\absv{\log c_{0}}).
	\eeq
	Note that this argument using a complex shift $\bsa$ only applies to the case $\K=\C$, and later for the real case we will need to use Lemma \ref{lem:s_R_C_R} instead.
	
	Now we apply similar arguments as in Lemma \ref{lem:tech}; first, define $\mu_{k}$ to be the distribution of $PU\bsb$; for a test function $f:\C^{k}\to \C$ we write
	\beqs
		\int_{\C^{N}}f\dd\mu\deq \int_{\C^{N}}f(PU\bsb)h(\bsb)\dd^{2N}\bsb
	\eeqs
	where $h(\bsb)$ was defined in \eqref{eq:dens}. Then we have
	\beq\label{eq:tech_weak_C_prf}
	\begin{aligned}
		\E\left(c_{0}+\norm{PU\bsb}^{2}\right)^{-k}=&\int_{\C^{k}}\left(c_{0}+\norm{\bsw}^{2}\right)^{-k}\dd\mu_{k}(\bsw)
		\leq 1+\int_{\C^{k}}\frac{\lone(\norm{\bsw}\leq 1)}{(c_{0}+\norm{\bsw}^{2})^{k}}\dd\mu_{k}(\bsw)\\
		\leq& 1+C\norm{\mu_{k}}_{\infty}\int_{\C^{k}}\frac{\lone(\norm{\bsw}\leq 1)}{(c_{0}+\norm{\bsw}^{2})^{k}}\dd^{2k}\bsw
		\leq1+C\norm{\mu_{k}}_{\infty}\int_{0}^{1}\frac{s^{2k-1}}{(c_{0}+s^{2})^{k}}\dd s,
	\end{aligned}
	\eeq
where we used in the first inequality that the integral over $\norm{\bsw}>1$ is less than $1$. Recall that this $\mu_{k}$ is exactly the same as that in the proof of Lemma \ref{lem:tech}, so that $\norm{\mu_{k}}_{\infty}\leq C(k,\frb)$. Since the last integral in \eqref{eq:tech_weak_C_prf} is bounded by $\absv{\log c_{0}}$, therefore we have
\beq
	\E\left(c_{0}+\norm{PU\bsb}^{2}\right)^{-k}\leq C(1+\absv{\log c_{0}}).
\eeq
This proves \eqref{eq:tech_weak_C}, thus concludes Lemma \ref{lem:tech_weak} for the complex case.

Now we move on to the real case. We first notice that 
	\beq\label{eq:PU=Q}
		\caJ[P_{k}U\bsb]=\caJ[P_{k}U]\caJ[\bsb]=\caJ[P_{k}]\caJ[U]\begin{pmatrix} I_{N} \\ O\end{pmatrix}\bsb=Q\bsb\in\R^{2k},
	\eeq
	where $Q$ was defined in \eqref{eq:def_Q}. Consider the singular value decomposition 
	\beq\label{eq:Q_decomp}
		Q=\sum_{i=1}^{2k}m_{i}\bsx_{i}\bsy_{i}\tp,\qquad \bsx_{i}\in\R^{2k},\bsy_{i}\in\R^{N},
	\eeq
	and define
	\beq\label{eq:def_R}\begin{aligned}
		&R\deq \sum_{i=k+1}^{2k}\bsy_{i}\bsx_{i}\tp,&
		&\wt{P}\deq \sum_{i=k+1}^{2k}\bsy_{i}\bsy_{i}\tp,&
		&V\deq \sum_{i=1}^{k}\bse_{i}\bsy_{i+k}\tp\in\R^{k\times N},&
		&D\deq\diag(m_{k+1},\cdots,m_{2k}).
	\end{aligned}\eeq
	Note that $\wt{P}$ is an orthogonal projection on $\R^{N}$, $V$ is a unitary map between $\R^{k}$ and $\spann(\bsy_{k+1},\cdots\bsy_{2k})$, and that
	\beq
		VRQ=\sum_{i=k+1}^{2k}m_{i}\bse_{i}\bsy_{i}\tp=DV\wt{P}.
	\eeq
	Hence, using $\norm{VR}\leq 1$ and $\norm{\caJ[\bsv]}=\norm{\bsv}$ for any complex vector $\bsv$, we obtain
	\beq\begin{aligned}
		\norm{P_{k}U(\bsb+\bsa)}=\norm{\caJ[P_{k}U(\bsb+\bsa)]}\geq \norm{VR\caJ[P_{k}U(\bsb+\bsa)]}	\\
		=\norm{DV(\wt{P}\bsb+V\tp D^{-1}VR\caJ[P_{k}U\bsa])}=:\norm{DV(\wt{P}\bsb+\wt{\bsa})},
	\end{aligned}\eeq
	where we abbreviated 
	\beq\label{eq:def_a}
		\wt{\bsa}=V\tp D^{-1}VR\caJ[P_{k}U\bsa]\in\spann(\bsy_{k+1},\cdots,\bsy_{2k}).
	\eeq
	Thus we have that
	\beq
		(c_{0}+\norm{PU(\bsb+\bsa)}^{2})^{-k/2}\leq (c_{0}+\norm{DV(\wt{P}\bsb+\wt{\bsa})}^{2})^{-k/2}.
	\eeq

Next, we define $\nu_{k}$ to be the distribution (on $\R^{k}$) of $V(\wt{P}\bsb+\wt{\bsa})$, that is,
	\beq
	\int_{\R^{k}}f(\bsx)\dd\nu_{k}(\bsx)=\int_{\R^{N}}f(V(\wt{P}\bsb+\wt{\bsa}))h(\bsb)\dd^{N}\bsb
	\eeq
	where $h(\bsb)$ was defined in \eqref{eq:dens}. Then we may follow the same lines as in \eqref{eq:tech_weak_C_prf} to write
	\beq\label{eq:tech_weak_R_prf}\begin{aligned}
		\E(c_{0}+\norm{PU(\bsb+\bsa)}^{2})^{-k/2}
		\leq & \E(c_{0}+\norm{DV(\wt{P}\bsb+\wt{\bsa})}^{2})^{-k/2}	\\
		=&\int_{\R^{k}}(c_{0}+\norm{D\bsx}^{2})^{-k/2}\dd\nu_{k}(\bsx)	
		\leq 1+\frac{\norm{\nu_{k}}_{\infty}}{(k+1)^{k/2}}\int_{0}^{\infty}\frac{s^{k-1}}{(c_{0}+s^{2})^{k/2}}\dd s
	\end{aligned}\eeq
where we used in the last inequality that $\det D\geq m_{k+1}^{k}\geq (k+1)^{k/2}$. 

Since the last integral on the right-hand side of \eqref{eq:tech_weak_R_prf} is comparable to $\absv{\log c_{0}}$, it suffices to prove $\norm{\nu_{k}}_{\infty}\leq C(k,\frb)$ in order to conclude \eqref{eq:tech_weak} for the real case. To prove this, we apply Lemma \ref{lem:dens_proj} with the choices $\wt{\bsb}=\bsb$ and $\wt{P}$ in \eqref{eq:def_R}. As a result, the densities of $\wt{P}\bsb$ and hence $(\wt{P}\bsb+\wt{\bsa})$ on $\spann(\bsy_{k+1},\cdots,\bsy_{2k})$ are both bounded by $C(k,\frb)$. Since $V$ is an isometry  between $\spann(\bsy_{k+1},\cdots,\bsy_{2k})$ and $\R^{k}$, we immediately find $\norm{\nu_{k}}_{\infty}\leq C(k,\frb)$. This completes the proof of Lemma \ref{lem:tech_weak} modulo Lemma \ref{lem:s_R_C_R}.
\end{proof}

\begin{proof}[Proof of Lemma \ref{lem:s_R_C_R}]
	First of all, $m_{2k}=\norm{Q}\leq 1$ follows directly from 
	\beq
		\norm{Q\bsv}=\norm{\caJ[P_{k}U\bsv]}=\norm{P_{k}U\bsv}\leq \norm{\bsv},\qquad \forall\bsv\in\R^{N}.
	\eeq
	To prove $m_{k+1}\geq (k+1)^{-1/2}$, we use the definition of $\caJ$ to write the matrix $Q$ in \eqref{eq:def_Q} as
	\beq\label{eq:Q_expa}
		Q=\begin{pmatrix} P_{k}\Re[U] \\ -P_{k}\Im[U]\end{pmatrix}\in\R^{2k\times N},
	\eeq
	so that, using the same singular decomposition as in \eqref{eq:Q_decomp},
	\beq\label{eq:def_mi}
	Q Q\tp=\begin{pmatrix} P_{k}\Re[U] \\ -P_{k}\Im[U]\end{pmatrix}\begin{pmatrix}\Re[U]\tp P\tp & -\Im[U]\tp P\tp\end{pmatrix}=\sum_{i=1}^{2k}m_{i}^{2}\bsx_{i}\bsx_{i}\tp.
	\eeq
	Then from $\Re[UU\adj]=\Re[I]=I$ we have
	\beq\begin{aligned}
		\sum_{i=1}^{2k}m_{i}^{2}&=\Tr \begin{pmatrix} P\Re[U] \\ -P\Im[U] \end{pmatrix}\begin{pmatrix} \Re[U]\tp P\tp & -\Im[U]\tp P\tp\end{pmatrix}	\\
		&=\Tr P(\Re[U]\Re[U]\tp+\Im[U]\Im[U]\tp)P\tp
		=\Tr P \Re[UU\adj]P\tp=\Tr PP\tp=k.
	\end{aligned}\eeq
	Hence for each $\ell\in\bbrktt{2k}$ we have
	\beq
	k=\sum_{i=1}^{2k}m_{i}^{2}\leq \ell m_{\ell}^{2}+(2k-\ell)m_{2k}^{2}\leq \ell m_{\ell}^{2}+(2k-\ell),
	\eeq
	where we used in the last inequality that $m_{2k}\leq 1$. Thus for all $\ell\geq k$ we have
	\beq\label{eq:s_R_C_R_k}
	m_{\ell}\geq \sqrt{\frac{\ell-k}{\ell}},\qquad \text{in particular}\quad m_{k+1}\geq \frac{1}{\sqrt{k+1}}.
	\eeq
	This completes the proof of Lemma \ref{lem:s_R_C_R}.
\end{proof}

\begin{rem}\label{rem:strong_vs_weak}
	As one can see from the proofs, the difference between Theorem \ref{thm:s_weak} and \ref{thm:s_strong} is rooted in that between \eqref{eq:repel_tech} and \eqref{eq:repel_tech_weak_0}. Comparing \eqref{eq:repel_tech_weak_0} to \eqref{eq:repel_tech}, we find that the left-hand sides concern different inverse powers of the same random variable $\sum_{i} c_{i}^{(j)} \absv{w_{i}^{(j)}}^{2}$ modulo irrelevant summands. On the other hand, the right-hand side of \eqref{eq:repel_tech_weak_0} has a factor of $(N\eta)^{\beta k/2}$, which is small, whereas that of \eqref{eq:repel_tech} is (roughly) $O(1)$. 
	
	Here we briefly explain how these two different estimates for the inverse powers of the same random variable arise. In fact, when $N\eta\ll1$, \eqref{eq:repel_tech_weak_0} better represents the natural size of this random variable. To see this, we recall \eqref{eq:c_size} and write for general $\bsw\in\C^{N}$ that, on the event $\Xi_{k}^{(j)}$,
	\beq\label{eq:heur_weak}
	\frac{1}{2N\eta}\sum_{i=1}^{k}\absv{w_{i}}^{2}\leq \sum_{i=1}^{k}c_{i}^{(j)}\absv{w_{i}}^{2} \leq\sum_{i\in\bbrktt{N}}c_{i}\absv{w_{i}}^{2}\leq \frac{1}{N\eta}\sum_{i=1}^{k}\absv{w_{i}}^{2} +\frac{1}{N}\sum_{i>k}\frac{\eta}{(\lambda_{i}^{(j)})^{2}+\eta^{2}}\absv{w_{i}}^{2}.
	\eeq
	Typically we take $w_{i}$ to be $O(1)$ random variables with continuous joint distribution, and hence the second term on the right-hand side of \eqref{eq:heur_weak} is roughly of the same size as $\brkt{G(\ii\eta)}$ which is $O(1)$ (see Lemma \ref{lem:gap}). Thus essentially the first $k$ summands determine the size of $\sum_{i}c_{i}^{(j)}\absv{w_{i}^{(j)}}^{2}$, which is in turn comparable to $(N\eta)^{-1}\norm{P\bsw}^{2}$ recalling that the relevant regime is $N\eta\ll 1$. 
	
	Therefore, in effect, one can only use these $\beta k$ degrees of freedom (from $k$ variables in $\bbK$) upon estimating the $m$-th negative moment to regularize the potential singularity. This essentially leads to integrals of the form
	\beq
		\int_{\bbK^{k}}\frac{\lone(\norm{\bsw}\leq 1)}{\norm{\bsw}^{2m}}\dd^{\beta k}\bsw,
	\eeq
	which is finite only if $m<\beta k/2$. When $m$ is exactly $\beta k/2$ this integral has logarithmic singularity, which is responsible for the logarithmic corrections in Theorem \ref{thm:s_weak}. On the other hand, for any $m>\beta k/2$, the \emph{a priori} uncontrolled singular values $\{\lambda_{i}^{(j)}:i>k\}$ affect the $m$-th negative moment of $\sum_{i}c_{i}^{(j)}\absv{w_{i}^{(j)}}^{2}$ via $c_{i}^{(j)}$. The main point of Theorem \ref{thm:s_strong} is that $\lambda_{i}^{(j)}$'s, and hence $c_{i}^{(j)}$'s, for $i>k$ are controlled if replace $A$ with $A-z$ for $z$ in the bulk spectrum and thus effectively contribute to regularizing the integral.
\end{rem}

\section{Proof of Theorem \ref{thm:s_real_weak}}\label{sec:s_R_weak}
{\Cb Now we consider the singular values of real regular matrices, but shifted by genuinely complex matrix $A\in\C^{N\times N}$. More precisely, we show that \eqref{eq:s_R_weak} extends to $k=1$ with the improved rate $s^{2}$ if $\Im[A]$ is strictly positive definite.}

Recall, from the proofs of Theorems \ref{thm:s_weak} and \ref{thm:s_strong} for the complex case with $k=1$, that the scale $s^{2}$ is due to the fact that the random variable $w^{(1)}_{1}=(\bsv_{1}^{(1)})\adj\bsb$ is genuinely complex; more specifically, that the distribution $\mu_{1}$ of $w^{(1)}$ on $\C$ has a bounded density. Here we prove the same statement for the real case when the shift $A$ is genuinely complex. 

Before presenting the proof, we briefly explain which parts of the proof of $\norm{\mu_{1}}_{\infty}=O(1)$ have to be modified compared to those in Sections \ref{sec:s_strong} and \ref{sec:s_weak}. If $A$ is real, then $\bsv_{1}^{(1)}$, the null vector of $J^{1}(X+A)$, is also real and thus $w^{(1)}_{1}$ is real, so that its density $\mu_{1}$ is singular if viewed as a density on $\C$. If the shift $A$ is complex, then $\bsb=\sqrt{N}(X+A)\adj\bse_{1}$ becomes complex but only due to an irrelevant shift. Hence the regularity of $\mu_{1}$ should come from the fact that $\bsv_{1}^{(1)}$ is genuinely complex. Furthermore, whatever estimate we obtain for $\norm{\mu_{1}}_{\infty}$, it should deteriorate as $A$ tends to a real matrix; we have already seen in the proof of Theorem \ref{thm:s_strong} that $w_{1}^{(1)}$ is real when $A$ is. 

The next lemma summarizes the technical input we need for the proof of Theorem \ref{thm:s_real_weak}. It proves that $\bsv^{(1)}_{1}$ is a genuinely complex vector, quantitatively with an explicit dependence on $\Im A=(A-\ol{A})/(2\ii )$.
\begin{lem}\label{lem:real_1}
	Let $Y,B\in\R^{N\times N}$ and $\bsv\in\C^{N}$ be a unit null vector of $J^{(1)}(Y+\ii B)$. Then we have
	\beq\label{eq:real_1}
	\inf_{\theta\in[0,2\pi]}\norm{\re[\e{\ii\theta}\bsv]}^{2}\geq \frac{1}{5}\lambda_{1}(B)^{2}\frac{\norm{J^{(1)}Y\bsw}^{2}}{(\norm{J^{(1)}Y}+\norm{B})^{4}},
	\eeq
	where $\lambda_{1}(B)$ is the smallest singular value of $B$ and $\bsw\in\R^{N}$ is the unit null vector of $J^{(1)}B$.
\end{lem}
Note that zero is an eigenvalue of $B\adj (J^{(1)})\adj J^{(1)}B$ and that, by Cauchy interlacing theorem, its multiplicity is exactly one provided $\lambda_{1}(B)>0$. Hence $\bsw$ is uniquely determined if $\lambda_{1}(B)>0$.
We postpone the proof of Lemma \ref{lem:real_1} to the end of this section and move on to that of Theorem \ref{thm:s_real_weak}.
\begin{proof}[Proof of Theorem \ref{thm:s_real_weak}]
	In order to prove \eqref{eq:s_real_weak}, we follow the proof of Theorem \ref{thm:s_weak} in the complex case (hence $\beta=2$) for $j=k=1$. One can easily find that all the arguments until Lemma \ref{lem:tech_weak} remain intact, which is replaced by the following lemma;
	\begin{lem}\label{lem:tech_real_weak}
		Let $\bsb=(b_{i})_{i\in\bbrktt{N}}\in\R^{N}$ be a random vector with independent components whose densities are bounded by $\frb>0$. Then there exists a constant $\frC_{3}\equiv \frC_{3}(\frb)>0$ such that the following holds for all $c_{0}\in(0,1/2)$ and deterministic vectors $\bsv,\bsa\in\C^{N}$ with $\norm{\bsv}=1$;
		\beq
			\E(c_{0}+\absv{\bsv\adj(\bsb+\bsa)}^{2})^{-1}\leq \frC_{3}\left(1+\left(\min_{\theta\in[0,2\pi]}\norm{\re[\e{\ii\theta}\bsv]}\right)^{-1}\absv{\log c_{0}}\right).
		\eeq
	\end{lem}
The proof of Lemma \ref{lem:tech_real_weak} is presented after completing that of Theorem \ref{thm:s_real_weak}, and we proceed assuming its validity. If we replace Lemma \ref{lem:tech_weak} by \ref{lem:tech_real_weak}, we obtain
	\beq\label{eq:s_real_weak_prf_0}
		\P[\Xi_{1}]\leq (N\eta)^{2}\max_{i\in\bbrktt{N}}\E(N\eta^{2}+\absv{w_{1}^{(\{i\})}}^{2})^{-1}
		\leq 10\frC_{3}(N\eta)^{2}\max_{i\in\bbrktt{N}}\left(\absv{\log\eta}\E\left(\min_{\theta\in[0,2\pi]}\norm{\re[ \e{\ii\theta}\bsv_{1}^{(\{i\})}]}\right)^{-1}\right),
	\eeq
	so that it only remains to estimate the expectation on the right-hand side. Also, as before, we take $i=1$ without loss of generality since the final result is uniform in $i\in\bbrktt{N}$. To this end, we apply Lemma \ref{lem:real_1} with the choices $Y=X+\Re A$ and $B=\Im A$. As a result, we obtain
	\beq\label{eq:s_real_weak_prf}\begin{aligned}
	\E\left(\min_{\theta\in[0,2\pi]}\norm{\re[ \e{\ii\theta}\bsv^{(1)}]}\right)^{-1}\leq &\frac{C}{\lambda_{1}(\Im A)}\E\frac{(\norm{J^{(1)}(X+\Re A)}+\norm{\Im A})^{2}}{\norm{J^{(1)}(X+\Re A)\bsw}}	\\
	\leq &\frac{C}{\lambda_{1}(\Im A)}\Norm{\norm{J^{(1)}(X+\Re A)\bsw}^{-1}}_{2}(\Norm{\norm{X}^{2}}_{2}+\norm{A}^{2})
	\end{aligned}\eeq
	for a numeric constant $C>0$, where $\bsw$ is the unit null vector of $J^{(1)}\Im A$.
	
	We next handle the first factor on the right-hand side of \eqref{eq:s_real_weak_prf}. We claim that $\E\norm{J^{(1)}(X+\Re A)\bsw}^{-r}\leq C$ as long as $N-1>r$. To this end, we estimate the lower tail of $\norm{J^{(1)}(X+\re A)\bsw}$ by its Laplace transform;
	\beq\begin{aligned}
		\P\left[\norm{J^{(1)}(X+\Re A)\bsw}^{2}\leq t^{2}\right]=&\P\left[-\frac{1}{t^{2}}\sum_{j=2}^{N}\left(\bse_{j}\tp(\sqrt{N}(X+\Re A))\bsw\right)^{2}\geq -N\right]	\\
		\leq &\e{N}\prod_{j=2}^{N}\E \exp\left(-\frac{1}{t^{2}}\left(\bse_{j}\tp(\sqrt{N}(X+\Re A))\bsw\right)^{2}\right).
	\end{aligned}\eeq
	Then for each $j\in\bbrktt{2,N}$, since $\norm{\bsw}=1$, the random variable $\bse_{j}\tp(\sqrt{N}X+\Re A)\bsw$ has a density bounded by $C\frb$ by Lemma \ref{lem:dens_proj}. This gives 
	\beq
		\P[\norm{J^{(1)}(X+\Re A)\bsw}\leq t]\leq (C_{0}t)^{N-1},
	\eeq
	for a constant $C_{0}$ depending only on $\frb$, which in turn implies
	\beq\begin{aligned}
		\E\norm{J^{(1)}(X+\Re A)\bse_{1}}^{-r}\leq& C_{0}^{r}+\E\norm{J^{(1)}(X+\Re A)\bse_{1}}^{-r}\lone_{\norm{J^{(1)}(X+\Re A)\bse_{1}}\leq C_{0}^{-1}}	\\
		=&C_{0}^{r}+\int_{C_{0}^{r}}^{\infty}\P[\norm{J^{(1)}(X+\Re A)\bse_{1}}\leq x^{-1/r}]\dd x
		\leq C_{0}^{r}+C_{0}^{r}\int_{1}^{\infty} x^{-(N-1)/r}\dd x.
	\end{aligned}\eeq
	Taking $r=2$ and recalling $N\geq 4>r+1$ we have
	\beq
		\Norm{\norm{J^{(1)}(X+\Re A)\bse_{1}}^{-1}}_{2}\leq C(\frb).
	\eeq
	Substituting this into \eqref{eq:s_real_weak_prf} and then \eqref{eq:s_real_weak_prf_0} proves \eqref{eq:s_real_weak}. Given \eqref{eq:s_real_weak}, \eqref{eq:s_real} follows directly from $\E\norm{X}^{4}\leq C(\frm)$, which is a classical result that can be proved following \cite{Bai-Krishnaiah-Yin1988}. 
\end{proof}
\begin{proof}[Proof of Lemma \ref{lem:tech_real_weak}]
	We reuse the same notations as in the proof of Lemma \ref{lem:tech_weak} for the real case. Namely, for a suitable unitary matrix $U$ we may write $\bsv=U\adj P_{1}\adj$ with $P_{1}$ defined in \eqref{eq:def_P}, so that
	\beq
		\begin{pmatrix} \re[\bsv\adj\bsb]\\ \im[\bsv\adj\bsb]\end{pmatrix}=\caJ[P_{1}U\bsb]=Q\bsb,
	\eeq
	with the same $Q$ as in Lemma \ref{lem:s_R_C_R} for $k=1$. Again we write the same singular value decomposition of $Q$ as in \eqref{eq:Q_decomp}, but here we replace the matrices in \eqref{eq:def_R} by the following, that use both $m_{1}$ and $m_{2}$:
	\beq\label{eq:def_R'}\begin{aligned}
		&R'\deq\sum_{i=1,2}\bsy_{i}\bsx_{i}\tp,&
		&\wt{P}'\deq\sum_{i=1,2}\bsy_{i}\bsy_{i}\tp,&
		&V'\deq \sum_{i=1,2}\bse_{i}\bsy_{i}\tp,
		&D'\deq \diag(m_{1},m_{2}).
	\end{aligned}\eeq
	Now we may follow similar lines as in the proof of Lemma \ref{lem:tech_weak}: We define $\wt{\bsa}'$ by the exact same formula as in \eqref{eq:def_a}, but with matrices in \eqref{eq:def_R} replaced by those in \eqref{eq:def_R'} so that $\wt{\bsa}'\in\spann(\bsy_{1},\bsy_{2})$. Then, denoting the distribution of $V'(\wt{P}'\bsb+\wt{\bsa}')$ by $\nu'_{1}$, we have
	\beq\label{eq:tech_weak_real_prf}\begin{aligned}
		\E(c_{0}+\absv{\bsv\adj(\bsb+\bsa)}^{2})^{-1}=&\E(c_{0}+\norm{P_{1}U(\bsb+\bsa)}^{2})^{-1}\leq \E(c_{0}+\norm{D'V'(\wt{P}'\bsb+\wt{\bsa}')}^{2})^{-1}	\\
		=&\int_{\R^{2}}(c_{0}+\norm{D'\bsx}^{2})^{-1}\dd\nu'_{1}(\bsx)\leq 1+C\frac{\norm{\nu'_{1}}_{\infty}}{\det D'}\absv{\log c_{0}}.
	\end{aligned}\eeq
	By the exact same argument as in the proof of Lemma \ref{lem:tech_weak} we have $\norm{\nu'_{1}}_{\infty}\leq C(\frb)$, thus it only remains to estimate $(\det D')^{-1}=(m_{1}m_{2})^{-1}$.
	
	Applying Lemma \ref{lem:s_R_C_R} we immediately have $m_{1}\geq 1/\sqrt{2}$, hence it suffices to estimate $m_{2}$ from below. Writing out the definition of $Q$ in this case, we have
	\beq
		Q=\begin{pmatrix}
			\re[\bsv]\tp \\ -\im[\bsv]\tp
		\end{pmatrix}.
	\eeq
	Then we immediately find
	\beq
	m_{2}=\min_{\bss\in\R^{2},\norm{\bss}=1}\norm{Q\tp \bss}=\min_{\theta\in[0,2\pi]}\Norm{\begin{pmatrix}
			\re[\bsv]& -\im[\bsv]
		\end{pmatrix}\begin{pmatrix}\cos \theta \\\sin\theta\end{pmatrix}}
	=\min_{\theta\in[0,2\pi]}\norm{\re[\e{\ii\theta}\bsv]}.
	\eeq
	We thus obtain
	\beq
		\det D=m_{1}m_{2}\geq \sqrt{2}\min_{\theta\in[0,2\pi]}\norm{\re[\e{\ii\theta}\bsv]},
	\eeq
	and plugging this into \eqref{eq:tech_weak_real_prf} concludes the proof of Lemma \ref{lem:tech_real_weak}.
\end{proof}

\begin{proof}[Proof of Lemma \ref{lem:real_1}]
	Throughout the proof, we abbreviate $J\deq J^{(1)}$ and $\lambda\deq \lambda_{1}(B)$ for simplicity.
	Since $\bsv$ is a null vector of $J(Y+\ii B)$,  we get from taking the real part of $\e{\ii\theta}J(Y+\ii B)\bsv=0$ that
	\begin{align}\label{eq:real_v0}
		0=JX\re[\e{\ii\theta}\bsv]-JB\im [\e{\ii\theta}\bsv]=JX\re[\e{\ii\theta}\bsv]-JBP_{\bsw^{\perp}}\im[\e{\ii\theta}\bsv]
	\end{align}
	for any $\theta\in\R$, where $P_{\bsw^{\perp}}\in\R^{N\times N}$ is the projection onto the orthogonal complement of $\bsw$. Then, writing the smallest nonzero singular value of $JB$ by $\lambda_{1}(JB)$, \eqref{eq:real_v0} implies 
	\beq\label{eq:real_1_crude}
	\begin{aligned}
	\lambda^{2}(\norm{\im[\e{\ii\theta}\bsv]}^{2}-\absv{\im[\e{\ii\theta}\bsw\tp\bsv]}^{2})=&\lambda^{2}\norm{P_{\bsw^{\perp}}\im[\e{\ii\theta}\bsv]}^{2}	
	\leq \lambda_{1}(JB)\norm{P_{\bsw^{\perp}}\im[\e{\ii\theta}\bsv]}^{2}	\\
	\leq&  \norm{JBP_{\bsw^{\perp}}\im[\e{\ii\theta}\bsv]}^{2}\leq \norm{JY}^{2}\norm{\re[\e{\ii\theta}\bsv]}^{2},
	\end{aligned}\eeq
	where we used $\lambda\leq \lambda_{1}(JB)$ from Cauchy interlacing theorem in the first inequality. This in turn gives
	\beq\label{eq:real_1_prf}
	\lambda^{2}(1-\absv{\bsw\adj\bsv}^{2})	
	\leq (\norm{JY}^{2}+\lambda^{2})\norm{\re[\e{\ii\theta}\bsv]}^{2}\leq (\norm{JY}+\norm{B})^{2}\norm{\re[\e{\ii\theta}\bsv]}^{2}.
	\eeq
	Comparing \eqref{eq:real_1} with \eqref{eq:real_1_prf} and noticing the first factor on the right-hand side of \eqref{eq:real_1_prf}, since $\theta$ was arbitrary, in order to prove \eqref{eq:real_1} it only remains to show
	\beq\label{eq:v01}
	1-\absv{\bsw\adj\bsv}^{2}\geq\frac{1}{5}\frac{\norm{JY\bsw}^{2}}{(\norm{JY}+\norm{B})^{2}}.
	\eeq
	
	Next, we prove \eqref{eq:v01}. First of all, since $\bsv$ is a null vector of $J(Y+\ii B)$, we have the following inequality of positive semi-definite matrices for all $\wt{\eta}>0$;
	\beq
		\bsv\bsv\adj\leq \frac{\wt{\eta}^{2}}{\wt{\eta}^{2}+(Y+\ii B)\adj J\adj J(Y+\ii B)}.
	\eeq
	We then take $\wt{U}\in\R^{N\times N}$ to be a real orthogonal matrix with $\wt{U}\bse_{1}=\bsw$, so that for all $\wt{\eta}>0$
	\beq\label{eq:v01_G1}\begin{aligned}
		\absv{\bsw\adj\bsv}^{2}=\absv{\bse_{1}\adj \wt{U}\adj\bsv}^{2}\leq \bse_{1}\adj \frac{\wt{\eta}^{2}}{\wt{\eta}^{2}+\wt{U}\adj(Y+\ii B)\adj J\adj J(Y+\ii B)\wt{U}}\bse_{1}
		=\wt{\eta} \im[\wt{G}^{(1)}(\ii\wt{\eta})]_{N+1,N+1}
	\end{aligned}\eeq	
	where in the last equality we used \eqref{eq:Schur_G22} and defined 
	\beq
		\wt{G}^{(1)}(\ii\wt{\eta})\deq(\wt{H}^{(1)}-\ii\wt{\eta})^{-1},\qquad \wt{H}^{(1)}\deq \begin{pmatrix}
			0 & J(Y+\ii B)\wt{U} \\
			\wt{U}\adj(Y+\ii B)\adj J\adj & 0
		\end{pmatrix}\in\C^{\bbrktt{2,2N}\times\bbrktt{2,2N}}.
	\eeq
	On the other hand, we use Schur complement formula to get
	\begin{align}
		-\frac{1}{[\wt{G}^{(1)}(\ii\wt{\eta})]_{N+1,N+1}}=\ii\wt{\eta}+\sum_{i,j\in\bbrktt{2,N}\cup\bbrktt{N+2,2N}}\wt{H}^{(1)}_{N+1,i}[(\wt{H}^{(1,1)}-\ii\wt{\eta})^{-1}]_{ij}\wt{H}^{(1)}_{j,N+1},
	\end{align}
	where $\wt{H}^{(1,1)}$ is the matrix obtain from $\wt{H}^{(1)}$ by deleting the $(N+1)$-th rows and columns. Equivalently, $\wt{H}^{(1,1)}$ is the Hermitization of $J(Y+\ii B)\wt{U}J\adj\in\C^{\bbrktt{2,N}^{2}}$, that is,
	\beq
	\wt{H}^{(1,1)}\deq\begin{pmatrix}
		0 & J(Y+\ii B)\wt{U}J\adj \\ J\wt{U}\adj(Y+\ii B)\adj J\adj & 0
	\end{pmatrix}\in\C^{(\bbrktt{2,N}\cup\bbrktt{N+2,2N})^{2}}.
	\eeq
	Thus, by the fact that $(\wt{G}^{(1)}(\ii\wt{\eta}))_{N+1,N+1}\in\ii\R$, the definition of $\wt{H}^{(1)}$, and another application of Schur complement formula, we have
	\begin{align}
		\frac{1}{\wt{\eta}\im[\wt{G}^{(1)}(\ii\wt{\eta})]_{N+1,N+1}}=&1+\frac{1}{\wt{\eta}}\bse_{1}\adj\wt{U}\adj(Y+\ii B)\adj J\adj\left(\frac{\wt{\eta}}{J(Y+\ii B)\wt{U} J\adj J\wt{U}\adj(Y+\ii B)\adj J\adj+\wt{\eta}^{2}}\right)J(Y+\ii B)\wt{U}\bse_{1}	\nonumber\\
		\geq&1+\frac{\norm{J(Y+\ii B)\bsw}^{2}}{\norm{J(Y+\ii B)\wt{U}J\adj}^{2}+\wt{\eta}^{2}}
		\geq 1+\frac{\norm{JY\bsw}^{2}}{\norm{J(Y+\ii B)}^{2}+\wt{\eta}^{2}}, \label{eq:G1_G11}
	\end{align}
	where in the second line we used $\wt{U}\bse_{1}=\bsw$, $JB\bsw=0$, and the positive semi-definiteness of
	\beq
		\frac{1}{ZZ\adj+1}-\frac{1}{\norm{Z}^{2}+1}\geq 0
	\eeq
	which is true for any matrix $Z$. We then combine \eqref{eq:v01_G1} and \eqref{eq:G1_G11} to get
	\beq\label{eq:v01_G11}
	1-\absv{\bse_{1}\adj\bsv}^{2}\geq 1-\wt{\eta}\im[\wt{G}^{(1)}(\ii\wt{\eta})]_{N+1,N+1}\geq 1-\left(1+\frac{\norm{JY\bsw}^{2}}{\norm{J(Y+\ii B)}^{2}+\wt{\eta}^{2}}\right)^{-1}.
	\eeq
	Finally, since $\wt{\eta}$ can be arbitrary, we take the limit $\wt{\eta}\searrow 0$ in \eqref{eq:v01_G11} to obtain
	\beq
	1-\absv{\bse_{1}\adj\bsv}^{2}\geq\frac{\norm{JY\bsw}^{2}}{\norm{JY\bsw}^{2}+\norm{J(Y+\ii B)}^{2}}\geq \frac{1}{5}\frac{\norm{JY\bsw}^{2}}{(\norm{JY}+\norm{B})^{2}}.
	\eeq
	This completes the proof of Lemma \ref{lem:real_1}.
\end{proof}

\begin{rem}\label{rem:real_1}
	As pointed out in Remark \ref{rem:s_R_C}, the suboptimality in \eqref{eq:s_real} is due to that of Lemma \ref{lem:real_1}. In particular the inequality \eqref{eq:real_1_crude} is far from being optimal: For Gaussian $X$ and $A=-z$, so that $Y=X-\re z$ and $B=-\im z$, numerical experiments show that $\re[\bsv]$ and $\im[\bsv]$ have almost equal size when $\absv{\im z}\gtrsim N^{-1/2}$. In this case, we believe that the typical size of the right-hand side of \eqref{eq:real_1} is $\left(1\wedge N\absv{\im z}^{2}\right)$ up to a positive random variable of size $O(1)$. Nonetheless, we do not know whether the first negative moment of this random variable is finite which is crucial for our proof; see \eqref{eq:s_real_weak_prf}
\end{rem}

\appendix
\section{Generalized domain for Theorem \ref{thm:overlap}}\label{append:Minkowski}
The goal of this section is to prove that we may replace the square $\caD$ in Theorem \ref{thm:overlap} by a general Borel set $\caD_{0}$, instead of a square, under the condition that
\beq\label{eq:ovlp_general}
	\absv{\caD_{0}+[-N^{-K},N^{-K}]^{2}}\leq C\absv{\caD_{0}}
\eeq
for some constants $C,K>0$, where we recall that $\absv{\caD}$ is the Lebesgue measure of any planar domain $\absv{\caD}$.

For simplicity, we write $x\deq N^{-K}$ and $\caS\deq[-x,x]^{2}$. First, note that it suffices to prove that Theorem~\ref{thm:overlap} is true with the choice $\caD=\caA+\caS$ for any Borel set $\caA\subset\C$. Suppose this has been done, then for a given $\caD_{0}$ satisfying \eqref{eq:ovlp_general}, we define $\Xi_{\caD_{0}}\deq\Xi_{\caD_{0}+\caS}$ and write
\beqs\begin{aligned}
	\E\lone_{\Xi_{\caD_{0}}}\sum_{i:\sigma_{i}\in\caD_{0}}\caO_{ii}
	\leq	 \E\lone_{\Xi_{\caD_{0}+\caS}}\sum_{i:\sigma_{i}\in\caD_{0}+\caS}\caO_{ii}	
	\lesssim N^{1+\xi}(N\absv{\caD_{0}+\caS})\lesssim N^{1+\xi}(N\absv{\caD_{0}}).
\end{aligned}\eeqs
This proves Theorem \ref{thm:overlap} for $\caD=\caD_{0}$.

In order to prove the result for $\caD=\caA+\caS$, we prove that this Minkowski sum can be covered by non-intersecting copies of $\caS$, so that the total area is comparable to $\absv{\caA+\caS}$.
	\begin{lem}\label{lem:lattice_minkowski}
		For any Borel measurable $\caA\subset\C$,
		the discrete set $\caL\deq (2x\Z)^{2}\cap(\caA+2\caS)$ satisfies 
		\beq\label{eq:lattice_minkowski}
			\caA+S\subset \caL+\caS \AND \absv{\caL+\caS}\leq 3\absv{\caA+\caS}.
		\eeq
	\end{lem}
	\begin{proof}
		Define $\absv{z}_{\infty}=\absv{\re z}\vee\absv{\im z}$ for $z\in\C$. To prove the first assertion, take a point $z\in\caA+\caS$. Then we have a point $w\in(2x\Z)^{2}$ and $z_{0}\in\caA$ such that $\absv{z-w}_{\infty}\leq x$ and $\absv{z_{0}-z}_{\infty}\leq x$. Thus we immediately have $\absv{w-z_{0}}_{\infty}\leq 2x$ so that $w\in(2x\Z)^{2}\cap(z_{0}+2\caS)\subset\caL$. Thus $z\in w+\caS\subset\caL+\caS$.
		
		The second assertion follows from $\caL\subset\caA+2\caS$ since
		\beqs
		\absv{\caL+\caS}\leq \absv{\caA+2\caS+\caS}=\absv{\caA+3\caS}\leq 3\absv{\caA+\caS}.
		\eeqs
		This finishes the proof.
	\end{proof}

	Next, we prove that Theorem \ref{thm:overlap} is true for $\caD=\caA+\caS$ with any Borel set $\caA$. Let $\caL_{0}\deq\caL\cap[-2,2]^{2}$ where $\caL$ is the discrete set given by Lemma \ref{lem:lattice_minkowski}. Apply Theorem \ref{thm:overlap} for each square $z+\caS$ with $z\in\caL_{0}$, so that $\P[\Xi_{z+\caS}^{c}]\leq N^{-2K-D-1}$ and \eqref{eq:overlap_strong} holds true with the choice $\caD=z+\caS$.
	
	Define $\Xi_{\caA+\caS}\deq\bigcap_{z\in\caL_{0}}\Xi_{z+\caS}\cap\Xi_{0}$ where $\Xi_{0}$ is the event $[\norm{X}\leq 3]$. Then we have $\P[\Xi_{\caA+\caS}^{c}]\leq N^{-D}$ from $\absv{\caL_{0}}=O(N^{2K})$ and $\P[\Xi_{0}^{c}]\leq N^{-D-1}$. Note that on the event $\Xi_{0}$ there is no eigenvalue in the rightmost side, hence all, of
	\beqs
		(\caA+\caS)\setminus(\caL_{0}+\caS)\subset(\caL\setminus\caL_{0})+\caS=(\caL\cap([-3,3]^{2})^{c})+\caS,
	\eeqs
	where we used the first inequality of \eqref{eq:lattice_minkowski}. Then it follows that
	\beqs\begin{aligned}
	\E\lone_{\Xi_{\caA+\caS}}\sum_{i:\sigma_{i}\in\caA+\caS}\caO_{ii}
	\leq&\sum_{z\in\caL_{0}} \E\lone_{\Xi_{z+\caS}}\sum_{i:\sigma_{i}\in(z+\caS)}\caO_{ii}	
	\lesssim
	N^{1+\xi}(N\sum_{z\in\caL_{0}}\absv{z+\caS})
	=N^{1+\xi}(N\absv{\caL_{0}+\caS_{}})\leq 3N^{1+\xi}(N\absv{\caD}),
	\end{aligned}\eeqs
	where in the last step we used the second inequality of \eqref{lem:lattice_minkowski}. This proves that we may take $\caD=\caA+\caS$, and hence $\caD_{0}$ satisfying \eqref{eq:ovlp_general}, in Theorem \ref{thm:overlap}.

\end{document}